\documentclass[oneside,reqno,12pt]{amsart}

\usepackage{amsmath,amssymb,amsthm,bbm,amsxtra}
\usepackage[T1]{fontenc} 
\usepackage[ansinew]{inputenc} 

\usepackage{color}
\usepackage{verbatim}
\usepackage{enumitem}
\usepackage{csquotes}
\usepackage[english]{babel}

\usepackage[nonewpage]{imakeidx}
\makeindex[columns=1,options=-s indstyle.ist]

\usepackage[linktocpage]{hyperref}

\usepackage[backend=biber,style=alphabetic]{biblatex}
\setlist[enumerate]{leftmargin=1cm}

\makeatletter
\def\@tocline#1#2#3#4#5#6#7{\relax
 \ifnum #1>\c@tocdepth 
 \else
  \par \addpenalty\@secpenalty\addvspace{#2}%
  \begingroup \hyphenpenalty\@M
  \@ifempty{#4}{%
   \@tempdima\csname r@tocindent\number#1\endcsname\relax
  }{%
   \@tempdima#4\relax
  }%
  \parindent\z@ \leftskip#3\relax \advance\leftskip\@tempdima\relax
  \rightskip\@pnumwidth plus4em \parfillskip-\@pnumwidth
  #5\leavevmode\hskip-\@tempdima
   \ifcase #1
    \or\or \hskip 1em \or \hskip 2em \else \hskip 3em \fi%
   #6\nobreak\relax
  \dotfill\hbox to\@pnumwidth{\@tocpagenum{#7}}\par
  \nobreak
  \endgroup
 \fi}
\makeatother

\makeatletter
\def\@seccntformat#1{%
  \protect\textup{\protect\@secnumfont
    \ifnum\pdfstrcmp{subsection}{#1}=0 \bfseries\fi
    \csname the#1\endcsname
    \protect\@secnumpunct
  }%
}  
\makeatother

\theoremstyle{plain} 
\newtheorem{theorem}{Theorem} 
\newtheorem{lemma}[theorem]{Lemma}
\newtheorem{proposition}[theorem]{Proposition}
\newtheorem{corollary}[theorem]{Corollary}
\theoremstyle{definition}

\newtheorem{definition}[theorem]{Definition}
\newtheorem{example}[theorem]{Example}

\theoremstyle{remark}
\newtheorem{remark}[theorem]{Remark}

\numberwithin{equation}{section}
\numberwithin{theorem}{section}

\newcommand{\N}{\mathbb{N}}  
\newcommand{\Z}{\mathbb{Z}}  
\newcommand{\R}{\mathbb{R}}  
\newcommand{\C}{\mathbb{C}}  

\newcommand{\SL}{\mathrm{SL}}
\newcommand{\GL}{\mathrm{GL}}

\newcommand{\eps}{\varepsilon} 
\newcommand{\iu}{\mathrm{i}}   
\newcommand{\e}{\mathrm{e}}    

\newcommand{\dif}{\mathrm{d}}

\newcommand{\imbound}{q_0^{-1/2}r^{-1}} 

\newcommand{\defi}{\; \stackrel{\mathrm{def}}{=} \;}
\newcommand{\tdefi}{:=}

\newcommand{\bgr}[1] {\4\bigr#1}
\newcommand{\bgl}[1]{\bigl#1\4}

\newcommand{\na}{\,\, {\raise.4pt\hbox{$\shortmid$}}{\hskip-2.0pt\to}\, \, }

\newcommand{\wh}{\widehat}
\newcommand{\q}{\quad}

\newcommand{\vol}{\mathrm{vol}_\mathbb{Z}\,}
\renewcommand{\v}{\mathrm{vol}\,}
\newcommand{\norm}[1]{\lVert#1\rVert}

\newcommand{\4}{\kern1pt}

\DeclareMathOperator{\volu}{\operatorname{vol}}

\DeclareMathOperator{\Nm}{\operatorname{Nm}}
\DeclareMathOperator{\Tr}{\operatorname{Tr}}

\newcommand{\ffrac}[2]{\raise.5pt\hbox{\small$\4\displaystyle\frac{\,#1\,}{\,#2\,}\4$}}

\newcommand{\abs}[1]{\lvert#1\rvert}

\newcommand{\largewedge}{\mbox{\large $\wedge$}}

\DeclareSymbolFont{spfont}{U}{cmr}{m}{n}
\DeclareMathSymbol{\specialv}{\mathord}{spfont}{118}

\begin{document}

\title[Distribution of Values of Quadratic Forms]{Distribution of Values of Quadratic Forms at Integral Points ${}^*$}
\author{P.~Buterus, F.~G\"{o}tze, T.~Hille, G.~Margulis}
\thanks{${}^*$ Research supported by the DFG, CRC 701}


\date{\today}



\begin{abstract}
  The number of lattice points in $d$-dimensional hyperbolic or elliptic shells $\{m : a<Q[m]<b\}$, which are restricted to rescaled and growing domains $r\4\Omega$, is approximated by the volume. An effective error bound of order $o(r^{d-2})$ for this approximation is proved based on Diophantine approximation properties of the quadratic form $Q$. These results allow to show effective variants of previous non-effective results in the quantitative Oppenheim problem and extend known effective results in dimension $d \geq 9$ to dimension $d \geq 5$. They apply to wide shells when $b-a$ is growing with $r$ and to positive definite forms $Q$. For indefinite forms they provide explicit bounds (depending on the signature or Diophantine properties of $Q$) for the size of non-zero integral points $m$ in dimension $d\geq 5$ solving the Diophantine inequality $\lvert Q[m] \rvert < \varepsilon$ and provide error bounds comparable with those for positive forms up to powers of $\log r$.
\end{abstract}

\maketitle

\vspace{-1.5em}

\tableofcontents

\vspace{-2em}



\section{Introduction}
\label{section:introduction}
\noindent Let $Q[x]$ denote an indefinite quadratic form in $d$ variables. We say that the form $Q$ is \textit{rational}, if it is proportional to a form with integer coefficients; otherwise it is called \textit{irrational}. The Oppenheim conjecture, proved by G.~Margulis \cite{margulis:1987} in 1986, states that $Q[\Z^d]$ is dense in $\R$ if $d \geq 3$ and $Q$ 
is irrational. Initially this was conjectured for $d \ge 5$ by A.~Oppenheim \cite{oppenheim:1929,oppenheim:1931} in 1929 and in 1946 strengthened (for diagonal forms) to $d\ge 3$ by H.~Davenport \cite{davenport-heilbronn:1946}. The proof given in 1986 uses a connection, noticed by M.~S.~Raghunathan, between the Oppenheim conjecture and questions concerning closures in $\SL (3, \R) / \SL (3, \Z)$ of orbits of certain subgroups of $\SL (3, \R)$. It is based on the study of minimal invariant sets and the limits of orbits of sequences of points tending to a minimal invariant set. Previous studies have mostly used analytic number theory methods. In fact, 
B.~J. Birch, H.~Davenport and D.~Ridout proved in a series of papers that 
$Q[\Z^d]$ is dense in $\R$ if $d \geq 21$ provided that $Q$ is irrational, see \cite{lewis:1973} and \cite{margulis:1997} for a complete historical overview until 1997.

For a measurable set $B \subset \R^d$ let $\v \, B$\index{V@ $\v \, (B)$, 
Lebesgue measure on $\R^d$} denote the Lebesgue measure of $B$ and let\index{V@ $\v_\Z(B)= \# (B \cap \Z^d)$, counting measure on $\Z^d$} $\v_\Z 
\, B \4 \tdefi \4 \#(B \cap \Z^d)$ denote the number of integer points in 
$B$. We define for $a, b \in \R$ with $a < b$ the \textit{hyperbolic shell}\index{E@ $E_{a,b} \tdefi \{ x \in \R^d \, \colon \, a < Q [x] < b \}$, hyperbolic or ellipsoidal shell}
\begin{equation*}
  E_{a, b} \defi \{ x \in \R^d \, \colon \, a < Q [x] < b \}.
\end{equation*}
The Oppenheim conjecture is equivalent to the statement that if $d \geq 3$ and $Q$ is irrational, then $\vol E_{a,b} = \infty$ whenever $a < b$. 
We would like to study the distribution of values of $Q$ at integer points, often referred to as ``quantitative Oppenheim conjecture'' with an emphasis on establishing effective error bounds for the approximation of the 
number of lattice points restricted to growing domains. Our methods rely mainly on G\"{o}tze's Fourier approach \cite{goetze:2004} via Theta series, translating the lattice point counting problem into averages of certain functions on the space of lattices, for which we extend the mean-value estimates obtained by Eskin-Margulis-Mozes \cite{eskin-margulis-mozes:1998}.

\subsection{Related Results}
Let $\mathcal{R}$ be a continuous positive function on the sphere $\{v \in \R^d \,\colon \, \norm{v}=1\}$ and let $\Omega= \{v\in \R^d\,:\, \norm{v} \le 1/ \mathcal{R}(v/\norm{v})\}$. Note that the Minkowski functional of $\Omega$, that is $M(v)=\inf\{r>0 \, \colon \, v \in r \Omega\}$, may be rewritten as $M(v)=\norm{v} \4 \mathcal{R}(v/\norm{v})$ and therefore $\Omega =\{v\in \R^d \,:\, M(v) \le 1\}$. Without loss of generality we may assume that $\Omega\subset [-1,1]^d$. We denote by $r\Omega$ 
the dilate of $\Omega$ by $r > 1$. In \cite{dani-margulis:1993} S.~G.~Dani and G.~Margulis obtained the following asymptotic exact lower bound under the same assumptions that $Q$ is irrational and $d\ge 3$:
\begin{equation}
  \label{eqa}
  \liminf_{r\to\infty} \frac{\vol (E_{a,b}\cap r\Omega)}{\v (E_{a,b} \cap 
r\Omega)}\ge 1.
\end{equation}

\begin{remark}
  \label{remark:volume}
  It is not difficult to prove (see Lemma 3.8 in \cite{eskin-margulis-mozes:1998}) that as $r\to \infty$,
  \begin{equation*}
    \v(E_{a,b}\cap r\Omega)\sim \lambda_{Q, \Omega}(b-a)r^{d-2},
  \end{equation*}
  where
  \begin{equation}
    \label{eq1}
    \lambda_{Q,\Omega}\defi \int_{L\cap\Omega}\frac{\dif A}{\norm{\nabla Q}},
  \end{equation}
  $L$ is the light cone $Q=0$ and $\dif A$ is the area element on $L$.
\end{remark}

The situation with asymptotics and upper bounds is more subtle. It was proved in \cite{eskin-margulis-mozes:1998} that if $Q$ is an irrational indefinite quadratic form of signature $(p,q)$ with $p+q=d$, $p \geq 3$ and $q \geq 1$, then for any $a<b$
\begin{equation}
  \label{eqb}
  \lim_{r\to\infty}\frac{\vol (E_{a,b}\cap r\Omega)}{\v (E_{a,b}\cap r\Omega)}= 1
\end{equation}
or, equivalently, as $r\to \infty$
\begin{equation}
  \label{eq2}
  \vol(E_{a,b}\cap r\Omega)\sim \lambda_{Q, \Omega}(b-a)r^{d-2},
\end{equation}
where $\lambda_{Q, \Omega}$ is as in \eqref{eq1}.

If the signature of $Q$ is $(2,1)$ or $(2,2)$, then no universal formula like \eqref{eq2} holds. In fact, one can show (see Theorem 2.2 in \cite{eskin-margulis-mozes:1998}) that if $\Omega$ is the unit ball and $q=1$ or $q=2$, then for every $\eps >0$ and every $a<b$ there exists an irrational quadratic form $Q$ of signature $(2,q)$ and a constant $c>0$ such that for an infinite sequence $r_j\to \infty$
\begin{equation*}
  \vol(E_{a,b}\cap r_j\Omega)>cr_j^{d-2}(\log r_j)^{1-\eps}.
\end{equation*}
While the asymptotics as in \eqref{eq2} do not hold in the case of signatures $(2,1)$ and $(2,2)$, one can show (see \cite{eskin-margulis-mozes:1998}) that in these cases there is an upper bound of the form $r^{d-2}\log r$. This upper bound is effective and it is uniform over compact sets in the space of quadratic forms. In addition, there is  an 
effective uniform upper bound (see \cite{eskin-margulis-mozes:1998}) of the form $cr^{d-2}$ for the case $p\ge 3$, $q\ge 1$.

\noindent The examples in \cite{eskin-margulis-mozes:1998} for the cases of signatures $(2,1)$ and $(2,2)$ are obtained by considering irrational forms which are very well approximated by split rational forms. More precisely, a quadratic form $Q$ is called \textit{extremely well approximable 
by split rational forms} (EWAS) if for any $N>0$ there exists a split integral form $Q^\prime$ and a real number $t \geq 2$ such that
\begin{equation*}
  \norm{tQ-Q^\prime} \le t^{-N},
\end{equation*}
where $||\cdot||$ denotes a norm on the linear space of quadratic forms. It is shown in \cite{eskin-margulis-mozes:2005} that if $Q$ is an indefinite quadratic form of signature $(2,2)$, which is not (EWAS), then for any interval $(a,b)$, as $r\to \infty$,
\begin{equation}
  \label{eqc}
  \widetilde{N}_{Q,\Omega}(a,b,r)\sim \lambda_{Q, \Omega}(b-a)r^2,
\end{equation}
where $\lambda_{Q, \Omega}$ is the same as in \eqref{eq1} and $\widetilde{N}_{Q,\Omega}(a,b,r)$ counts all the integral points in $E_{a,b}\cap r\Omega$ not contained in rational subspaces isotropic with respect to $Q$. It should be noted that
\begin{enumerate}[label=(\roman*)]
  \item an irrational quadratic form of signature $(2,2)$ may have at most four rational isotropic subspaces,
  \item if $0\not\in (a,b)$, then $\widetilde{N}_{Q,\Omega}(a,b,r)=\vol 
(E_{a,b}\cap r\Omega).$
\end{enumerate}
The above mentioned results have analogs for inhomogeneous quadratic forms
\begin{equation*}
  Q_\xi[x]= Q[x+\xi], \quad \xi \in \R^d.
\end{equation*}
We define for $a,b\in \R$ with $a<b$ the shifted hyperbolic shell
\begin{equation*}
  E_{a,b,\xi}\defi \{x\in \R^d:\, a<Q_\xi[x]<b\}.
\end{equation*}
We say that $Q_\xi$ is \textit{rational} if there exists $t>0$ such that the coefficients of $tQ$ and the coordinates of $t\xi$ are integers; otherwise $Q_\xi$ is \textit{irrational}. Then, under the assumptions that $Q_\xi$ is irrational and $d\ge 3$, we have that (see \cite{margulis-mohammadi:2011})
\begin{equation}
  \label{eqa'}
  \liminf_{r\to\infty} \frac{\vol (E_{a,b, \xi}\cap r\Omega)}{\v (E_{a,b, 
\xi}\cap r\omega)}\ge 1.
\end{equation}
The proof of \eqref{eqa'} is similar to the proof of \eqref{eqa}.\par Let 
$(p,q)$ be the signature of $Q$. If $p \geq 3$, $q \geq 1$ and $Q_\xi$ is 
irrational then
\begin{equation}
  \label{eqb'}
  \lim_{r\to\infty}\frac{\vol (E_{a,b, \xi}\cap r\Omega)}{\v (E_{a,b, \xi}\cap r\Omega)}= 1,
\end{equation}
or, equivalently, as $r\to \infty$,
\begin{equation}
  \label{eq2'}
  \vol(E_{a,b,\xi} \cap r\Omega) \sim \lambda_{Q, \Omega}(b-a)r^{d-2}.
\end{equation}
The proof of \eqref{eqb'} is similar to the proof of \eqref{eqb}, see \cite{margulis-mohammadi:2011}. The latter paper \cite{margulis-mohammadi:2011} also contains an analog of \eqref{eqc} for inhomogeneous forms in the 
case of signature $(2,2)$. One should also mention related results of Marklof \cite{marklof:2002,marklof:2003}.

\begin{remark}
  The proofs of the above mentioned results use such notions as a minimal 
invariant set (in the case of the Oppenheim conjecture) and an ergodic invariant measure. These notions do not have in general effective analogs. Because of that it is very difficult to get `good' estimates for the size 
of the smallest non-trivial integral solution of the inequality $\abs{Q[m]}<\eps$ and `good' error terms in the quantitative Oppenheim conjecture by applying dynamical and ergodic methods.
\end{remark}

\subsection{Diophantine Inequalities}
\label{subsection:dio_ineq}
One of our main objective is to develop effective analogs of \eqref{eq2'} 
and show that \textit{all indefinite quadratic forms $Q$ of rank at least 
$5$} admit a non-trivial integral solution to the Diophantine inequality $\abs{Q[m]}<\eps$ whose size can be bounded effectively in terms of $\eps^{-1}$. On the one hand, we will exploit Schlickewei's results \cite{Schlickewei:1985} on small zeros of integral forms (see Subsection \ref{section:integer_val:qforms}) in order to establish effective bounds depending on the signature $(r,s)$ of $Q$. On the other hand we will introduce an appropriate Diophantine condition on the space of quadratic forms, which will enable us to significantly improve our effective bounds due to the exponents appearing in the Diophantine approximation of $Q$. To state these 
bounds we need to introduce notation. \par
Denote by $Q$\index{Q@ $Q$, as quadratic form and the corresponding symmetric matrix} also the symmetric matrix in $\GL(d, \R)$ associated with the form $Q[x] \tdefi \langle x, Q\, x\rangle$, where $\langle\, \cdot\,, \,\cdot \, \rangle $\index{Q@ $Q[x]=\langle Q x,x \rangle$, Siegel's notation} is the standard Euclidean scalar product on $\R^d$\index{1@ $\langle \, \cdot \,, \, \cdot \, \rangle$, $\norm{\, \cdot  \,}$, Euclidean inner product and associated norm}. Let $Q_{+}$\index{Q@ $Q_+$, positive definite square root of $Q^2$} denote the unique positive symmetric matrix 
such that $Q_{+}^2=Q^2$ and let $Q_{+}[x]=\langle x, Q_{+}\, x\rangle$ denote the associated positive form with eigenvalues being the eigenvalues of $Q$ in absolute value. Let $q$, resp.\ $q_0$,\index{Q@ $q$, largest eigenvalue of $Q$ in absolute value}\index{Q@ $q_0$, smallest eigenvalue of $Q$ in absolute value, $q_0 \geq 1$} denote the largest, resp.\ smallest, of the absolute value of the eigenvalues of $Q$ and assume 
$q_0 \ge 1$. In the first case, where we compare $Q$ with rational forms, 
we can replace the form $Q$ by $Q / \eps$ and consider the solubility of the inequality $\abs{Q[m]} <1$. Since this Diophantine inequality includes the case of integral-valued indefinite forms, we shall appeal to Corollary \ref{bound:schlickewei:corollary} (a variant of \textit{Folgerung 3} in \cite{Schlickewei:1985}) on the size of non-trivial integral solutions.

\begin{theorem}
  \label{dio-ineq}
  For all indefinite and non-degenerate quadratic forms $Q$ of dimension $d \ge 5$ and signature $(r,s)$ there exists for any $\delta >0$ a non-trivial integral solution $m \in \Z^d \setminus \{0\}$ to the Diophantine inequality $\abs{Q[m]} < 1$ satisfying
  \begin{equation}
    \label{size_bound}
    \norm{Q_{+}^{1/2} m} \ll_{\delta,d} (q/q_0)^{\frac{d+1}{d-2}} q^{\frac{1}{2} + \max\{ \rho d +2,d+1\}/(d-4)  +\delta},
  \end{equation}
  where the dependency on the signature $(r, s)$\index{R@ $\rho=\rho(r,s)$, Schlickewei exponent} is given by
  \begin{equation}
    \label{dio-ineq-rho}
    \rho := \rho(r,s) := \begin{cases}
      \frac{1}{2} \frac{r}{s}     & \text{for} \ r \geq s+3 \\
      \frac{1}{2} \frac{s+2}{s-1} & \text{for} \ r=s+2 \ \text{or} \ r=s+1 \\
      \frac{1}{2} \frac{s+1}{s-2} & \text{for} \ r=s.
    \end{cases}
  \end{equation}
\end{theorem}
In particular, for indefinite non-degenerate forms in $d\ge 5$ variables of signature $(r,s)$ and eigenvalues in absolute value contained in a compact set $[1, C]$, i.e $1 \le q_0 \le q \le C$,  Theorem \ref{dio-ineq} yields non-trivial solutions $m \in \Z^d$ of $\abs{Q[m]} <\eps$ of size bounded by
\begin{equation*}
  \norm{m} \ll_{C,\delta}  \eps^{- \max\{ \rho d +2,d+1\}/(d-4) - \delta}.
\end{equation*}
As an example, we obtain solutions of order $\ll_{C,\delta} \eps^{-1- \frac{5}{(d-4)}-\delta}$ for the special case $r= s+3$ and $d \geq 12$. More generally, we may embed $\Z^{d_1} \subset \Z^d$ for dimensions $d \geq 
d_1 \geq 5$, in such a way that the restricted form is indefinite and of rank $d_1$, and apply Theorem \ref{dio-ineq} to this form in $d_1$ dimensions. As a consequence, since $(Q^*)^2 \le Q^2$ in the ordering of positive forms we get $q\ge q^{*} \geq q_0^* \geq q_0 \geq 1$ and $\abs{\det Q^*} \le \abs{\det Q}$, we obtain the following corollary.

\begin{corollary}
   For all indefinite and non-degenerate quadratic forms $Q$ in $d \geq 5$ variables there exists for any $\varepsilon >0$ at least one non-trivial integral solution $m \in \Z^d$ of
   \begin{equation}
      \begin{aligned}
        \abs{Q[m]} &< \varepsilon, \\
        \norm{m} &\leq c_{C,\delta}\4 \varepsilon^{-f_d -\delta},
      \end{aligned}
   \end{equation}
   for any $\delta>0$, where $f_d=12,\4 8\frac12,\4 7\frac23$ for $d=5,\4 6,\4 7$ respectively and $f_d=7\frac12$ for all $d\ge 8$. The constant $c_{C,\delta}$ depends only on $\delta$ and $C>0$ for forms $Q$ satisfying $1 \leq q_0 \leq q \leq C$.
\end{corollary} 

\begin{remark}
  (a) For the special case of \textit{diagonal} indefinite forms $Q[x]=\sum_{j=1}^5 q_j x_j^2$ with $\min \abs{q_j} \ge 1$ Birch and Davenport 
(1958), \cite{birch-davenport:1958}, obtained a sharper bound. They showed for arbitrary small $\delta>0$ that there exists an $m\in \Z^5 \setminus \hspace{-1pt} \{0\}$ with $\abs{Q[m]}<1$ and $Q_{+}[m]\ll_{d,\delta} \abs{\det{Q}}^{1+\delta}$. This implies (as above) for a compact set of forms $Q$ that there exists an integral vector $m$ satisfying $\abs{Q[m]}< \eps$ and $\norm{m} \le c_{d,\delta}\, \eps^{-2+\delta}$ for any fixed $\delta >0$. In \cite{buterus-goetze-hille:2019} Buterus, G\"{o}tze and Hille extended the approach of Birch and Davenport to improve the size of a solution by using Schlickewei's result \cite{Schlickewei:1985} on small zeros of integral forms: Let $Q[x] = \sum_{j=1}^d q_j x_j^2$ be an indefinite form of signature $(r,s)$ in $d = r+s \ge 5$ variables. Then for 
any $\eps >0$ the Diophantine inequality $\abs{Q[m]} < \eps$ admits a non-trivial solution $m \in \Z^d$, whose size is bounded by $\ll \eps^{-\rho+\delta}$ for any fixed $\delta>0$.\\[2mm]
  (b) Recently, quantitative versions of the Oppenheim conjecture were studied by Bourgain \cite{Bourgain:2016}, Athreya and Margulis \cite{Athreya-Margulis:2018}, and Ghosh and Kelmer \cite{Ghosh-Kelmer:2018}. Bourgain 
\cite{Bourgain:2016} proves essentially optimal results for one-parameter 
families of diagonal ternary indefinite quadratic forms under the Lindel\"of hypothesis by using also a Fourier approach, based on Epstein-Zeta functions. In contrast, Ghosh and Kelmer \cite{Ghosh-Kelmer:2018} consider the space of all indefinite ternary quadratic forms and use spectral methods (an effective mean ergodic theorem). Lastly, Athreya and Margulis apply classical bounds of Rogers for $L^2$-norm of Siegel transforms in order to prove that for every $\delta >0$ and almost every $Q$ (with respect to the 
Lebesgue measure) with signature $(r,s)$, there exists a non-trivial integral solution $m \in \Z^d$ to the Diophantine inequality $\lvert Q[m] \rvert < \eps$ whose size is bounded by $\norm{m} \ll_{\delta,Q} \eps^{-\frac{1}{d-2}-\delta}$ if $d \geq 3$.
\end{remark}

As mentioned above let us introduce a class of Diophantine forms as follows.

\begin{definition}
  \label{def:dio_type}
  We call $Q$ Diophantine of type $(\kappa,A)$\index{D@ Diophantine quadratic form of type $(\kappa,A)$}, where $\kappa, A >0$, if for any $m\in \Z \setminus \{0\}$ and $M \in M(d,\Z)$ we have
  \begin{equation}
    \label{diophant}
    \inf_{t \in [1,2]}\norm{M- m\4t\4 Q} \ge A\,\abs{m}^{-\kappa},
  \end{equation}
  where $\norm{\,\cdot \,}$ denotes the operator norm induced by the Euclidean norm on $\R^n$.
\end{definition}

We shall see in Section \ref{subsection:diophantinity} that almost every form satisfies this property for some $\kappa$ and $A$. In particular, fixing an integer $k$ such that $1 \leq k \leq \frac{d(d+1)}{2}-1$, we shall show that a form $Q$ for which $k+1$ non-zero entries $y,x_1,\dots, x_k$ exist such that $x_1/y,\dots, x_k/y$ are algebraic and $1, x_1/y,\dots, x_k/y$ are linearly independent over $\mathbb{Q}$ is Diophantine in this sense and admits a non-trivial solution to the Diophantine inequality $\abs{Q[m]}<\epsilon$ of order $\ll_{Q,d, \delta} \epsilon^{-\frac{d(3+2k)-4}{2k(d-4)}-\delta}$ for any $\delta>0$. In 
particular, for $k = \frac{d(d+1)}{2}-1$ we can give a bound for the size of the least solution of order $\ll_{Q,d,\delta} \epsilon^{-\frac{d^3+d^2+d-4}{(d^2+d-2)(d-4)}-\delta}$ and in this case for $d=5$ of order $\ll_{Q,\delta} \epsilon^{-151/28-\delta}$.

\begin{corollary}
  \label{th:diophante:smallzeros}
  Let $Q$ be an indefinite quadratic form in $d \geq 5$ variables and of Diophantine type ($\kappa$,$A$) and fix $\delta >0$. Then for any $\eps >0$ there exists a non-trivial lattice point $m \in \Z^d \setminus 0$ satisfying
  \begin{equation*}
    \abs{Q[m]} < \eps \quad \text{and} \quad \norm{m} \ll_{Q,d,\delta} \eps^{- \frac{2d+3\kappa d-4\kappa}{2d-8}-\delta}.
  \end{equation*}
\end{corollary}

For \textit{irrational indefinite} quadratic forms we may quantify the density of values $Q[m]$, $m \in r\Omega \cap \Z^d$, where $\Omega$ denotes a (not necessarily admissible) parallelepiped satisfying \eqref{eq:sec7:Def:Omega} (see Subsection \ref{specialp}) as follows: Consider the set
\begin{equation*}
  V(r) \defi \bgl\{ Q[m] \4 :\4 m\in r \Omega \cap \Z^d \bgr\}\cap [-c_0\4 r^2, c_0\4 r^2]
\end{equation*}
of values of $Q[x]$, $x \in r \Omega \cap \Z^d$ lying in the interval $[-c_0\4 r^2, c_0\4 r^2]$, where $c_0$ denotes the constant introduced in Lemma \ref{l3}. For each $r \geq 1$ we arrange the values $V(r)$ in increasing order $v_0(r)< \ldots < v_k(r)$, $k=k(r)$, and define the maximal gap between successive values of $V(r)$ as
\begin{equation}
  \textstyle
  d(r) \tdefi \sup_{i \in \{1,\ldots,k(r)\}} \abs{v_{i}(r)-v_{i-1}(r)}.
\end{equation}
As a consequence of our technical quantitative bounds we obtain

\begin{corollary}
  \label{corgaps}
  Let $Q$ denote a non-degenerate indefinite form in $d \ge 5$ variables and of Diophantine type $(\kappa,A)$. For $\delta >0$ we obtain for the maximal gap $d(r)$ between successive values of the quadratic form in the set $V(r)$
  \begin{equation}
    d(r) \le r^{-\nu_0+\delta},
  \end{equation}
  for sufficiently large $r \geq c_{\delta,d,\Omega,\kappa,A,Q}$, where $\nu_0 \tdefi \frac{2d-8}{2d + 3 \kappa d - 4 \kappa}$ and $c_{\delta,d,\Omega,\kappa,A,Q}>0$ denotes a constant depending on $\kappa, A, Q, \Omega, d$ and $0<\delta <1/10$ (here we omit a description of the explicit dependence). 
\end{corollary}

For \textit{positive definite} quadratic forms Davenport and Lewis (see \cite{davenport-lewis:1972}) conjectured, that the distance between successive values $v_n$ of the quadratic form $Q[x]$ on $\Z^d$ converges to zero as $n\to \infty$, provided that the dimension $d$ is at least five and $Q$ is irrational. This conjecture was proved by G\"otze in \cite{goetze:2004}. It also follows by the results of the present paper which provides 
error bounds for the lattice point counting problem \emph{for the indefinite case as well as the positive definite case}.

The proof is similar as in the case of positive forms solved in \cite{goetze:2004}: For any $\eps>0$ and any interval $[b, b+\eps]$, we find at least two lattice points in the shell $E_{b, b+\eps}$ (and the box of size $r= \sqrt{2b}$) by Corollary \ref{positive}, provided that $b$ is larger than a threshold $b(\eps)$. Here $b(\eps)$ and consequently the distance between successive values (as a function of $b$) depends on the rate of convergence of the Diophantine characteristic $\rho_Q^{\mathrm{ell}}(r)$ in the bound of Corollary \ref{positive} towards zero. For quadratic forms of Diophantine type $(\kappa,A)$ this 
dependency can be stated explicitly.

\subsection{Discussion of Effective Bounds and Outline of the Proofs}

In order to prove an effective result like Theorem \ref{dio-ineq} we need 
an explicit bound for the error, say $R( I_{E_{a,b} \cap  r\Omega})$ (for a formal definition see \eqref{eq:df:formal_lattice_rem} below) with $I_B$ denoting the indicator of a set $B$\index{I@ $I_B$, indicator function of a set $B$}, of approximating the number of integral points $m \in E_{a,b}$ in a bounded domain $r\,\Omega$ by the volume $\v(E_{a,b}\cap r\Omega)$, compare Remark \ref{remark:volume}. First, we simplify the problem by replacing the weights $I_{r\Omega}(m)=1$ of integral points $m\,\in r\4\Omega$ by suitable smoothly changing weights $\specialv(m/r)$ (for notational simplicity, we will write $\specialv_r(m) := \specialv(m/r)$), which tend to zero as $m/r$ tends to infinity. This smoothing (together with a smoothing of the indicator function of $[a,b]$) allows us to use techniques from Fourier analysis, but we are forced to restrict the region $\Omega$ to parallelepipeds in order to ensure that the corresponding error has logarithmic growth only.

\subsubsection{Fourier analysis}
\label{subsec:intro_fourier}
Starting with smooth weight functions $\specialv_r$ (which depend on the dilation parameter $r$), we also construct a $w$-smooth\-ing $g$ of the indicator function of $[a,b]$ via convolution with an appropriate kernel $k$ whose Fourier transform  decays like $\abs{\wh k(t)} \ll \exp\{-\sqrt{\abs{wt}}\}$\index{K@ $k$, compactly supported kernel with sufficiently fast decaying Fourier transform}. This allows us to replace the indicator function of $[a,b]$ in the lattice point counting problem by a smooth function, gaining an error bounded in Corollary \ref{co1}. After this smoothing procedure, writing $g^Q(x) \tdefi g(Q[x])$, our main objective will be to estimate the weighted lattice remainder\index{R@ $R(I_{E_{a,b}} \4 \specialv_r), R(g \4 \specialv_r),R( I_{E_{a,b} \cap r\Omega})$, lattice point remainder}
\begin{equation}
  \label{eq:df:formal_lattice_rem}
  R(g^Q \4 \specialv_r) \defi \sum_{m\in \Z^d} g(Q[m]) \specialv (\tfrac{m}{r}) - \int_{\R^d} g(Q[x]) \specialv (\tfrac{m}{r}) \4 \dif x,
\end{equation}
where $g$ and $\specialv$ are smooth functions whose Fourier transforms decay fast enough as well. More precisely, we will assume that $\specialv$ satisfies \eqref{zeta-decay}. (At this point we should note that the abbreviation introduced in \eqref{eq:df:formal_lattice_rem} will frequently be used to denote remainder terms.) Next we shall use inverse Fourier transforms in order to express the weights as
\begin{equation*}
  g(Q[m]) = \int_\R \wh{g}(t) \exp\{ 2 \pi \iu \4 t\4 Q[m]\} \, \dif t,
  \q \zeta(m)=\int_{\R^d} \wh{\zeta}(u) \exp\{2 \pi \iu \4 \langle u,m\rangle\} \, \dif u,
\end{equation*}
where $\zeta(x) = \specialv(x) \exp\{Q_{+}[x]\}$. Combining the resulting factors $\exp\{ 2 \pi \iu \4 t \4 Q[m] \}$, \linebreak $\exp\{ 2 \pi \iu \langle v,m\rangle\}$ and $\exp\{-Q_{+}[\frac x r]\}$ in \eqref{eq:df:formal_lattice_rem} into terms of the generalized theta series\index{T@ $\theta_v(t)$, Theta series for $Q$}
\begin{equation*}
  \theta_v(t)\defi \sum_{m \in \Z^d} \exp\{ -2 \pi \iu \4 \langle v ,m\rangle /r - 2 \pi \iu \4t\4 Q[m] - Q_{+}[m]/r^2\}
\end{equation*}
one arrives at an expression for the sum $ V_r \tdefi \sum_{m\in \Z^{d}} \specialv(\tfrac{m}{r}) \4 g(Q[m])$ by the following integral (in $t$ and $v$) over $\theta_v(t)$:
\begin{equation}
 V_r = \int_{\R^d} \wh{\zeta}(v) \int_\R \wh{g}(t) \theta_v(t) \, \dif t \, \dif v.
\end{equation}
The approximating integral $W_r \tdefi \int_{\R^d} \specialv(\tfrac m r)\4 g(Q[x]) \, \dif x$ to this sum $V_r$ can be rewritten in exactly the same way by means of the theta integral\index{T@ $\vartheta_v(t)$, Theta integral for $Q$}
\begin{equation*}
  \vartheta_v(t) \defi \int_{\R^d} \exp\{ - 2 \pi \iu \4 \langle v ,x\rangle/r - 2 \pi\iu\4t\4 Q[x] - Q_{+}[x]/r^2\} \, \dif x,
\end{equation*}
replacing the theta sum $\theta_v(t)$. Thus, in order to estimate the error $\abs{ R( g^Q \4 \specialv_r )} =\abs{V_r-W_r}$, the integral over $t$ and $v$ of $\abs{\theta_v(t)- \vartheta_v(t)} \abs{\wh g(t) \wh \zeta(v)}$ has to be estimated.

For $\abs{t} \le \imbound$ and $\norm{x} \ll r$ the functions $ x\mapsto \exp\{ 2 \pi \iu\4t\4 Q[x]\}$ are sufficiently smooth, so that the sum $\theta_v(t)$ is well approximable by the first term of its Fourier series, 
that is the corresponding integral $\vartheta_v(t)$, see \eqref{poisson1} 
and \eqref{sum-bound1}. The error of this approximation, after integration over $v$, yields the second error term in \eqref{simplerror}, which does not depend on the Diophantine properties of $Q$. Additionally, we may restrict the integration to $\abs{t} \le T_{+}$  for an appropriate choice of $T_{+}$ (depending on the width of the shell) by using the decay rate of the kernel $k$. So we end up with the remaining error term
\begin{equation}
  I=\int_{T_{+} >\abs{t}>\imbound} \int_{\R^d} \abs{\theta_v(t) \4 \wh{g}(t)\4 \wh{\zeta}(v)} \, \dif v \, \dif t,
\end{equation}
which we estimate as follows
\begin{equation}
  \label{errfac}
  I \le \norm{\wh \zeta}_1 \4 \sup_{v \in \R^d}  \int_{T_{+}\ge \abs{t}>\imbound} \4 \abs{\theta_v(t)} \, \abs{\wh g(t)} \, \dif t.
\end{equation}
The second factor in the bound of $I$ in \eqref{errfac} encodes both the Diophantine behavior of $Q$ as described above as well as the growth rate 
with respect to $r$. We shall describe in the next subsection our method to extract out of this factor the correct rate of growth, while simultaneously avoiding the loss of information on the Diophantine properties of $Q$, provided that $d >4$. However, let us first state that the resulting bound (the choice of $T_+$ depending on the width of the shell) is an error bound depending on characteristics of $\wh{\zeta}
(v)$ of the form (see Theorem \ref{maintheorem})
\begin{equation}
  \label{roughbound}
  R(I_{E_{a,b}} \4 \specialv_r) \ll_{\kappa,d,Q,} w r^{d-2} + \norm{\wh \zeta}_{1} \rho_{Q,b-a}^w(r) r^{d-2} + \norm{\wh \zeta}_{1,*}\4 r^{d/2}\log \Big(1 +\frac {b-a}{r}\Big),
\end{equation}
which has to be optimized in the smoothing size $w$ (compare e.g.\ Corollary \ref{positive}) and $\rho_{Q,b-a}^w(r)$ depends on the Diophantine properties of $Q$ and $r$ (see Theorem \ref{maintheorem}).

\subsubsection{Mean-Value Estimates}
In order to describe the second term in \eqref{roughbound}, we follow \cite{goetze:2004} (by using a modified Weyl differencing argument) to show in Lemma \ref{thetaestimate2} that uniformly in $v$ and pointwise in $t$
\begin{equation}
\label{averagetheta}
  \abs{\theta_v(t)}^2 \ll r^d \, \abs{\det Q}^{-1/2} \sum_{ v \in \Lambda_t} \exp\{ -\norm{v}^2\},
\end{equation}
where $\{\Lambda_t\}_{t \in \R}$ is a family of $2d$-dimensional unimodular lattices generated by orbits of one-parameter subgroups of $\SL(2,\R)$ indexed by $t$ and $r$, see \eqref{df:lambda_t} for the precise definition. It is well-known that the expression $\psi(r,t) \tdefi \sum_{v \in  \Lambda_t} \exp \{- \norm{v}^2\}$ can be bounded by the number of lattice points $v \in \Lambda_t$ satisfying $\norm{v}_\infty \ll 1$. Combining this estimate together with the symplectic structure of $\Lambda_t$ (see Section \ref{qf-geom-num}) yields the estimate
\begin{equation*}
  \psi(r,t) \ll \frac{1}{M_{1}(\Lambda_t) \ldots M_{d}(\Lambda_t)} \asymp_d 
\alpha_d(\Lambda_t),
\end{equation*}
where $M_i(\Lambda_t)$ denotes the $i$-th successive minima of $\Lambda_t$ and $\alpha_d(\Lambda_t)$ the $d$-th $\alpha$-char\-ac\-ter\-is\-tic of 
$\Lambda_t$, that is $\alpha_d(\Lambda_t) =  \sup \{\abs{\det(\Lambda')}^{-1}: \Lambda' \text{ is a $d$-dimensional sublattice of $\Lambda_t$} \}$. After a local approximation of a certain one-parameter unipotent subgroup by the compact group $\mathrm{SO}(2)$ (see Section \ref{compact_groups}), we estimate the average of $\alpha_d(\Lambda_t)^\beta$ over $t$ for $0 < \beta \le 1/2$ in Lemmas \ref{GM}, \ref{special_int1} and \ref{special-int2}. This argument involves a recursion in the size of $r$ and builds upon a method developed in \cite{eskin-margulis-mozes:1998} on upper estimates of averages of certain functions on the space of lattices along translates of orbits of compact subgroups.

Let us give a brief sketch of the main ideas involved in this argument. Let $\mathrm G= \SL(2,\R), \, \mathrm K = \mathrm{SO}(2)$ endowed with 
the probability Haar measure $\dif k$ and denote by $A_r$ the mean-value operator on $\mathrm K \backslash \mathrm G$ defined by
\begin{equation*}
  A_r(f)(h) = \int_{\mathrm K} f(gkh) \,\dif k,
\end{equation*}
where $f$ is any continuous function on $\mathrm K \backslash \mathrm G$, 
$g \in \mathrm G$ denotes any element for which $\norm{g}=r$ and $\| \, \cdot \, \|$ denotes the operator norm induced by the standard Euclidean norm. Fixing $2/d < \beta \leq 1/2$, we shall show that uniformly in $v$ and for all intervals $I$ of fixed bounded length there exists a positive function $f$ depending only on $Q$ and $\beta$ such that
\begin{equation*}
  \int_I \abs{\theta_v(t)} \, \dif t \ll r^{d-\beta d} \abs{\det Q}^{-1/4} \gamma_{I,\beta}(r)\, A_{r}(f)(\mathbbm 1),
\end{equation*}
where $\gamma_{I,\beta}(r)$ contains information on the Diophantine properties of $Q$ and tends to zero for irrational forms as $r$ tends to infinity (see Corollary \ref{irr-dio}).

The function $f$ does not appear isolated but emerges as the maximum of a 
family of positive functions $f_1, \dots, f_{2d}$. For a positive number $r_0>0$ and any $g_0 \in \mathrm G$ such that $\norm{g_0}=r_0$ we show that this family satisfies two main properties. First, the value of each $f_i$ on any orbit of the form $g_0 \mathrm K h$ is bounded (up to a constant depending only on $r_0$) by its value at $f_i(h)$. Second, the mean-value $A_{r_0}(f_i)$ of any $f_i$  satisfies the following functional inequality (see Lemma \ref{Lemma 9.7})
\begin{equation*}
  A_{r_0}f_i \ll  \tau_{\lambda_i}(g_0) f_i + \max_{0<j \leq \bar i}\sqrt{f_{i-j}f_{i+j}},
\end{equation*}
where we set $\bar i = \min\{i,2d-i\}$, $\lambda_i := \max \{2, \beta 
\bar i\}$ and  $\tau_{\lambda_i}$ denotes the spherical function
\begin{equation*}
  \tau_{\lambda_i}(g)=  \int_{\mathrm K} \norm{g k e_1}^{-\lambda_i} \dif k,
\end{equation*}
where $e_1=(1,0)$ denotes the first standard unit vector on $\R^2$. 

The asymptotic growth of spherical functions is well-understood and in our case $\tau_{\lambda}(g) \asymp \norm{g}^{\lambda-2}$ whenever $\lambda >2$ and $g \not \in \mathrm K$. Here spherical functions are crucial precisely because they are the eigenfunctions of the mean-value operator. We show, in a first instance, that any positive function $f$ satisfying an inequality of the form
\begin{equation}
  \label{integralinequality}
  A_{r_0}f \ll \tau_\lambda(g_0)f + b \tau_\eta,
\end{equation}
for $\lambda>2$ and $0<\eta <\lambda$ satisfies
\begin{equation}
  \label{meanvalueestimate}
  A_r f (\mathbbm 1) \ll \tau_\lambda(g) f(\mathbbm 1),
\end{equation}
for any $r>0$, where $g \in \mathrm G$ is any element for which $\norm{g}=r$. In other words, the growth of the mean value at $\mathbbm 1$ grows at most as fast as the associated spherical function. In a second instance we obtain, after radializing the family, a preliminary estimate of the form
\begin{equation}
  \label{inductiveblabla}
  A_r(f)(\mathbbm 1) \ll_\mu f(\mathbbm 1)\tau_{\mu}(g),
\end{equation}
for any fixed $\mu > \lambda_d$. We then show inductively, using repeatedly \eqref{integralinequality}, \eqref{meanvalueestimate} and \eqref{inductiveblabla}  that
\begin{equation}
  \label{inductiveestimates}
  A_r(f_i) \ll f(\mathbbm 1)\tau_{\mu_i}, \text{ for all } i \neq d
\end{equation}
for an appropriate sequence $\lambda_d>\mu_i >\lambda_i$. Combining these 
estimates again with \eqref{integralinequality} in the case $i =d$ then 
yields the inequality $A_{r_0}f_d \ll\tau_{\lambda_d}(g_0)f_d + f(\mathbbm 1)\tau_\eta$, for some $\eta <\lambda_d$, which implies together with \eqref{integralinequality} and \eqref{inductiveestimates} the desired and expected estimate (see Theorem \ref{GM}), namely that  $A_r(f)(\mathbbm 1) \ll \tau_{\lambda_d}(g) f(\mathbbm 1) \asymp r^{\beta d -2} f(\mathbbm 1)$
for any $r \gg 1$ and any $g \in \mathrm G$ for which $\norm{g}=r$. In particular for any such interval $I$ we obtain the following bound
\begin{equation*}
\int_I \abs{\theta_v(t)} \, \dif t \ll r^{d-2} \abs{\det Q}^{-1/4} \gamma_{I,\beta}(r) f(\mathbbm{1}).
\end{equation*}
At this point the current approach is fundamentally different to the approach of previous effective bounds for $R(I_{E_{a,b} \cap r\Omega})$ by Bentkus and G\"otze \cite{bentkus-goetze:1999} (see also \cite{bentkus-goetze:1997}) valid for $d \geq 9$ and positive as well as indefinite forms. The reduction to \eqref{averagetheta} and the Diophantine factor $\rho_{Q,b-a}^{w}(r)$ follows the approach used by G\"otze in \cite{goetze:2004}, where the average on the right-hand side of \eqref{averagetheta} was estimated for $d \geq 5$ by methods from the Geometry of Numbers and essentially required \textit{positive definite forms}. A variant of that method was applied to \textit{split indefinite forms} in a PhD thesis by G.\ Elsner \cite{elsner:2009}.

\subsubsection{Smooth weights on $\Z^d$}
For the Gaussian weights $\specialv_r(x) = \exp\{-2 \4Q_{+}[x]/r^2\}$ our techniques yield effective bounds for the approximation of a weighted count of lattice points $m \in \Z^d$ with $Q[m] \in [a,b]$ by a corresponding integral with an error
\begin{equation}
 \label{remainedererroreq}
  R(I_{E_{a,b}} \4 \specialv_r) = \sum_{m\in E_{a,b} \cap \Z^d} \specialv_r(m) - \int_{ E_{a,b}} \specialv_r(x) \, \dif x.
\end{equation}
The following bounds for $R(I_{E_{a,b}} \4 \specialv_r)$ are \textit{identical} for the case of positive and indefinite $d$-dimensional forms $Q$, provided that $d \geq 5$. Using Vinogradov's notation $A \ll_B C$\index{1@ $A\ll_B C, A\asymp_B C$, Vinogradov's notation}, meaning that $A< c_B \, C$ with a constant $c_B>0$ depending on $B$, we have

\begin{theorem}
  \label{th:gaussian_weights}
  Let $Q$ be a non-degenerate quadratic form in $d\geq 5$ variables. Choose $\beta =\tfrac{2}{d} + \tfrac{\delta}{d}$ for some arbitrary small $\delta \in (0,\tfrac{1}{10})$. Then for any $r \ge q^{1/2}$, where $q$ denotes the maximal eigenvalue of $Q$, $b>a$ and $ 0 < w < (b-a)/4$ we have
  \begin{equation}
  \label{simplerror}
  R(I_{E_{a,b}}\4 \specialv_r) \ll_{Q,\beta,d} r^{d-2}\4 (w+\rho_{Q,b-a}^{w}(r))+ r^{d/2-1} (b-a),
\end{equation}
provided that $b-a \leq r$. If $r < b-a \ll r^2$ the second term in the bound has to be replaced by $ r^{d/2} \log{r}$.
\end{theorem}
 
 In Theorem \ref{maintheorem} an explicit description of the Diophantine factor $\rho_{Q,b-a}^{w}(r)$ will be provided. Depending on whether $Q$ is definite or indefinite, this factor will be further refined in Corollary \ref{positive}, resp.\ Corollary \ref{variable_smooth}. Moreover, the function $\rho_{Q,b-a}^{w}(r)$ tends to zero as $r$ tends to infinity if $Q$ is irrational. Additionally, if $Q$ is Diophantine of type $(\kappa,A)$, as we shall introduce in  Definition \ref{def:dio_type}, we find a polynomial decay $\rho_{Q,b-a}^{w}(r) \ll_{Q,d,A} r^{-\nu}$ for an appropriate choice of $0<w<(b-a)/4$, where $\nu \in (0,\infty)$ depends on $d$, $\kappa$ and $A$, see Corollary \ref{diophantex}. These results follow from Theorem \ref{maintheorem} with parameters chosen for the indefinite, positive and effective Diophantine cases in the proofs in Section \ref{subsection:Applications}.

\subsubsection{The role of the region $\Omega$}
\label{specialintro}
In order to estimate the lattice point deficiency \linebreak $R(I_{E_{a,b} \cap r\Omega})$ we have to $\eps$-smooth the indicator function of $\Omega$ which yields weights $\zeta=\zeta_{\eps}$ and an additional error of order $\eps (b-a) r^{d-2}$ in case of \textit{indefinite forms} due to the intersection of $E_{a,b}$ with the boundary $\partial r\Omega$. For \textit{positive definite forms}, $r \Omega$ contains $E_{a,b}$, that is $\eps >0 $ could be fixed independent of $r$, since this boundary intersection term is not present here.

In the indefinite case one needs to match the actual size of the error by choosing $\eps$ small enough in \eqref{roughbound}. This leads to a critical dependence on $\eps$ through the Fourier transform of $\zeta_{\eps}$ and its characteristics (see \eqref{defzeta*}). Here $\norm{\wh{\zeta_{\eps}}}_1$ moderately grows like $(\log 1/\eps)^d$ for arbitrary small $\eps$ in the case of polyhedra only, see Lemma \ref{l2}. The dependence of $\norm{\wh \zeta_{\eps}}_{1,*}$, see \eqref{defzeta*}, is again critically dependent on $\Omega$ and the width $b-a$ of the hyperbolic shell $E_{a,b}$. For $b-a \gg r$ the boundary of $ r\Omega \cap E_{a,b}$ will contain a larger segment of $\partial r\Omega$. For a sequence of scalings $r$ these segments of the $(d-1)$-polytope potentially contain a large number of lattice points which induce large errors in the lattice point approximation, for which the technical restriction to the region $\Omega$ is solely responsible. In order to avoid this artefact which is reflected by a large growth of $\norm{\wh{\zeta}_{\eps}}_{1,*}$ when $\eps$ is small, we restrict ourselves to special admissible regions $r\Omega$, where $\Omega=B^{-1}[-1,1]^d$, and $B\in \operatorname{GL}(d, \R)$ is chosen such that the lattice $\Gamma= B \Z^d$ is \textit{admissible} in the sense of Subsection \ref{specialp}, i.e.\ both \eqref{eq:sec7:Def:Omega} and \eqref{Def:Admissible} are satisfied. This ensures that the lattice point remainder of $r \Omega$ satisfies $\abs{\vol r\Omega - \v r\Omega} \ll_{\Omega} (\log r)^{d-1}$ uniformly which is `abnormally' small. Likewise $\norm{\wh \zeta_{\eps}}_{1,*}$ grows of order $(\log 1/\eps)^d$ only. The resulting error bounds in Corollary \ref{variable_smooth} for wide shells 
with $\max\{\abs{a},\abs{b}\} \ll_{B} r^2$ are then comparable up to at most $(\log 1/\eps)^d$ factors to the case of positive forms in Corollary \ref{positive}.

\subsection{Organization of this Paper}
The paper is organized mostly in logical order. In Section \ref{section:2} we describe the explicit technical estimates on lattice point remainders for both positive definite and indefinite forms. In the following Section \ref{section:3} we transfer the problem to Fourier transforms of the error starting with a first smoothing step and rewrite the lattice remainder in terms of integrals over $d$-dimensional theta sums. Section \ref{geom-num} provides a reformulation of the problem via upper bounds in terms of integrals over the absolute value of other theta sums with an underlying symplectic structure on $\R^{2d}$ which, in turn, are estimated using basic arguments from the Geometry of Numbers. Section \ref{section-aol} contains crucial estimates for averages of functions on the space of lattices. Finally, in Section \ref{conclus} all these results are combined to prove Theorem \ref{maintheorem}. Starting with the applications, we collect in Section \ref{section:parallelepipeds} the geometric bounds related to parallelepiped regions $\Omega$ used in this paper and afterwards conclude (in Subsection \ref{subsection:Applications}) the results of Section \ref{section:2}. In the last Section \ref{sec:small_values:qform} we focus on small values of indefinite quadratic forms: After recollecting and refining some results due to Schlickewei \cite{Schlickewei:1985} on the size of small zeros of integral quadratic forms, we shall prove Theorem \ref{dio-ineq}.

Compared to an earlier preprint \cite{goetze-margulis:2010} this version has been rewritten so that it allows to separate the error contributions due to the Diophantine properties of $Q$ and the influence of weights for 
the lattice points in Theorem \ref{maintheorem}. The latter has been developed for special choices of regions $\Omega$ which are particularly relevant for wide shells $E_{a,b}$ in Section \ref{section:parallelepipeds}. Moreover, the effective bounds for non-trivial solutions of the Diophantine inequality $\abs{Q[m]} < \eps$ have been improved in terms of the signature $(r,s)$ by using Schlickewei's result \cite{Schlickewei:1985} on small zeros of quadratic forms. In addition, we included a number of corrections concerning the explicit dependence on $Q$ (resp.\ $\Omega$) and the 
dimensions, and corrected typos as well.



\newpage
\section{Effective Estimates}
\label{section:2}
\noindent We consider the quadratic form
\begin{equation*}
  Q[\4x\4] \defi \langle \4 x, \4 Q x\4\rangle \quad \text{for} \quad x \in \R^d,
\end{equation*}
where $\langle \cdot ,\cdot \rangle$ resp.\ $\norm{\,\cdot \,}$ denote the standard Euclidean scalar product and norm\index{1@ $\langle \, \cdot \,, \, \cdot \, \rangle$, $\norm{\, \cdot  \,}$, Euclidean inner product and associated norm}, $Q \colon \R^d\to \R^d$ denotes a symmetric linear operator in $\GL(d, \R)$\index{Q@ $Q$, as quadratic form and the corresponding symmetric matrix} with eigenvalues $q_1,\dots, q_d$. Write\index{Q@ $q$, largest eigenvalue of $Q$ in absolute value}\index{Q@ $q_0$, smallest eigenvalue of $Q$ in absolute value, $q_0 \geq 1$}\index{Q@ $Q[x]=\langle Q x,x \rangle$, Siegel's notation}\index{D@ $d_Q:= \abs{\det Q}^{-1/2}$}
\begin{equation}
  \label {eq:1.1}
  q_0\defi\min_{1\leq j\leq d}\, \abs{q_j}, \quad\quad q\defi\max_{1\leq j\leq d}\, \abs{q_j}, \quad\quad d_Q\defi \abs{\det Q}^{-1/2}.
\end{equation}
In what follows we shall always assume that the form is \textit{non-degenerate}, that is $q_0>0$. In order to describe the explicit bounds we need to introduce some more notations. Let $\beta > \tfrac{2}{d}$ such that $0< \tfrac 1 2 -\beta <\tfrac 1 2 - \tfrac 2 d$ for $d>4$\index{B@ $\beta$, exponent in the range $(\frac{2}{d},\frac{1}{2})$}. For a lattice $\Lambda \subset \R^n$, $n \in \N$, with $\dim \Lambda = n$ we define for $1\le l\le n$ its $\alpha_l$-characteristic\index{A@ $\alpha,\alpha_l$-characteristic of a lattice} by
\begin{equation}
  \label{alp0}
  \alpha_l( \Lambda)\defi \sup \Bigl\{\abs{\det(\Lambda')}^{-1}: \Lambda'\subset \Lambda,\ \ l\text{-dimensional sublattice of $\Lambda$} \Bigr \}.
\end{equation}
Here $\Lambda'=B\,\Z^n$ is determined by a $n \times l$-matrix $B$ and $\det(\Lambda')= \det(B^T\, B)^{1/2}$ is the volume of a fundamental domain.

\begin{remark}
  \label{remark_existence_alpha_l}
  Given $\Lambda = g \Z^n$ with $g \in \GL(n,\R)$, then any $l$-dimensional sublattice $\Delta \subset \Lambda$ is spanned by $g n_1,\ldots,g n_l$, where $n_i \in \Z^n$ and $\det(\Delta) = \norm{g n_1 \wedge \ldots \wedge g n_l}$. If $\Delta' \subset \Lambda$ is a sublattice distinct from $\Delta$ with basis $g n_1',\ldots,g n_l'$, $n_i' \in \Z^n$, then
  \begin{equation*}
    \norm{(g n_1 \wedge \ldots \wedge g n_l) - (g n_1' \wedge \ldots \wedge g n_l')} \gg_g \norm{(n_1 \wedge \ldots \wedge  n_l) - (n_1' \wedge \ldots 
\wedge n_l')} \ge 1,
  \end{equation*}
  since the $l$-th exterior product of $g$ is invertible. This argument shows that the $\alpha_l$-characteristic is attained at some $l$-dimensional sublattice $\Lambda' \subset \Lambda$.
\end{remark}

In the special case $n=2d$ we also introduce\index{G@ $\gamma_{[a,b],\beta}(r)$, Diophantine factor for $Q$ on $[a,b]$ with exponent $\beta$}
\begin{equation}
  \label{defgamma0}
  \gamma_{[T_{-},T],\beta}(r) \defi \sup \; \bgl\{ \bigl(r^{-d}\alpha_d(\Lambda_t)\bigr)^{1/2-\beta}:\,\,T_{-} \le \abs{t} \le T \bgr\},
\end{equation}
where $\Lambda_t = d_r u_t \Lambda_Q$ denotes a $2d$-dimensional lattice obtained by an appropriate action of $d_r, u_t \in \SL(2,\R)$ on $\R^{2d}$ (see \eqref{svo4}), where $d_r$ and $u_t$ denote the usual diagonal and unipotent elements and $\Lambda_Q$ denotes a fixed $2d$-dimensional lattice depending on $Q$ (see \eqref{svo5}). Recall that $E_{a,b} = \{ x \in \R^d \4 : \4 a < Q[x] < b \}$ and let $\specialv(x)$ denote a smooth weight function\index{V@ $\specialv$, weight function on $\R^d$ of sufficiently fast decay} such that $\zeta(x) \tdefi \specialv(x) \exp\{Q_{+}[x]\}$ satisfies
\begin{equation}
  \label{zeta-decay} \textstyle
  \sup_{x\in \R^d}\big(\abs{\zeta(x)} + \abs{\wh{\zeta}(x)}\big) (1+\norm{x})^{d+1} < \infty.
\end{equation}
An explicit construction of weight functions for parallelepiped regions will be given in Section \ref{section:parallelepipeds}. Nevertheless, as a simple example, one can take the Gaussian weights $\specialv(x) = \exp\{-2 Q_{+}[x]\}$.




\begin{theorem}
  \label{maintheorem}
  Let $Q$ be a non-degenerate quadratic form in $d\geq 5$ variables with $q_0\ge 1$. Choose $\beta =\tfrac{2}{d} + \tfrac{\delta}{d}$ for some arbitrary small $\delta \in (0,\tfrac{1}{10})$. Write $(b-a)_q \tdefi b-a$ 
if $b-a \leq q$ and $(b-a)_q\tdefi q^{\beta d -1/2}$\index{1@ $(b-a)_q := 
(b-a)I(b-a \le q) + q^{(2\beta d -1)/2} I(b-a > q)$} if $b-a >q$, and $(b-a)^* \tdefi (b-a)$ if $b-a \leq 1$\index{1@ $(b-a)^* \defi (b-a)I(b-a \le 1) + I(b-a > 1)$} and $(b-a)^* \tdefi 1$ if $b -a >1$. Then for any $r \ge q^{1/2}$, $b>a$ and $ 0 < w < (b-a)/4$ we have
  \begin{equation}
    \label{eq:mth:1}
    \begin{aligned}
      \Big\lvert\sum_{m\in E_{a,b}\cap \Z^d} \specialv(\tfrac{m}{r}) - \int_{E_{a,b}} \specialv (\tfrac{x}{r}) \, \dif x \Big\rvert \ll_{\beta,d} \big\{ w \4\norm{\specialv}_Q & + \norm{\wh\zeta}_1 C_Q \4 \rho_{Q,b-a}^{w}(r)\big\} r^{d-2} \\
      & \! \! \! + d_Q \4 r^{d/2}\norm{\hat \zeta}_{*,r} \log\big(1+ \tfrac{\abs{b-a}} {q_0^{1/2} r}\big),
    \end{aligned}
  \end{equation}
  where $C_Q \tdefi q \4 \abs{\det{Q}}^{-1/4-\beta/2}$\index{C@ $C_Q := 
q \4 \abs{\det{Q}}^{-1/4-\beta/2}$} and $\norm{\specialv}_Q$ is defined in Lemma \ref{l3} (the quantity $\norm{\specialv}_Q$ depends additionally on $r,a,b$ and $w$, but we will suppress this dependence),\index{R@ $\rho_{Q,b-a}^{w}(r)$}
  \begin{equation*}
    \begin{aligned}
      \rho_{Q,b-a}^{w}(r) \defi \inf \Big\{ (b{-}a)_q\4 \big( c_Q T_{-}^{\frac{d}{2}-2-\delta} \! \! + \gamma_{[T_{-},1],\beta}(r)\big) + \gamma_{(1,T_+],\beta}(r) \4 \big(1{+}\log((b{-}a)^* \4 T_+)\big) \ \  & \\
      + c_Q^{-1} \4 (T_+ w)^{-1/2} \4 \e^{- (T_+w)^{1/2}} : T_{-} \in [\imbound,1], \ T_+ \ge 1 \Big\}&
    \end{aligned}
  \end{equation*}
  and $c_Q \tdefi \abs{\det{Q}}^{1/4-\beta/2}$\index{C@ $c_Q := \abs{\det{Q}}^{1/4-\beta/2}$}. Furthermore\index{Z@ $\norm{\wh \zeta}_{*,r} := 
q^{d/4}\big((\frac{q}{q_0})^{d/2} \norm{\wh{\zeta}}_1 + \int_{\norm{v}_{\infty} > r/2} \frac{\abs{\wh \zeta(v)}} {({q^{1/2}} r^{-1} +\norm{ v r^{-1} }_{\Z})^{d/2}} \, \dif v\big)$}
  \begin{equation}
    \label{defzeta*}
    \norm{\wh \zeta}_{*,r} \defi q^{d/4}\Big( \Big(\frac{q}{q_0}\Big)^{d/2} \4 \norm{\wh{\zeta}}_1 + \int_{\norm{v}_{\infty} > r/2} \frac{\abs{\wh 
\zeta(v)}} {({q^{1/2}} r^{-1} +\norm{ v r^{-1} }_{\Z})^{d/2}} \, \dif v\Big)
  \end{equation}
  and here $\norm{v}_{\Z} \tdefi \min_{m \in \Z^d} \norm{v-m}_{\infty}$.
\end{theorem}

We use the notation $A \asymp_d B$ for quantities of equivalent size up to constants depending on $d$ only, i.e.\ $A \ll_d B \ll_d A$.\index{1@ $A\ll_B C, A\asymp_B C$, Vinogradov's notation}

\begin{remark}
  \label{mainrem}
  Note that
  \begin{enumerate}[label=\alph*)]
    \item Theorem \ref{maintheorem} extends to affine quadratic forms $Q[x+\xi]$ uniformly in $\abs{\xi}_{\infty} \le 1$.
    \item Depending on the application, the lattice remainder \eqref{eq:mth:1} will be optimized in the parameters $w$, $\eps$ and $T_{+}$ differently: For thin shells the error should also scale with the length $b-a$. This forces $T_{+}$ to be large and requires `strong' Diophantine assumptions. In the case of wide shells it is possible to choose $w$ relatively large.
    \item If $Q$ is irrational, then Corollary \ref{irr-dio} implies that 
$\rho_{Q,b-a}^{w}(r) \to 0$ for $r \to \infty$, provided that $w$ and $(b-a)$ are fixed. The first factor in the definition of $\rho_{Q,b-a}^{w}$ corresponds to small values of $t$ on the Fourier side and the last factor to the decay rate of the $w$-smoothing of the interval $[a,b]$.
  \end{enumerate}
\end{remark}

With these notations we state a result providing quantitative bounds for the difference between the volume and the lattice point volume in $E_{a,b}$.



\subsection{Ellipsoids \texorpdfstring{$E_{0,b}$}{}}
Here $Q$ is positive definite and we may assume that $b$ tends to infinity. Let $r=\sqrt{2b}$ in Theorem \ref{maintheorem}. Then the ellipsoid $E_{0,b}=\{ x \in \R^d\;:\; Q[x] \le b \}$\index{E@ $E_{a,b} \tdefi \{ x \in \R^d \, \colon \, a < Q [x] < b \}$, hyperbolic or ellipsoidal shell} is contained in $r\Omega = Q_+^{-1/2}[-r,r]^d$. Choosing in Theorem \ref{maintheorem} a smoothing of $I_{\Omega}$, say $\specialv_{\eps}$ of width $\eps=\tfrac{1}{15}$, which equals $1$ on $E_{0,b}$, and the smoothing parameter $w$ in terms of $T_+$, such that the right-hand side in \eqref{eq:mth:1} is minimal, will lead to

\begin{corollary}
  \label{positive}
  Let $Q$ denote a non-degenerate $d$-dimensional positive definite form with $d \ge 5$ and $q_0 \ge 1$. For any $r \ge q^{1/2}$ and $r= \sqrt{2b}$ we have with $H_r \tdefi E_{0,b}$\index{H@ $H_r \tdefi E_{0,b}$, with 
$r = \sqrt{2b}$, if $Q$ is positive definite}
  \begin{equation}
    \abs{\volu_{\Z} H_r - \volu H_r} \ll_{\beta,d} d_Q \4 r^{d-2} \big( \rho_Q^{\mathrm{ell}}(r) + q^{d/4}\4 r^{-d/2+2} \4 (q/q_0)^{d/2} \4 \log(r) \big),
  \end{equation}
  where\index{R@ $\rho_Q^{\mathrm{ell}}(r)$}
  \begin{equation*}
    \rho_Q^{\mathrm{ell}}(r) \hspace{-1mm} \defi \hspace{-1mm} \inf \Big\{ a_Q \big( q^{\tfrac{3}{2}+\delta} (c_Q \4 T_{-}^{\frac{d}{2}-2-\delta} \! + \gamma_{[T_{-},1],\beta}(r)) \! +  \gamma_{(1,T_+],\beta}(r) \log(T_{+}\!+ \! 1) \big) \! + \! \tfrac{\log(1+q \4 T_+)^2}{T_+} \Big\}
  \end{equation*}
  and the infimum is taken over $T_{-} \in [\imbound,1]$ and $T_+ \geq 1$, where $a_Q = q \abs{\det{Q}}^{\frac{1}{4}-\frac{\beta}{2}}$, $c_Q = 
\abs{\det{Q}}^{1/4-\beta/2}$. Furthermore, $\lim_{r\to \infty}   \rho_Q^{\mathrm{ell}}(r) =0$ as $r$ tends to infinity, provided that $Q$ is irrational.
\end{corollary}

Compared to the quantitative results in \cite{bentkus-goetze:1997} and \cite{bentkus-goetze:1999}, this bound holds already for $d \ge 5$. Moreover, Corollary \ref{positive} refines the estimates obtained in \cite{goetze:2004}.



\subsection{Hyperboloid Shells \texorpdfstring{$E_{a,b}$}{}}
If $Q$ is indefinite, we distinguish, depending on $b-a$, between `small' 
and `wide' shells $E_{a,b}$. Here we restrict ourselves to a special class of rescaled \textit{admissible} parallelepipeds $r\Omega$ for $r>0$: We 
suppose that $\Omega=B^{-1}[-1,1]^d$ is determined by some $B\in \operatorname{GL}(d, \R)$ such that the lattice $\Gamma= B \Z^d$ is admissible in the sense of Subsection \ref{specialp}, i.e.\ both \eqref{Def:Admissible} and \eqref{eq:sec7:Def:Omega} should be satisfied (for examples, see Remark \ref{remark:app_a:set_of_admis} and Example \ref{example:app_a:admis_latt}). Note that the latter condition \eqref{eq:sec7:Def:Omega}, that is $Q_{+} \leq B^T B \leq c_B Q_{+}$ with $c_B \geq 1$,  ensures that the region $\Omega$ is rescaled with respect to the quadratic form $Q$.

To estimate the lattice point remainder for this restriction of $E_{a,b}$ 
given by $H_{r} \tdefi E_{a,b}\cap r\Omega$\index{H@ $H_r := E_{a,b} \cap r\Omega$, if $Q$ is indefinite} we smooth the indicator function $I_{\Omega}$ in an $\eps$-neighborhood with an error of order $\mathcal O(\eps 
(b-a) r^{d-2})$ using Lemma \ref{l3}. This yields a smooth function $\specialv_{\eps}$ and a final weight function $\zeta_{\eps}$, according to \eqref{zeta-decay} in Theorem \ref{maintheorem}. Since $\Omega$ is admissible, both $\norm{\zeta_{\eps}}_1$ and $\norm{\zeta_{\eps}}_{*,r}$ in \eqref{defzeta*} are growing with a power of $\abs{\log \eps}$ only, see Lemmas \ref{l2} and \ref{omega-remainder}.

In the next step we calibrate both smoothing parameters $w$ and $\eps$ in order to get Corollary \ref{variable_smooth} below for `wide' and  `thin' shells. The actual choice of $\eps$ is then determined by calibrating the main terms $\eps r^{d-2}$ and $\norm{\zeta_{\eps}}_1 \rho_{Q,b-a}^{w}(r) r^{d-2}$ depending on the speed of convergence of $\lim_{r\to \infty} \rho_{Q,b-a}^{w}(r)=0$. The resulting error bound for \textit{indefinite forms} will then differ at most by some $\abs{\log \eps}$-factors from the positive definite case, and is thus dominantly influenced by the Diophantine properties reflected in the decay of the $\gamma_{[T_{-},T_{+}],\beta}$, resp.\ the $\rho_{Q,b-a}^{w}$-characteristic of irrationality. In particular we have uniformly for `small' and `wide' shells $E_{a,b}$ and admissible regions $\Omega$ the following bound:

\begin{corollary}
  \label{variable_smooth}
  Under the assumptions of Theorem \ref{maintheorem} we get for an admissible region $\Omega$, all $\max\{\abs{a},\abs{b}\} \le c_0 r^2$, where $c_0>0$ is chosen as in Lemma \ref{l3}, and $b-a \geq q$\index{D@ $\Delta_r := \abs{ \vol H_r -  \volu H_r}$, lattice point deficiency}
  \begin{equation}
    \label{eq:variable_smooth:1}
    \Delta_r \defi \abs{ \vol H_r -  \volu H_r} \ll_{\beta,d} d_Q \4 r^{d-2} \big( \rho_{Q,b-a}^{\mathrm{hyp}+}(r) + R_{Q,A}(r) \big),
  \end{equation}
  where
  \begin{equation}
    \label{eq:variable_smooth:2}
    R_{Q,A}(r) \defi q^{\frac{d}{4}} \4 r^{-\frac{d}{2}+2} \log(r {+} 1)^d \big( (\tfrac{q}{q_0})^{\frac{d}{2}} + \tfrac{c_B^{d/2} q_0^{-\frac{d}{4}}}{\Nm(\Gamma)} \4 \log(2 {+} \tfrac{1}{\Nm(\Gamma)})  \big)  \log \big(1 {+} \tfrac{b-a}{q_0^{1/2} r}\big),
  \end{equation}
  $\Nm(\Gamma) \tdefi \inf_{\gamma \in \Gamma \setminus \{0\}} \abs{\gamma_1 \ldots \gamma_d}$ in standard coordinates $\gamma = (\gamma_1,\ldots,\gamma_d)$ and\index{R@ $\rho_{Q,b-a}^{\mathrm{hyp}+}(r)$}
  \begin{equation*}
    \begin{aligned}
      \rho_{Q,b-a}^{\mathrm{hyp}+}(r) \defi {\inf}^*_{T_{+},T_{-}} \Big\{ \log \big( (b-a)T_{-}^{-(\frac{d}{2}-2-\delta)}{+}1 \big)^d \Big(  a_Q \4 
q^{\frac{3}{2}+\delta} ( c_Q T_{-}^{\frac{d}{2}-2-\delta} + \gamma_{[T_{-},1],\beta}(r)) \ \ & \\
      + a_Q \gamma_{(1,T_{+}],\beta}(r) \log(T_{+}+1)  + \tfrac{\log(q T_{+}+1)^2}{T_{+}} \Big) \Big\} & ,
    \end{aligned}
  \end{equation*}
  where the infimum is taken over all $T_{-} \in [\imbound,1]$ and $T_{+} 
\geq 1$. If $b-a \leq q$, then \eqref{eq:variable_smooth:1} holds, too, whereby the Diophantine factor $\rho_{Q,b-a}^{\mathrm{hyp}+}(r)$ has to be replaced by\index{R@ $\rho_{Q,b-a}^{\mathrm{hyp}-}(r)$}
  \begin{equation*}
    \begin{aligned}
      \rho_{Q,b-a}^{\mathrm{hyp}-}(r)  \defi {\inf}_{T_{-},T_{+}}^* \Big\{ a_Q \log \big(1\!+\!T_{-}^{-(\frac{d-4}{2}-\delta)}\big)^d \Big( &  (b-a) \big( c_Q \4 T_{-}^{\frac{d}{2}-2-\delta}+ \gamma_{[T_{-},1],\beta}(r)\big) \\
      & +\gamma_{(1,T_+],\beta}(r) \4 (\log( (b-a)^* \4 T_+)+1) \Big) \! \Big\}.
    \end{aligned}
  \end{equation*}
  In the last equation the infimum is taken over all $T_{-} \in [\imbound,1]$ and $T_{+} \geq 1$ with
  \begin{equation}
    \label{eq:cor2.4:small}
    T_{+} \geq 4 (b-a)^{-1} T_{-}^{-(\frac{d}{2}-2-\delta)}  \max \Big \{1,\log \big(c_Q^2 (b-a) T_{-}^{\frac{d}{2}-2-\delta}\big)^2 \Big\}.
  \end{equation}
\end{corollary}

These bounds refine the results obtained in \cite{bentkus-goetze:1999} providing explicit estimates in terms of $Q$ and are valid for $d \geq 5$. Note that, due to the `uncertainty principle' for the Fourier transform, we need to choose $T_{+}$ at least as large as in \eqref{eq:cor2.4:small} 
if $E_{a,b}$ is `thin' in order to control the factor $\exp \{- \abs{T_{+} w}^{1/2} \}$ (occurring in the definition of $\rho_{Q,b-a}^{w}$) which scales with $b-a$. In Section \ref{subsection:Applications} we prove a variant of Corollary \ref{variable_smooth} for thin shells 
and non-admissible regions $\Omega$ as well, see Corollary \ref{variable_smooth:non-admissible}.

\subsection{Quadratic Forms of Diophantine Type \texorpdfstring{$(\kappa,A)$}{(k,A)}}
For any fixed $T_{+} >1 > T_{-}>0$ and irrational $Q$ it is shown in Corollary \ref{irr-dio} that
\begin{equation}
  \lim_{r\to \infty} \gamma_{[T_{-},T_{+}],\beta}(r) = 0,
\end{equation}
with a speed depending on the Diophantine properties of $Q$. For indefinite forms $Q$, this implies for fixed $b-a >0$ that
\begin{equation}
  \lim_{r \to \infty} \rho_{Q,b-a}^{\mathrm{hyp}+}(r) =0, \q \lim_{r \to 
\infty} \rho_{Q,b-a}^{\mathrm{hyp}-}(r) =0
\end{equation}
and hence $\Delta_r = o(r^{d-2})$ as $r\to \infty$. This holds \textit{uniformly} for all intervals $[a,b]$ with $0<u_r \le b-a \le v_r\le c_0 r^2$ and sequences $\lim_r u_r=0$, $\lim_r v_r = \infty$, $r \to \infty$ depending on $Q$. For the special class of quadratic forms of Diophantine type $(\kappa,A)$, as introduced in Definition \ref{def:dio_type}, we 
may apply Corollary \ref{irr-dio} to obtain explicit bounds on the Diophantine factors in the previous theorems as follows.

\begin{corollary}
  \label{diophantex}
  Consider an indefinite quadratic form $Q$ that is Diophantine of type $(\kappa,A)$. Moreover, let $\beta = 2/d + \delta/d $ for some sufficiently small $0<\delta < \tfrac{1}{10}$. Then for the case of wide shells $b-a \ge q$ in Corollary \ref{variable_smooth} we have
  \begin{equation}
    \label{diophantex0}
    \rho_{Q,b-a}^{\mathrm{hyp}+}(r) \ll_{\beta,d} \log(r+1)^d \4 h_Q \4 q^{\frac{3}{2}+\nu+\delta} (1+ A^{-\nu}) ( r^{-\frac{d-2(2+\delta)}{d(\kappa +1)+1}} + r^{-\frac{2\nu}{\kappa \nu +1}} \log(q \4 r +1 )),
  \end{equation}
  where $h_Q = q \4 \abs{\det{Q}}^{1/2-\beta}$, $\nu = (1-2\beta)/(2\kappa + 2)$ and $\sigma = d(1/2-\beta)$. Thus for an admissible region $\Omega$ satisfying \eqref{eq:sec7:Def:Omega} we have for all $r \ge q^{1/2}$ and $\max \{ \abs{a},\abs{b}\} \le c_0 r^2$
  \begin{equation}
    \label{diophantex1}
    \Big\lvert \frac{\volu_{\Z} H_r}{\volu H_r} -1 \Big\rvert \ll_{Q,\Omega,\beta,d} 
\frac{\log(r+1)^d}{b-a} \Big( r^{-\frac{(1-2\beta)d}{1+(\kappa+1)d}} + r^{- \frac{2-4\beta}{2+(3-2\beta) \kappa}} + r^{-\frac{d}{2}+2} \log \big(1+ \tfrac{b-a}{r}\big) \Big),
  \end{equation}
  where the implied constant in \eqref{diophantex1} can be explicitly determined. For thin shells, i.e.\ $b-a \le q$, we have
  \begin{align*}
    \rho_{Q,b-a}^{\mathrm{hyp}-}(r) \ll_{\beta,d} {\inf}_{T_{-},T_{+}}^* \Big\{  h_Q \log \big(1+T_{-}^{-(\frac{d}{2}-2-\delta)} \big)^d \big(  (b-a) (T_{-}^{\frac{d}{2}-2-\delta} + q^\nu A^{-\nu} \4 T_{-}^{-\nu} r^{- 2\nu } ) \q & \\
    + q^\nu A^{-\nu} \4 T_{+}^{ \kappa \nu} r^{-2 \nu} (\log((b-a)^*T_{+}) \big)+1)\big) & \Big\},
  \end{align*}
  where the infimum is taken over all $T_{-} \in [\imbound,1]$ and $T_{+} 
\geq 1$ restricted to
  \begin{equation*}
    T_{+} \geq 4 (b-a)^{-1} T_{-}^{-(\frac{d}{2}-2-\delta)} \max \Big \{1, \log \big(c_Q^2 (b-a) T_{-}^{-(\frac{d}{2}-2-\delta)}\big)^2 \Big\}.
  \end{equation*}
\end{corollary}



\section{Fourier Analysis}
\label{section:3}

\subsection{Smoothing}
\label{subsection:3.1}
The first step in the proof of Theorem \ref{maintheorem} is to rewrite the lattice point counting error (i.e.\ the left hand side of \eqref{eq:mth:1}) in terms of integrals over appropriate smooth functions. To this end, we introduce smooth approximation of the indicator functions of $E_{a,b}$ and $\Omega$ constructed as follows. Denote by $k=k(x) \4 \dif x$ a probability measure (symmetric around $0$) with compact support satisfying $k([-1,1])=1$ and $ \abs{\wh{k}(t)} \le C \exp\{-\abs{t}^{1/2}\}$ for all $t\in\R$ and a positive constant $C>0$, where $\wh{k}(t) \tdefi \int k(x) \exp \{ - 2 \pi \iu \4 t\4 x \} \, \dif x$ denotes the Fourier transform of the measure $k$\index{K@ $k$, compactly supported kernel with sufficiently fast decaying Fourier transform}. For an example of $k$ we refer to Corollary 10.4 in \cite{bhattacharya-rangarao:1986}. More generally, by a result of Ingham \cite{Ingham:1934} (see e.g.\ Theorem 10.2 in \cite{bhattacharya-rangarao:1986}) there is a probability density $k$ such that $\abs{\wh{k}(t)} \le C \exp \{ - u(\abs{t}) \abs{t} \}$, where $u$ is a continuous, non-negative, non-increasing function on $[0,\infty)$ satisfying $\int_1^\infty u(t) \4 t^{-1} \, \dif t < \infty$ and this condition is also necessary. However, we will not need this improved decay rate. For $\tau>0$ let $k_{\tau}$ denote the rescaled measures $k_{\tau}(A) \tdefi k(\tau^{-1}A)$ for any $A\in\mathcal{B}^d$, where $\mathcal{B}^d$ denotes the Borel $\sigma$-algebra. Using the same notation, let $k_{\tau}(x) = k_{\tau}(x_1)\ldots k_{\tau}(x_d)$, $x=(x_1,\ldots,x_d)$, denote its multivariate extension on $\R^d$, $d \geq 1$. Furthermore, let $f*k_{\tau}$ denote the convolution of a function $f$ on $\R^d$ and $k_{\tau}$. We need the following standard estimate for smooth approximations.

\begin{lemma}
  \label{lemma:1}
  Let $\mu$ and $\nu$ be (positive) finite measures on $\R^d$, let $f$ and $f_{\tau}^{\pm}, \tau>0$, denote bounded real-valued Borel-measurable functions on $\R^d$ such that for any $\tau >0$
  \begin{equation}
    \label{smooth-dom}
    \begin{aligned}
       & f^{-}_{\tau}(x) \le \inf\{ f(y) : \norm{y -x}_{\infty} \hspace{-2pt} < \hspace{-2pt}\tau\}              & \text{and} &  & f^{+}_{\tau}(x)\ge \sup\{ f(y) : \norm{y -x}_{\infty}\hspace{-2pt} < \hspace{-2pt}\tau\}, \\
       & f^{-}_{2\tau}(x) \le \inf\{ f^{-}_{\tau}(y) : \norm{y -x}_{\infty} \hspace{-2pt} < \hspace{-2pt} \tau\} & \text{and} &  & f^{+}_{2\tau}(x) \ge \sup\{ f^{+}_{\tau}(y) : \norm{y -x}_{\infty}\hspace{-2pt} < \hspace{-2pt} \tau\}.
    \end{aligned}
  \end{equation}
  Then
  \begin{equation}
    \label{smooth-basic}
    \Big\lvert\int f \, \dif (\mu-\nu)\Big\rvert \le \max_{\pm} \Big\lvert \int f^{\pm}_{\tau} \, \dif (\mu-\nu)*k_{\tau} \Big\rvert + \int (f^{+}_{2\4 \tau} - f^{-}_{2\4 \tau}) \, \dif \nu.
  \end{equation}
\end{lemma}

\begin{proof}
  Note that $k_{\tau}$ is a probability measure with support contained in 
a $\norm{\cdot}_\infty$-ball of radius $\tau$. Hence, \eqref{smooth-dom} implies the following chain of inequalities
  \begin{equation}
    \label{smooth-basic2}
    f^{-}_{2\tau} \le  f^{-}_{\tau}*k_{\tau} \le  f \le f^{+}_{\tau}*k_{\tau}\le f^{+}_{2\tau},
  \end{equation}
  which leads to
  \begin{equation}
    \label{smooth-basic3}
    \int f \, \dif (\mu-\nu) \leq \int f^{+}_{\tau}*k_{\tau} \, \dif(\mu-\nu)+ \int (f^{+}_{\tau}*k_{\tau}-f) \, \dif \nu
  \end{equation}
  together with a similar lower bound. Since by \eqref{smooth-basic2} $f \le f^{+}_{\tau}*k_{\tau}\le f^{+}_{2\tau}$ and $f \ge f^{-}_{\tau}*k_{\tau}\ge f^{-}_{2\tau}$, the upper bound \eqref{smooth-basic3} together with the corresponding lower bound proves the lemma.
\end{proof}

First we shall investigate approximations to the sum under consideration, counting the lattice points in $E_{a,b}$ with weights $\specialv_r(x) \tdefi \specialv(x/r)$\index{1@ $\phi_r(x)= \phi(x/r)$, for a function $\phi$ on $\R^m$}\index{R@ $R(I_{E_{a,b}} \4 \specialv_r), R(g \4 \specialv_r),R( I_{E_{a,b} \cap r\Omega})$, lattice point remainder}. In accordance with the notation introduced in \eqref{eq:df:formal_lattice_rem} at the beginning of Section \ref{subsec:intro_fourier}, we write
\begin{equation}
  \label{smooth-remainder}
  \sum_{m\in \Z^d} I_{[a,b]}(Q[m]) \specialv_r (m) = \int_{\R^d} I_{[a,b]}(Q[x]) \specialv_r (x) \4 \dif x +R(I_{E_{a,b}} \specialv_r),
\end{equation}
where $\specialv(x)$ is a sufficiently fast decreasing smooth function such that the function\index{Z@ $\zeta(x):= \specialv(x)\4 \exp\{ Q_{+}[x]\}$}
\begin{equation}
  \label{zeta-def}
  \zeta(x)\defi \specialv(x)\4 \exp\{ Q_{+}[x]\}
\end{equation}
satisfies \eqref{zeta-decay}. For such weights both sides of \eqref{smooth-remainder} are well defined and $R(I_{E_{a,b}} \specialv_r)$ may be estimated by Poisson's formula, see \cite{bochner:1948}, \S 46. By means of Lemma \ref{lemma:1} we now replace the indicator $I_{[a,b]}$ by a smooth approximation.

\begin{corollary}
  \label{co1}
  Let $[a,b]_{\tau} \tdefi [a-\tau,b +\tau]$ and write\index{G@ $g_{\pm w} := I_{[a,b]_{\pm w}}*k_{w}$}\index{G@ $g^Q_{\pm w}(x):= g_{\pm w} (Q[x]), \ x\in \R^d$}
  \begin{equation*}
    g_{\pm w} \defi I_{[a,b]_{\pm w}}*k_{w} \q \text{and} \q g^Q_{\pm w}(x)\defi g_{\pm w} (Q[x]), \ x\in \R^d,
  \end{equation*}
  where $0 < w < (b-a)/4$. Then
  \begin{equation}
    \label{zeta-boundary}
    \abs{R(I_{E_{a,b}}\specialv_r)} \le  \max_{\pm} \4 \abs{R(g^Q_{\pm w} 
\4\specialv_r)}+c_d \4 w \norm{\specialv}_Q \4 r^{d-2},
  \end{equation}
  where $R(g^Q_{\pm w}\specialv_r)$ is defined in accordance to \eqref{smooth-remainder}, $\norm{\specialv}_Q$ is defined in Lemma \ref{l3} and $c_d$ is a positive constant depending on $d$ only.
\end{corollary}

\begin{proof}
  In Lemma \ref{lemma:1} we choose the measure $\mu$, resp.\ $\nu$, on $\R$ 
as the induced measure under the map $x\mapsto Q[x]$ of the counting measure with weights $\specialv_r(m)$, resp.\ the measure $\specialv_r(x)\, \dif x$. Let $f(z)=I_{[a,b]}(z)$ and $f^{\pm}_{\tau}(z)=I_{[a,b]_{\pm \tau}}(z)$. Then \eqref{smooth-dom} is satisfied and \eqref{smooth-basic} applies with $\tau=w$. In order to bound the remainder term in \eqref{smooth-basic} observe that
  \begin{equation*}
    f^{+}_{2\4 w} - f^{-}_{2\4 w}\le I\big(\{ x \in \R^d\, : \, Q[x] \in [ a-2 w ,a+2w] \cup [ b-2w,b+2w ] \}\big)
  \end{equation*}
  and apply the geometric estimate of Lemma \ref{l3}; that is \eqref{Q-smooth1} of Subsection \ref{subsection:smoothing_regions}.
\end{proof}

Thus we have reduced the determination of the lattice point remainder $R(I_{E_{a,b}} \4\specialv_r)$ to the remainder $R(g^Q_{\pm w}\specialv_r)$ for smooth weights. In the next subsection we shall rewrite the latter by means of the corresponding Fourier transforms.



\subsection{Fourier Transforms and Theta-Series}
Rewrite the weight factor $\specialv$ in \eqref{smooth-remainder} as $\specialv(x) = \exp \{- Q_{+}[x] \} \4 \zeta(x)$. Since by definition (see the previous Subsection \ref{subsection:3.1})
\begin{equation}
  \label{h-estimate1}
  \abs{\wh{g}_{\pm w}(t)} \ll \abs{ \wh{I}_{[a,b]_{\pm w}}(t)\4 \wh{k}_w(t)} \ll s_{[a,b]_{\pm w}}(t)\exp\{ - \abs{t\4 w}^{1/2}\} \q \text{and} \q 
\wh{\zeta}\in L^1(\dif v),
\end{equation}
where\index{S@ $s_{[a,b]_{\pm w}}(t) := \abs{(2\pi t)^{-1}\4\sin(\pi t\4(b-a \pm 2w))}$} 
\begin{equation}
 \label{h-estimate2}
  s_{[a,b]_{\pm w}}(t) \defi \abs{(2\pi t)^{-1}\4\sin(\pi t\4(b-a \pm 2w))},
\end{equation}
we may express the weight functions $g_{\pm w}$ and $\zeta$ by their Fourier transforms
\begin{equation*}
  \wh{g}_{\pm w}(v)= \int_{\R} g_{\pm w} (x) \exp\{ - 2 \pi \iu \4 t\4 x\} \, \dif x \q \text{and} \q \wh{\zeta}(v)= \int_{\R^d} \zeta(x) 
\exp \{- 2 \pi  \iu \4 \langle v, x\rangle \} \, \dif x.
\end{equation*}
This yields
\begin{align}
  \label{g-inversion}
  g_{\pm w} (Q[x]) & = \int_\R \wh{g}_{\pm w}(t) \exp\{ 2 \pi \iu \4 t\4 Q[x]\} \, \dif t, \\
  \label{zeta-inversion}
  \zeta(x) & = \int_{\R^d} \wh{\zeta}(v)  \exp\{ 2 \pi \iu \langle x,v\rangle\} \, \dif v.
\end{align}
Using \eqref{g-inversion} we obtain by interchanging summation and integration in \eqref{smooth-remainder}
\begin{equation}
  \label{Q-smooth}
  R(g^Q_{\pm w} \4 \specialv_r) = \int_\R R(e_{tQ}\4 \specialv_r) \4 \wh{g}_{\pm w} (t) \, \dif t
\end{equation}
with $e_{tQ}(x) \tdefi \exp\{ 2 \pi \iu \4 t \4 Q[x]\}$.  (Here $R(e_{tQ}\4 \specialv_r)$ denotes the inner integral with respect to the variable $v$.) In the same way, writing $\tilde{e}_{v,r}(x) \tdefi \exp\{ - Q_{+}[x/r] + 2 \pi \iu \4\langle x,v \4 r^{-1} \rangle\}$, we derive by \eqref{zeta-inversion} the remainder
\begin{equation}
  \label{v-Fourier}
  R(e_{tQ} \4 \specialv_r) = \int_{\R^d} R(e_{tQ} \4 \tilde{e}_{v,r}) \4 \wh{\zeta}(v) \, \dif v.
\end{equation}
The sum $R(e_{tQ}\4 \tilde{e}_{v,r}) $ is the remainder between the generalized theta series\index{T@ $\theta_v(t)$, Theta series for $Q$} and its 
corresponding theta integral\index{T@ $\vartheta_v(t)$, Theta integral for $Q$}, that is $R(e_{tQ}\4 \tilde{e}_{v,r})= \theta_v(z)- {\vartheta}_{v}(t)$, where
\begin{equation}
  \label{thetadef}
  \theta_v(t) \defi  \sum_{x\in\Z^d} \exp\left\{Q_{r,v}(t,x)\right\} \q \text{and} \q \vartheta_{v}(t) \defi\int_{\R^d} \exp\left\{ Q_{r,v}(t,x)\right\}\4 \dif x,
\end{equation}
\index{Q@ $Q_{r,v}(t,x) := 2 \pi \iu\4 t\4 Q[x] - r^{-2} \4 Q_{+}[x] + 2 \pi \iu \4\langle x,v\4 r^{-1}\rangle$}
\begin{equation}
  \label{Qrv-def}
  Q_{r,v}(t,x) \defi 2 \pi \iu\4 t\4 Q[x] - r^{-2} \4 Q_{+}[x] + 2 \pi \iu \4\langle x,v\4 r^{-1}\rangle.
\end{equation}
Let us note that both $\vartheta_t(v)$ as well as $\theta_t(v)$ depend on the dilating variable $r$. However, we shall suppress this underlying dependency in order to reduce the notational burden. 
For $\abs{t} \le \imbound$ we shall use following representations of $ R(e_{tQ}\4 \tilde{e}_{v,r})= \theta_v(z)- {\vartheta}_{v}(t)$ in \eqref{Q-smooth} by means of Poisson's formula (see \cite{bochner:1948}, \S 46), which obviously applies here:
\begin{equation}
  \label{poisson1}
  \theta_v(t)- {\vartheta}_{v}(t)=\sum_{m \in \Z^d \setminus \{0\}} \vartheta_{v-r\4m}(t).
\end{equation}
Note that by definition \eqref{thetadef} the Fourier transform of $ x\mapsto \exp\{Q_{r,v}(t,x)\}$ at $u \in \R^d$ is given by $\vartheta_{v-r \4 u}(t)$, where
\begin{equation}
  \label{qsdef}
  \exp\{Q_{r,v}(t,x)\}=\exp\{-\tilde{Q}_t[x] + 2 \pi \iu \langle x, v\4 
r^{-1}\rangle\} \q \text{and} \q \tilde{Q}_t \defi r^{-2}Q_{+} - 2 \pi \iu \4 t\4 Q.
\end{equation}
In view of \eqref{v-Fourier} and \eqref{poisson1} we have
\begin{equation}
  \label{rep1}
  R(e_{tQ} \4 \specialv_r) = \int_{\R^d} \4 \Big(\sum_{m \in \Z^d \setminus \{0\}}\vartheta_{v- r\4m}(t)\Big)\4 \wh{\zeta}(v) \4 \dif v.
\end{equation}
From here we only consider the weight $g_w$. The same inequalities hold 
also for $g_w$ replaced with $g_{-w}$. Next, we decompose the integral over $t$ in \eqref{Q-smooth} into the segments $J_0 \tdefi [ - \imbound, \imbound] $\index{J@ $J_0 := [ - \imbound, \imbound] $} and $J_1 \tdefi \R \setminus J_0$\index{J@ $J_1 := \R \setminus J_0$} and obtain
\begin{align}
  \label{eq-splitting}
  \abs{R(g^Q_w \4 \specialv_r)}  \ll_d \ I_\Delta + I_\vartheta + I_\theta,\q
\end{align}
where,\index{I@ $I_\Delta :=  \abs{\int_{J_0} R(e_{tQ} \4 \specialv_r) 
\, \wh{g}_w(t) \, \dif t}$}\index{I@ $I_\vartheta  :=  \abs{\int _{J_1} 
\wh{g}_w(t) \int_{\R^d} {\vartheta}_{v}(t) \, \wh{\zeta}(v)\4 \dif v \, \dif t}$}\index{I@ $I_\theta  :=  \abs{ \int_{J_1} \wh{g}_w(t)\int_{\R^d} \theta_v(t) \, \wh{\zeta}(v) \4 \dif v \, \dif t}$} 
\begin{align}
  \label{I0}
  I_\Delta & \ \defi \ \Bigl\lvert \, \int_{J_0} R(e_{tQ} \4 \specialv_r) \, \wh{g}_w(t) \, \dif t \Bigr\rvert, \\
  \label{I0*}
  I_\vartheta & \ \defi \ \Bigl\lvert \, \int _{J_1} \wh{g}_w(t) \int_{\R^d} {\vartheta}_{v}(t) \, \wh{\zeta}(v)\4 \dif v \, \dif t\Bigr\rvert, \\
  \label{I1}
  I_\theta & \ \defi \ \Bigl\lvert \, \int_{J_1} \wh{g}_w(t) \int_{\R^d} \theta_v(t) \, \wh{\zeta}(v) \4 \dif v \, \dif t\Bigr\rvert.
\end{align}
We start with the integral over the sections $J_1$. In the term $I_\theta$ we separate the $t$ and $v$ integrals via
\begin{equation}
  \label{bound3.43}
  I_\theta \ll_d \norm{\wh{\zeta}}_1  \sup_{v \in \R^d} \int_{ \abs{t} >\imbound} \abs{\wh{g}_w(t) \4 \theta_v(t)} \, \dif t,
\end{equation}
where the estimation of the latter integral will be done in the Sections \ref{geom-num}--\ref{conclus}. In order to estimate the terms $I_\Delta$ and $I_\vartheta$ we need to estimate  $\abs{\vartheta_v(t)}$ first:

\subsubsection{Estimates for $\abs{\vartheta_v(t)}$}
For any symmetric complex $d\times d$-matrix $\Xi$, whose imaginary part is positive definite, we have
\begin{equation}
  \label{defthetaint}
  \int_{\R^d}\exp\bigl\{\pi \iu \4 \Xi [x]+2\pi \iu \langle x,v\rangle\bigr\} \, \dif x
  =\left(\det\left(\Xi/ \iu \right)\right)^{-1/2}\4 \exp\left\{ -\pi \iu \4 \Xi^{-1}[v]\right\},
\end{equation}
where we choose the branch of the square root which takes positive values 
on purely imaginary $\Xi$, $v\in \R^d$ and $\Xi^{-1}\bigl[ x\bigr]$ denotes the quadratic form $\langle\Xi^{-1}x,x\rangle$, defined by the inverse 
operator $\Xi^{-1} \colon \C^d\rightarrow\C^d$ whose imaginary part is negative definite (see \ \cite{mumford:1983}, p.\ 195, Lemma 5.8 and (5.6)). We shall apply \eqref{defthetaint} in the case $\Xi_t \tdefi \iu \pi^{-1} \tilde{Q}_t = 2tQ+\iu \pi^{-1} r^{-2}Q_+$ in order to obtain the following expression for $\vartheta_v$ in \eqref{thetadef} (see also \eqref{qsdef})
\begin{equation}
  \label{fourier-theta}
  \vartheta_{v}(t)  = \int_{\R^d} \hspace{-0.5mm} \exp\big\{\pi \iu \4 \Xi_t[x] + 2 \pi \iu \big\langle x, v/r \rangle \big\} \4 \dif x
  = (\det(\Xi_t / \iu))^{-\frac{1}{2}} \exp \{  - \pi \iu \4\Xi_t^{-1} [v/r] \}.
\end{equation}
Hence, the Fourier transform of $x \mapsto \exp\{Q_{r,v}(t,x)]\}$ takes the following shape
\begin{equation}
  \label{fourier-def}
  \det \big(\pi^{-1} \tilde{Q}_t \big)^{-1/2}\4 \exp \left\{ -\pi^2 \tilde{Q}_t^{-1}[u-v/r] \right\}=\vartheta_{v- r\4 u}(t)=\vartheta_{r\4 u -v}(t).
\end{equation}
A short calculation shows that $\tilde{Q}_t^{-1}=(4 \pi^2 t^2 +r^{-4})^{-1}(2 \pi \iu t Q^{-1}+r^{-2}Q_+^{-1})$ and it follows immediately that
\begin{equation}
  \label{detqc} \textstyle
  \det \tilde{Q}_t^{-1}  = (4 \pi^2 t^2 +r^{-4})^{-d} \prod_{i=1} ^d (2 \pi \iu t q_i^{-1} +r^{-2}\abs{q_i}^{-1}).
\end{equation}
Taking the absolute value of \eqref{fourier-theta} and \eqref{detqc} we conclude that
\begin{equation}
  \label{vartheta-bound}
  \abs{\vartheta_{u\4 r}(t)} \ll_d d_Q\4 r^{d/2}r_t ^{d/2} \exp \Big\{- \pi^2  r_t^{2} \4 Q_+^{-1}[u] \Big\},
\end{equation}
where $r_t \tdefi r (4 \pi^2 t^2 r^4 +1)^{-1/2}$\index{R@ $r_t :=r (4 \pi^2 t^2 r^4 +1)^{-1/2}$} and $d_Q \tdefi \abs{\det Q}^{-1/2}$ as already defined in \eqref{eq:1.1}.\footnote{The first of these notations will be used throughout this section only and should not be confused with the notation $r_* \tdefi r q^{-1/2}$ which will be introduced latter in Lemma \ref{special_int}.}



\subsubsection{Estimation of $I_\vartheta$}
\label{Sub:Estimation-I_vartheta}
By \eqref{vartheta-bound} with $v=ur$ we have $\abs{\vartheta_v(t)} \ll_d d_Q\4 r^{d/2} \4 r_t^{d/2}$ and therefore we obtain by using \eqref{h-estimate1} after integrating over $v$ in \eqref{I0*}
\begin{equation}
  \label{I0*-bound}
  I_\vartheta \ll_d d_Q\4 r^{d/2} \4 \norm{\wh{\zeta}}_1\4 \int_{\abs{t}> 
\imbound}s_{[a,b]_{\omega}}(t) \exp\{ - \abs{w\4 t}^{1/2}\}  \4 r_t^{d/2}\4 \dif t.
\end{equation}
If $\abs{b-a}^{-1} \leq \imbound$, then we use $s_{[a,b]_w}(t) \le \abs{t}^{-1}$ and $r_t \leq (rt)^{-1}$ to get the bound
\begin{equation*}
  \int_{\imbound}^{\infty} s_{[a,b]_w}(t) \4 r_t^{d/2} \, \dif t  \le r^{-d/2} \int_{\imbound}^{\infty} t^{-d/2-1} \, \dif t  \ll_d q_0^{d/4}.
\end{equation*}
In the case $\abs{b-a}^{-1} > \imbound$ we shall estimate the $t$-integral in \eqref{I0*-bound} by means of $s_{[a,b]_w}(t) \leq \abs{b-a+2w}/2$. Using $\abs{w} < (b-a)/4$ additionally leads to
\begin{equation*}
  \int_{\abs{t}> \imbound}s_{[a,b]_{\omega}}(t) \4 r_t^{d/2} \4 \dif t \leq r^{-d/2} \abs{b-a+2w} \int_{\imbound}^{\infty} t^{-d/2} \, \dif t \ll_d 
\frac{\abs{b-a}}{q_0^{1/2} \4 r} \4 q_0^{d/4}.
\end{equation*}
Summarizing, we have established the bound
\begin{equation}
  \label{Istar0}
  I_\vartheta \ll_d  d_Q \4\norm{\wh{\zeta}}_1\4 \min\{\abs{b-a} \imbound 
,1\}\4 r^{d/2} \4 q_0^{d/4},
\end{equation}
provided that $d>2$.



\subsubsection{Estimation of $I_\Delta$}
According to \eqref{I0}, \eqref{v-Fourier} and \eqref{poisson1} we may write
\begin{equation}
  \label{I01}
  \begin{aligned}
       & I_\Delta = \bigg\lvert \int_{J_0}  \wh{g}_w(t)\4 R(e_{tQ} \4 \specialv_r) \, \dif t \, \bigg\rvert, \q &                                          
                              & \text{where} \\
       & R(e_{tQ} \4 \specialv_r) = \int_{\R^d} \4 S_{t,v}\4 \wh{\zeta}(v)\4 \dif v,
    \q &                                                                  
                           & S_{t,v}\defi \sum_{m \in \Z^d \setminus \{0\}} \vartheta_{v- r\4m}(t).
  \end{aligned}
\end{equation}
In order to use the estimate \eqref{vartheta-bound} let $v \in \R^d$ and write $v = ru$ with $u = u_0 + m_u$, where $u_0 \in [-1/2,1/2]^d$ and 
$m_u \in \Z^d$. Then
\begin{equation}
  \abs{S_{t,v}} \leq \sum_{m \neq m_u} \abs{\vartheta_{r(u_0+m)}(t)}\ll d_Q r^{d/2}r_t^{d/2} \sum_{m\neq m_u} \exp\{ -\pi^2 r_t^2 Q_+^{-1}[u_0+m] \}.
\end{equation}
Note that $\norm{m+u_0} \geq \norm{m+u_0}_\infty \geq \frac{1}{2}$ for any $m \in \Z^d \setminus \{0\}$ and therefore $\frac{\pi^2}{2}Q_+^{-1}[u_0 
+m] \geq \frac{\pi^2}{8}q^{-1} \geq q^{-1}$ which yields the bound
\begin{equation}
  \label{sum-bound1}
  \abs{S_{t,v}} \ll d_Q r^{d/2}\abs{r_t}^{d/2}\Big(\e^{-\pi^2 r_t^2 Q_+ ^{-1}[u_0]} I_r(v) + \e^{-r_t^2 / q} K_{u_0} \Big),
\end{equation}
where $I_r(v) \tdefi I_{[r/2,\infty)}(\norm{v}_{\infty})$ and $K_{u_0} \tdefi \sum_{m \in \Z^d} \exp \{-\frac{\pi^2}{2}r_t^2 Q_+ ^{-1}[m+u_0]\}$. The sum $K_{u_0}$ may be estimated by an integral as follows: Since the map  $t\mapsto r_t^2 = r (4 \pi^2 t^2 r^4 +1)^{-1/2}$ is strictly monotone increasing on $t < 0$ and decreasing on $t >0$, we find that $r_t^2 \ge q_0/(4 \pi^2+1)$ for $\abs{t} \le \imbound$ as $r \geq q^\frac{1}{2}$ and thus $\exp\{- \pi^2 r_t^2 Q_+^{-1}[u]\} \le \exp\{- \frac{q_0}{5}  
Q_{+}^{-1}[u]\}$. Let $I \tdefi [-\frac{1}{2} , \frac{1}{2}]^d$ and note that $ Q_+^{-1}[x] \leq \tfrac{d}{4q_0}$ for $x\in I$, from which we deduce that
\begin{equation*}
  k_u \defi \int_I \exp \{-\tfrac{q_0}{5} Q_+^{-1}[u+x] \} \, \dif x \gg_d \exp\{-\tfrac{q_0}{5}  Q_+^{-1}[u] \} \4 \int_I \exp\{- \tfrac{2 q_0}{5} \langle Q_+^{-1} u ,x \rangle \} \, \dif x,
\end{equation*}
where the integral on the right-hand side is at least one by Jensen's inequality. Hence
\begin{equation}
  \label{integ1}
  K_{u_0} \le \sum_{m \in \Z^d} \e^{- \frac{q_0}{5} \4 Q_{+}^{-1}[m + u_0]} \ll_d \sum_{m \in \Z^d} k_{m + u_0} = \int_{\R^d} \e^{ - \frac{q_0}{5} \4 Q_{+}^{-1}[x]} \, \dif x \ll_d  \Big(\frac{q}{q_0}\Big)^{\frac{d}{2}}.
\end{equation}
Using \eqref{I01} together with \eqref{sum-bound1} and \eqref{integ1}, we 
may now estimate $I_\Delta$ by the following integrals. Writing $v_0 = v- r m$, $\norm{v_0}_{\infty} \leq \frac{r}{2}$, $m \in \Z^d$, we have
\begin{equation}
  \label{theta2}
  I_\Delta \ll_d d_Q \int_{J_0} \abs{\wh{g}_w(t)}\4 \big(\Theta_{t,1}+ \Theta_{t,2}\big) \, \dif t,
\end{equation}
where
\begin{align*}
  \Theta_{t,1} & \defi \Big(\frac{q}{q_0}\Big)^{d/2} \4 r^{d/2} r_t^{d/2} 
\e^{-\tfrac{r_t^2} q} \int_{\R^d} \abs{ \wh{\zeta}(v)} \, \dif v, \\
  \Theta_{t,2} & \defi r^{d/2} r_t^{d/2} \int_{\norm{v}_{\infty} > \4 r /2} \exp\{- \pi^2 r_t^2 Q_{+}^{-1}[v_0 r^{-1}]\}\4\abs{\wh{\zeta}(v)}\4 \dif v.
\end{align*}
If we write $h(s;x) \tdefi s^{d/4} \e^{-s \4 x}$ with $s,x>0$, then the maximum of $s \mapsto h(s;x)$ is attained at $s_0=d/(4\4x)$. Hence, $\max_{t\in J_0} h(r_t^2;x)\ll_d \min(x^{-d/4}, r^{d/2})\ll_d (x+ \tfrac 1 {r^2})^{-d/4}$. Thus, we obtain with $x=1/q$
\begin{equation}
  \label{Theta1-bound}
  \max_{t\in J_0} \, \Theta_{t,1} \ll_d (q/q_0)^{d/2} \4 r^{d/2}\4q^{d/4}\4\norm{\wh \zeta}_1.
\end{equation}
Note that the value $x=1/q$ is within the range of $t \mapsto r_t^2$, $t \in J_0$, since its maximum is $r_0^2=r^2$ and its minimum is  $q_0/(4 \pi^2+1) \le r_{t^*}^2 \le q_0$, where $t^*=\pm \imbound$. In order to estimate $\Theta_{t,2}$, we choose $x=Q_{+}^{-1}[v_0/r]/4$ and get
\begin{equation}
  \label{Theta2-bound}
  \sup_{t\in J_0} \Theta_{t,2} \ll_d r^{d/2} \int_{\norm{v}_{\infty} > r/2 } \frac{\abs{\wh{\zeta}(v)}}{(r^{-2} +Q_+^{-1}[v_0/r] )^{d/4}} \, \dif v.
\end{equation}
Now we integrate the bounds \eqref{Theta1-bound} and \eqref{Theta2-bound} 
in $t \in J_0$ weighted with $\abs{\wh{g}_w(t)}$: In view of \eqref{h-estimate1} we have $\int_{J_0} \abs{\wh{g}_w(t)} \, \dif t \ll \log(1+\abs{b-a} \4 \imbound )$ and thus we finally get, using the quantity $\norm{\wh{\zeta}}_{*,r}$ as defined in \eqref{defzeta*} for the weights $\zeta(x)$, the estimate
\begin{equation}
  \label{eq:I0estimate1}
  I_{\Delta} \ll_d d_Q \4 r^{d/2} \log (1+ \abs{b-a} \4 \imbound ) \4 \norm{\wh{\zeta}}_{*,r}.
\end{equation}
Applying \eqref{zeta-boundary} of Corollary \ref{co1} with \eqref{eq-splitting}, \eqref{Istar0} and \eqref{eq:I0estimate1} we may now collect the results obtained so far as follows for the lattice point remainder of \eqref{smooth-remainder}. We have\index{Z@ $\norm{\wh \zeta}_{*,r} := q^{d/4}\big((\frac{q}{q_0})^{d/2} \norm{\wh{\zeta}}_1 + \int_{\norm{v}_{\infty} > r/2} \frac{\abs{\wh \zeta(v)}} {({q^{1/2}} r^{-1} +\norm{ v r^{-1} }_{\Z})^{d/2}} \, \dif v\big)$}
\begin{equation}
  \label{crucialbound}
  \begin{aligned}
    \Big\lvert \sum_{m\in\Z^d} & I_{[a,b]}(Q[m]) \specialv_r (m) - \int_{\R^d} 
I_{[a,b]}(Q[x]) \specialv_r (x) \, \dif x \Big \rvert \\
    & \ll_d I_\theta + d_Q \4  r^{d/2} \4 \norm{\wh{\zeta}}_{*,r} \4 \log 
(1+ \abs{b-a} \4 \imbound)+  w \4 \norm{\specialv}_Q \4 r^{d-2}.
  \end{aligned}
\end{equation}



\subsubsection{Estimation of $I_\theta$}
We shall now estimate the crucial error term $I_\theta$, see \eqref{I1} and \eqref{bound3.43}. At first we shall bound the theta series $\theta_v(t)$ uniformly in $v$ by another theta series in dimension $2\4d$ in order 
to transform the problem to averages over functions on the space of lattices subject to an appropriate action of $\SL(2,\R)$. We have

\begin{lemma}
  \label{thetaestimate2}
  Let $\theta_v(t)$ denote the theta function in \eqref{thetadef} depending on $Q$, $r \in \R$ and $v\in \R^d$. For $r \geq 1$, $t \in \R$ the following bound holds uniformly in $v \in \R^d$\index{H@ $H_{t}(m,n) := r^2 \4 Q_+^{-1} [m - 4 \4 t\4 Q n ] + r^{-2}\4 Q_+[\4n\4]$}\index{P@ $\psi (r,t) := \sum_{m,n \in \Z^d} \exp\{- H_t(m,n)\}$}
  \begin{alignat}{3}
    \label{eq:thetaestimate1}
    \bigl\lvert \theta_v (t) \bigr\rvert & \q \ll_d \q &  & (\det{Q_+})^{-1/4} \4 r^{d/2} \4 \psi (r,t)^{1/2},           &  & \text{where} \\
    \label{eq:thetaestimate2}
    \psi (r,t) & \q \defi \q &  & \sum_{m,n \in \Z^d} \exp\{- H_t(m,n)\}, 
&  &  \\
    \label{eq:thetaestimate3}
    H_{t}(m,n)                   & \q \defi \q &  & r^2 \4 Q_+^{-1} [m - 4 \4 t\4 Q n ] + r^{-2}\4 Q_+[\4n\4],
  \end{alignat}
  and $H_t(m,n)$ is a positive quadratic form on $\Z^{2d}$. Note that $H_t(m,n)$ depends as well on the currently fixed dilating variable $r$ which we suppress here.
\end{lemma}

\begin{proof}
  For any $x,y \in \R^d$ the equalities
  \begin{equation}
    \label{eq:quad1}
    \begin{aligned}
      2 \4 \left(Q_{+}[\4x\4] + Q_{+}[\4y\4]\right)    & \ = \ Q_{+}[\4x + y\4] + Q_{+}[\4x -y\4], \\
      \left\langle Q\4 (x + y), \,x - y \right\rangle & \ = \ Q[\4x\4] - Q[\4y\4]
    \end{aligned}
  \end{equation}
  hold. Rearranging $\theta_v(z)\, \overline{\theta_v(z)}$ and using \eqref{eq:quad1}, we would like to use $m+n$ and $m-n$ as new summation variables on a lattice. But both vectors have the same parity, that is $m + n \equiv m - n \mod 2$. Since they are dependent one has to consider the $2^d$ affine sublattices indexed by \,$ \alpha = (\alpha_1, \dots, \alpha_d)$ \,with \,$\alpha_j \in \{0,1\}$ for $1 \leq j \leq d$:
  \begin{equation*}
    \Z^d_{\alpha} \defi \{ m \in \Z^d \,:\, m \equiv \alpha \mod 2 \},
  \end{equation*}
  where, for $ m = (m_1, \dots , m_d)$, $m \equiv \alpha \mod 2$ means $m_j \equiv \alpha_j \mod 2$ for all $ 1\leq j \leq d$. Thus writing
  \begin{equation*}
    \theta_{v,\alpha}(t) \defi \sum_{m \in \Z^d_{\alpha}} \exp\left[-\frac{1}{r^2}Q_+[m] - 2 \pi \iu \4 t \4 Q[m] + 2 \pi \iu \4\langle m,\frac{v}{r}\rangle\right],
  \end{equation*}
  we obtain $\theta_v(t) = \sum_{\alpha} \theta_{v,\alpha}(t)$ and hence by the Cauchy-Schwarz inequality
  \begin{equation}
    \label{eq:cauchy}
    \textstyle
    \bigl\lvert \theta_v(t) \bigr\rvert^2 \leq \ 2^d \sum_{\alpha \in \{0,1\}^d }\ \bigl\lvert \theta_{v,\alpha}(t) \bigr\rvert^2.
  \end{equation}
  Using \eqref{eq:quad1} and the absolute convergence of \,$\theta_{\alpha}(t)$, we can write
  \begin{align*}
    \abs{\theta_{v,\alpha}(t)}^2 & = \! \sum_{m, n \in \Z^d_{\alpha}} \exp \! \left[-\frac{1}{r^2} \bigl(Q_{+}[m]+Q_{+}[n]\bigr)- 2 \pi \iu t\4\bigl(Q[m]-Q[n]\bigr) - 2 \pi \iu \4\langle m-n,\frac{v}{r}\rangle\right] \\
    & = \! \sum_{m, n \in \Z^d_{\alpha}} \exp \! \left[-\frac{2}{r^2} \bigl(Q_{+}[\bar{m}]+Q_{+}[\bar{n}]\bigr)- 4 \pi \iu \4 \bigl\langle 2\4t\4Q\bar{m}+\frac{v}{r},\bar{n}\bigr\rangle\right]
  \end{align*}
  where $\bar{m} = \frac{m+n}{2}$, $\bar{n} = \frac{m - n}{2}$. Note that the map
  \begin{equation*}
    \textstyle
    \bigcup_{\alpha \in \{0,1\}^d } \Z^d_{\alpha} \times \Z^d_{\alpha} \longrightarrow \Z^d \times \Z^d, \ \ (m, n) \longmapsto \Big(\ffrac{m + {n}}{2}, \ffrac{m-n}{2}\Big)
  \end{equation*}
  is a bijection. Therefore we get by \eqref{eq:cauchy}
  \begin{equation}
    \label{eq:double1}
    \begin{aligned}
      \bigl\lvert \theta_v(t) \bigr\rvert^2 & \ll_d \sum_{\alpha \in\{0,1\}^d} \sum_{m, n \in \Z^d_{\alpha}} \exp\left[ -\frac{2}{r^2}\bigl( Q_{+}[\bar{m}]+Q_{+}[\bar{n}]\bigr)-4 \iu \pi \4\bigl\langle 2t\4Q\bar{m}+\frac{v}{r},\bar{n}\bigr\rangle \right] \\
                                 & = \sum_{\bar{m}, \bar{n} \in \Z^d} \exp\left[-\frac{2}{r^{2}}\bigl(Q_{+}[\bar{m}]+Q_{+}[\bar{n}]\bigr)-4 \iu \pi \4\bigl\langle 2t\4Q\bar{m}+\frac{v}{r},\bar{n}\bigr\rangle\right].
    \end{aligned}
  \end{equation}
  In this double sum fix $\bar{n}$ and sum over $\bar{m} \in \Z^d$ first, and call the inner sum $\theta_v(t, \bar{n})$. Using \eqref{defthetaint} with $\Xi = 2\iu Q_+ r^{-2}/\pi$ and $v = -4 \4 t \4 Q \4 \bar{n} +m$, we get for $\delta \tdefi \left( \det \left(\frac{2}{\pi r^2} \4 Q_{+}\right)\right)^{-1/2}$ by the symmetry of $Q$ and Poisson's formula (see \cite{bochner:1948}, \S 46)
  \begin{alignat*}{3}
    \theta_v(t, \bar{n}) \  & \defi
    &                                     
 &  & 
 & \sum_{\bar{m} \in \Z^d } \exp\left[-\frac{2}{r^2} \bigl(Q_+[\bar{m}]+Q_+[\bar{n}]\bigr)- 4 \pi \iu \4 \bigl\langle 2t \4 Q\bar{m} +\frac{v}{r},\bar{n}\bigr\rangle\right] \\ & \ = &&
    \delta                  &       & \sum_{m \in \Z^d} \exp\left[ - \frac{ \pi^2 \4 r^2}{2}\4 Q_+^{-1}[\4 m - 4\4t\,Q \4\bar{n}\4]-\frac{2}{r^2} Q_+[\bar{n}]- 4 \pi \iu \langle\4\frac{v}{r},\bar{n}\4\rangle\right].
  \end{alignat*}
  Thus, we have uniformly in $v \in \R^d$
  \begin{equation}
    \label{eq:double2}
    \bigl\lvert \theta_v(t, \bar{n}) \bigr\rvert \leq \delta \4 \sum_{m \in \Z^d} \exp \Bigl\{ - \frac{ \pi^2 \4 r^2}{2}\4 Q_{+}^{-1}[ m- 4 \4 t\4Q\4 \bar{n}\4]-\frac{2}{r^2}Q_{+}[\bar n ]\Bigr\}.
  \end{equation}
  Hence we obtain by \eqref{eq:double1} and \eqref{eq:double2}
  \begin{equation*}
    \bigl\lvert \theta_v(t)\bigr\rvert^2 \ll_d (\det\4 Q_{+})^{-1/2} \4 r^d \sum_{m, n \in \Z^d} \4 \exp \{ -G_t(m,n) \},
  \end{equation*}
  where $G_t(m,n) \tdefi \frac{\pi^2 r^2}{2} Q_{+}^{-1}[\4 m - 4 \4 t \4 Q\4 n \4] +\frac{2}{r^2} Q_{+}[n]$. Since $\pi^2/2 >1 $ we may bound $G_t(m,n)$ from below as follows:
  \begin{equation*}
    G_t(m,n)\ge r^2 Q_+^{-1} [m- 4 t Q n] + r^{-2} Q_+[n] = H_t(m,n)
  \end{equation*}
  which proves the claimed estimate \eqref{eq:thetaestimate1}. Finally, observe that we can write
  \begin{equation*}
  H_t(m,n) = \big \|\begin{pmatrix} r Q_+^{-\frac{1}{2}}(m-4t Qn) \\ r^{-1} Q_+ n \end{pmatrix} \big \|^2,
  \end{equation*}
  which shows that $H_t(m,n)$ is a positive definite quadratic form on $\Z^{2d}$.
\end{proof}

In view of Lemma \ref{thetaestimate2} we can introduce the $2d$-dimensional lattice\index{L@ $\Lambda_t := D_{rQ} \4 U_{4tQ} \Z^{2d}$, lattice on $\R^{2d}$}
\begin{equation}
  \label{df:lambda_t}
  \Lambda_t \defi D_{rQ} \4 U_{4tQ} \Z^{2d},
\end{equation}
where\index{D@ $D_{rQ}$, diagonalizable matrix on $\R^{2d}$}\index{U@ $U_{4tQ}$, unipotent matrix on $\R^{2d}$}
\begin{align}
  \label{df:action}
  D_{rQ} = \begin{pmatrix} r Q_+ ^{-\frac{1}{2}} & \\ & r^{-1} Q_+ ^{\frac{1}{2}} \end{pmatrix} \q \ \ \text{and} \q \ \ U_{4tQ}=
  \begin{pmatrix} \mathbbm 1_d & -4t Q \\ & \mathbbm 1_d \end{pmatrix},
\end{align}
in order to write $\psi(r,t)= \sum_{v \in \Lambda_t} \exp\{ -\norm{v}^2\}$ as the Siegel transform of $\exp \{-\norm{x}^2\}$ evaluated at the lattice $\Lambda_t$.  According to the Lipschitz principle in the Geometry of Numbers (see \cite{schmidt:1968}, Lemma 2, or \cite{eskin-margulis-mozes:1998}, Lemma 3.1) one can show that $\psi(r,t) \ll_d  \alpha(\Lambda_{t})$, where $\alpha$ is the maximum over all $\alpha_l$-characteristics 
(see \eqref{alp0}). However, we choose to follow a more direct and transparent argument for the sake of clarity and motivate the relation between the $\alpha_i$-characteristics and the successive minima of a lattice for the convenience of the reader. The following Lemma \ref{Dav1} (with $\eps=1$) reduces the problem of estimating the theta series \eqref{eq:thetaestimate2} to the problem of counting lattice points as follows
\begin{equation}
  \label{eq:further_estimate_theta}
  \psi(r,t) \asymp_d \# \{w \in \Lambda_t : \norm{w}_\infty \le 1\} \ll_d 
\# \{ w \in \Lambda_t : \norm{w} \le d^{1/2} \}.
\end{equation}

\begin{lemma}
  \label{Dav1}
  Let $\Lambda$ be a lattice in $\R^d$. Assume that $0<\eps\le1$, then
  \begin{equation}
    \label{eq:Hequiv}
    \exp\{ -d\eps\} \4 \# \mathcal H \,\le \, \sum_{v \in \Lambda} \exp \bgl \{- \eps\,\norm v^2 \bgr\} \ll_d \eps^{-d/2}\,\# \mathcal H,
  \end{equation}
  where $\mathcal H \tdefi \bgl\{v \in \Lambda \, : \, \norm v_\infty< 1 \bgr\}$\index{H@ $\mathcal H := \bgl\{v \in \Lambda \, : \, \norm v_\infty< 1 \bgr\}$}.
\end{lemma}

\begin{proof}
  The lower bound for the sum is obvious by restricting summation to the set of elements in $\mathcal H$. As for the upper bound introduce for \,$\mu = (\mu_1, \dots, \mu_d)\in \Z^d$ \,the sets
  \begin{equation*}
    B_\mu \defi \left[\,\mu_1-\ffrac{1}{2},\, \mu_1 + \ffrac{1}{2}\right)\times\cdots\times \left[\,\mu_{d}-\ffrac{1}{2},\,\mu_d + \ffrac{1}{2}\right)
  \end{equation*}
  such that $\R^d = \bigcup_{\mu \in \Z^d} B_\mu$. For any fixed $w^* \in \mathcal H_\mu \tdefi \Lambda \cap B_\mu $ we have $w - w^*\in \mathcal H$ for all $w\in \mathcal  H_\mu$. Hence we conclude for any $\mu\in \Z^d$
  \begin{equation*}
    \# \mathcal  H_\mu \, \leq \, \# \mathcal H.
  \end{equation*}
  Since $x \in B_\mu$ implies $\norm{x}_\infty \geq \norm{\mu }_{\infty}/2$, we obtain
  \begin{align*}
    \sum_{v \in \Lambda} \e^{-\eps\,\norm{v}^2 } \le \sum_{v \in \Lambda} 
\e^{- \eps\, \norm{v}_\infty^2} & \le  \sum_{\mu \in \Z^d} \, \sum_{v \in 
\Lambda \cap B_\mu } \e^{- \frac{\eps}{4} \,\norm{\mu}^2_{\infty}} \le \# 
\mathcal  H \sum_{\mu \in \Z^d} \e^{- \frac{\eps}{4} \,\norm {\mu}^2} \ll_d \eps^{-d/2}\,\# \mathcal  H.
  \end{align*}
  This concludes the proof of Lemma \ref{Dav1}.
\end{proof}



\section{Functions on the Space of Lattices and Geometry of Numbers}
\label{geom-num}
\noindent Let $n \in \N^+$ be fixed (later to be chosen as $n=2d$) and for every integer $l$ with $1\leq l\leq n$ we fix a quasinorm $\vert\cdot\vert_l$ on the exterior product $\largewedge^l\R^n$.  Let $L$ be a subspace of $\R^n$ and $\Delta$ a lattice in $L$ (i.e.\ $\Delta$ is a free $\Z$-module of full rank $\dim L$), then any two bases of $\Delta$ are related by a unimodular transformation, 
that is, if $u_1, \dots, u_l$ and $v_1,\dots, v_l$ are two bases of $\Delta$, where $l = \dim L$, then $v_1 \wedge \dots \wedge v_l = \pm u_1 \wedge \dots \wedge u_l$, which implies that the expression $\abs{v_1 \wedge \dots \wedge v_l}_l$ is independent of the choice of basis. \par
Let $\Delta$ be a lattice in $\R^n$, we say that a subspace $L$ of $\R^n$
is $\Delta$-\textit{rational} if $L\cap\Delta$ is a lattice in $L$. For any $\Delta$-rational subspace $L$, we denote by $d_\Delta (L)$, or simply 
by $d(L)$\index{D@ $d_\Delta (L)=d(L)$, covolume of the $\Delta$-rational subspace $L$}, the quasinorm  $\vert u_1\wedge\ldots\wedge u_l\vert_l$ 
where $\{ u_1,\ldots ,u_l\}$, $l =\dim L$, is a basis of $L\cap\Delta$ over $\Z$. For $L = \{ 0\}$ we write $d(L) \tdefi 1$. If the quasinorms 
$\vert\cdot\vert_{l}$ are the norms on $\largewedge^l \R^n$ induced from the standard Euclidean norm on $\R^n$, then $d(L)$ is equal to the determinant (or discriminant) $\det(L\cap \Delta)$ of the lattice $L \cap \Delta$, that is the volume of $L/(L\cap\Delta)$. In particular, in this case the lattice $\Delta$ is said to be unimodular if and only if $d_\Delta (\R^n) = 1$. Also in this case $d(L)d(M)\geq d(L\cap M)d(L+M)$ for any two $\Delta$-rational subspaces $L$ and $M$  (see Lemma 5.6 in \cite{eskin-margulis-mozes:1998}), but any two quasinorms on $\largewedge^l\R^n$ are equivalent, which proves

\begin{lemma}
  \label{Lemma 9.6}
  There is a constant $C\geq 1$ depending only on the quasinorm $\vert \,\cdot \,\vert_l$ and not on $\Delta$ such that
  \begin{equation}
    \label{9.1}
    C^2 d(L)d(M)\geq d(L\cap M)d(L+M)
  \end{equation}
  for any two $\Delta$-rational subspaces $L$ and $M$.
\end{lemma}

Let us introduce the following notations for $0\le l \le n$\index{A@ $\alpha,\alpha_l$-characteristic of a lattice},
\begin{alignat}{2}
  \label{9.2}
   & \alpha_l(\Delta ) &  & \defi \sup\{ d(L)^{-1} : L \,\, \text{is a}\,\, \Delta \text{-rational subspace of dimension}\,\, l\}, \\
  \label{9.3} \textstyle
   & \alpha (\Delta )  &  & \defi \max_{0\leq l\leq n} \,\, \alpha_l (\Delta ).
\end{alignat}
This extends the earlier definition \eqref{alp0} of $\alpha_l(\Delta )$ in the introduction of Section \ref{section:2} to the case of general seminorms on $\largewedge^l \R^n$. In this section the functions $\alpha_l$ and $\alpha$ will be based on standard Euclidean norms, that is, we have $d(L)=\det(L \cap \Delta)$.\par
In the following we shall use some facts from the Geometry of Numbers and 
the classical reduction theory for lattices in $\R^n$, see Davenport (1958, \cite{davenport:1958c}), Cassels (1959, \cite{Cassels:1959}) and  Einsiedler-Ward (\cite{einsiedler-ward:2019}).
The successive minima of a lattice $\Lambda$ are the numbers $M_1(\Lambda) \leq \dots \leq M_n(\Lambda)$\index{M@ $M_j(\Delta)$, $j$-th successive minimum of a lattice $\Delta$} defined as follows: $M_j(\Lambda)$ is the infimum of $\lambda > 0$ such that the set $\{v \in \Lambda : \norm{v} < \lambda\bgr \}$ contains $j$ linearly independent vectors and in particular $M_1(\Lambda)$ is the shortest non-zero vector of the lattice $\Lambda$.
It is easy to see that these infima are attained, that is, there exist linearly independent vectors $v_1,\dots, v_n \in \Lambda$ such that $\norm{v_j} = M_j(\Lambda)$ for all $j=1,\ldots, n$. Moreover, as a consequence of the reduction algorithm of Korkine and Zolotareff (see \cite{korkine-zolotareff:1872},\cite{korkine-zolotareff:1873}, and \cite{korkine-zolotareff:1877}) the $\alpha_l$-characteristic and the successive minima are related according to $\alpha_l(\Lambda) \asymp_d (M_1(\Lambda) \dots M_l(\Lambda))^{-1}$ (see \cite{einsiedler-ward:2019}, Chapter 1, Theorem 15).

\begin{lemma}
  \label{Dav4}
  Let $F$ be a norm in $\R^n$ and denote by $M_1 \leq \dots \leq M_n$ the 
successive minima with respect to $F$.  Let $\Lambda$ be a lattice in $\R^n$, then
  \begin{equation}
    \label{eq:alpham}
    \alpha_l(\Lambda) \asymp_n (M_1(\Lambda) \cdots M_l(\Lambda))^{-1},\q 
l=1,\ldots,n.
  \end{equation}
  Moreover, for any $\mu >0$, if $1\le j\le n$ is such that $M_j(\Lambda) 
\leq \mu <M_{j+1}(\Lambda)$, where the right-hand side is omitted if $j=n$, then
  \begin{equation}
    \label{boxcount}
    \# \{ v \in \Lambda\, :\, F(v) \le \mu \} \asymp_{n} \mu^j \, \alpha_j(\Lambda).
  \end{equation}
\end{lemma}

\begin{proof}
  First we prove the lower bound. We may assume that $M_j(\Lambda) \leq \mu < M_{j+1}(\Lambda)$, the right-hand side being omitted if $j=n$. Let $v_1,\dots, v_n$ denote the elements in $\Lambda$ corresponding to the successive minima $M_i(\Lambda)$, $i=1, \ldots , n$. For $m_1, \dots, m_j \in \Z$ with $\abs{m_i} \le j^{-1}\4 \mu \4 F(v_i)^{-1}$ notice that $v=m_1\4 v_1+ \ldots + m_j\4 v_j$ satisfies $F(v) \leq \mu$, thus
  \begin{equation}
    \label{eq:defnumber}
    N(\mu)\defi \# \{ v \in \Lambda \, :\, F(v) \le \mu \} \gg_m \mu^j (M_1(\Lambda) \cdots M_j(\Lambda))^{-1}.
  \end{equation}
  The upper bound is also proven in Davenport \cite{davenport:1958c} (see 
Lemma 1). We include the short argument here for the sake of completeness: Let $w_1,\dots, w_n$ be an integral basis of $\Lambda$ such that $v_i$ is linearly dependent on $w_1,\dots, w_i$ for any $i=1,\dots, n$. Consequently any lattice point $v \in \Lambda$ with $F(v) < M_{j+1}$ is linearly dependent on $w_1,\ldots, w_j$ and hence any element $v \in \Lambda$ with $F(v)\leq \mu$ can be written as $v= m_1\4 w_1+ \ldots + m_j\4w_j$ 
with $m_i \in \Z$. Suppose $v' \in \Lambda$ is another element with $F(v') \le \mu$ and write $v'= m_1'\4 w_1+ \ldots m_j'\4 w_j$ with $m_i' \in 
\Z$. Now define positive integers $\nu_1, \ldots, \nu_j$ by
  \begin{equation}
    \label{eq:defnu}
    2^{\nu_i-1} \le \frac {2\4 \mu } {M_i(\Lambda)} < 2^{\nu_i},
  \end{equation}
  and observe that $\nu_1 \ge \nu_2 \ge \ldots \ge \nu_j$. Assuming for the moment that $m_i \equiv m_i' \4 \mod 2^{\nu_i}$ for every $i=1,\ldots, j$ and let $i_0$ denote the largest index $i_0$ such that $m_{i_0} \ne 
m'_{i_0}$. Then $x \tdefi 2^{-\nu_{i_0}}\4 (v-v')$ is an element of $\Lambda$ and linearly independent of $w_1, \dots, w_{{i_0}-1}$. This implies $F(x) \ge M_{i_0}(\Lambda)$. On the other hand we have
  \begin{equation*}
    F(x) = 2^{-\nu_{i_0}}\4 F(v-v')\le 2^{-\nu_{i_0}}\4(F(v) + F(v'))\le 2^{-\nu_{i_0}}\4 2\4 \mu < M_{i_0}(\Lambda)
  \end{equation*}
  by \eqref{eq:defnu}. This contradiction shows that there is at most one 
lattice point in $\Delta$, implying that the coordinates $m_1, \ldots, m_j$ lie in the same residue classes modulo $2^{\nu_1}, \4 2^{\nu_2}, \dots, 2^{\nu_j}$ respectively. Hence, the number of lattice points $N(\mu)$ in \eqref{eq:defnumber} is bounded from above by the number of all residue classes, i.e.\ by $2^{\nu_1}\4 2^{\nu_2}\ldots \4 2^{\nu_j} \le (4 \4 \mu)^j (M_1(\Lambda)\ldots M_j(\Lambda))^{-1}$. This shows the upper bound in 
\eqref{boxcount}.
\end{proof}

\begin{lemma}[Davenport \cite{davenport:1958c}]
  \label{Dav3}
  Let $\Lambda = g \4 \Z^n$ and $\Lambda' = (g^{-1})^T \4 \Z^n$ denote dual lattices of rank $n$, then for all $j=1,\ldots, n$ we have
  \begin{align}
    \label{dual}
    1 \leq  M_j(\Lambda) M_{n+1-j}(\Lambda')\ll_n 1.
  \end{align}
\end{lemma}

This is a variant of Lemma 2 of Davenport \cite{davenport:1958c} for the Euclidean norm. Again, for the reader's convenience, we include the short argument here.

\begin{proof}
  Let $v_1, \ldots,v_n \in \Lambda$, resp.\ $v_1',\ldots,v_n' \in \Lambda'$, be linearly independent such that $\norm{v_i} = M_i(\Lambda)$, resp.\ $\norm{v_i'} = M_i(\Lambda')$. Then $v_1,\ldots,v_j$ cannot be orthogonal to all lattice points $v_1',\ldots,v_{n+1-j}'$, otherwise they would fail to be independent. Thus, we have $\langle v_i, v_k' \rangle \neq 
0$ for some $i =1,\ldots,j$ and $k=1,\ldots, n+1-j$, which implies that
  \begin{equation*}
    M_j(\Lambda) M_{n+1-j}(\Lambda') \geq M_i(\Lambda) M_k (\Lambda') = 
\norm{v_i} \norm{v_k'} \geq \abs{\langle v_i,v_k' \rangle} \geq 1
  \end{equation*}
  because of duality. The right-hand side of \eqref{dual} follows from \eqref{eq:alpham} with $l=n$, which is known as Minkowski's inequality. Indeed, $\det(\Lambda)= \alpha_n(\Lambda)^{-1} \asymp_n M_1(\Lambda) \dots M_n(\Lambda)$ and since $\det(\Lambda)\det(\Lambda')=1$ we conclude that
  \begin{equation*}
    \textstyle
    M_j(\Lambda) M_{n+1-j}(\Lambda') \ll_n  \prod_{h=1,h \neq j}^n (M_h(\Lambda) M_{n+1-h}(\Lambda'))^{-1} \ll_n 1. \qedhere
  \end{equation*}
\end{proof}



\subsection{\texorpdfstring{Sympletic Structure of $\Lambda_t$}{Structure 
of Symplectic Lattices}}
\label{qf-geom-num}
In the following we shall apply the previous results from the Geometry of 
Numbers to the special $2d$-dimensional lattice $\Lambda_t$ introduced in 
\eqref{df:lambda_t}. The symplectic structure of $\Lambda_t$ will allow us to establish a majorizing relation between the theta series \eqref{eq:thetaestimate2} and the $\alpha_d$-characteristic of $\Lambda_t$, see \eqref{eq:alpha_psi}. To do this, we shall apply Lemma \ref{Dav4} combined with Lemma \ref{Dav3} as follows. (We note that the results of this section remain valid regardless of whether $r \geq q^{1/2}$ or not.)

\begin{lemma}
  \label{succ-dual}
  Let $\Lambda_t$ be the lattice defined in \eqref{df:lambda_t}. Then we have for any $t \in \R$
  \begin{align}
    \label{dual_Min}
    M_j(\Lambda_t)\4 M_{2d+1-j}(\Lambda_t) \asymp_d 1   & \q (j=1,\ldots, d),
    \\
    \label{Minima_size}
    M_1(\Lambda_t)\le \ldots \le M_d(\Lambda_t) \ll_d 1 & \le M_{d+1}(\Lambda_t) \leq \ldots \leq M_{2d}(\Lambda_t),
  \end{align}
  and the lower bound
  \begin{equation}
    \label{Minima_size1}
    M_1(\Lambda_t) \ge \min\{r^{-1}q_0^{1/2}, r q^{-1/2} \}.
  \end{equation}
\end{lemma}

\begin{corollary}
  \label{alphamax}
  As a consequence, we find for $\mu\ge 1$
  \begin{align}
    \label{boxcount1}
     & \# \{ v \in \Lambda_t \4 :\4 \norm{v} \le \mu\} \ll_d \mu^{2\4d}\alpha_d(\Lambda_t), \\
     & \label{eq:alphamax}
    \alpha(\Lambda_t)= \max\{ \alpha_j(\Lambda_t)\, :\, j=1, \ldots, 2d\}  \asymp_d \alpha_d(\Lambda_t)
  \end{align}
  and
  \begin{equation}
    \label{eq:alpha_psi}
    \psi(r,t) \ll_d \alpha_d(\Lambda_t).
  \end{equation}
\end{corollary}

\begin{proof}[Proof of Lemma \ref{succ-dual}]
  First we prove \eqref{dual_Min}. Let
  \begin{equation*}
    J \defi \begin{pmatrix} & \mathbbm 1_d \\ -\mathbbm 1_d & \end{pmatrix},
  \end{equation*}
  and consider the lattice
  \begin{equation*}
    \Lambda_t' = J D_{rQ} U_{4tQ} J^{-1} \Z^{2d}.
  \end{equation*}
  Then $J D_{rQ} U_{4tQ} J^{-1} = D_{rQ}^{-1}U_{-4tQ}^T$ and hence $\Lambda_t '$ is the lattice dual to $\Lambda_t$ in the sense of Lemma \ref{Dav3}. We claim that they have identical successive minima. To this end, note that for any $N = (m, \bar m)^T \in \Z^{2d}$
  \begin{align}
    \label{mineq}
    \norm{D_{rQ} U_{4tQ} N} = \norm{J^{-1}J D_{rQ} U_{4tQ}J^{-1} J N} = 
\norm{D_{rQ}^{-1} U_{-4tQ}^TJ N},
  \end{align}
  where we use that $J$ is an orthogonal matrix. Since $J \Z^{2d} = \Z^{2d}$, the equation \eqref{mineq} implies that the successive minima of $\Lambda_t$ and $\Lambda'_t$ are identical and by Lemma \ref{Dav3} we conclude $ M_j(\Lambda_t) M_{2d+1-j}(\Lambda_t) \asymp_d 1$ for $j=1,\ldots,d$. \par
  To prove \eqref{Minima_size} we note that $M_d \le M_{d+1}$ and $1 \le M_d(\Lambda_t)\4 M_{d+1}(\Lambda_t) \ll_d 1$ implies
  \begin{equation*}
    M_j(\Lambda_t) \leq M_d(\Lambda_t) \ll_d 1 \q \text{and} \q 1 \le M_{d+1}(\Lambda_t) \leq M_{d+j}(\Lambda_t)
  \end{equation*}
  for all $j=1,\ldots,d$. Thus, it remains to show the lower bound \eqref{Minima_size1} for $M_1(\Lambda_t)$: Take $m,\bar{m} \in \Z^d$ with $M_1(\Lambda_t) = \norm{D_{rQ} U_{4tQ} (m, \bar{m})} = H_t(m,\bar{m})^{1/2}$, where $H_t$ denotes the special norm \eqref{eq:thetaestimate3} in the theta series \eqref{eq:thetaestimate2}. If $\bar{m} \neq 0$, then we have $M_1(\Lambda_t) \ge r^{-1} \norm{Q_+^{1/2}\4\bar{m}} \ge q_0^{1/2}r^{-1}$, but otherwise $M_1(\Lambda_t) = r \norm{Q_+^{-1/2}m} \ge r q^{-1/2}$.
\end{proof}

\begin{proof}[Proof of Corollary \ref{alphamax}]
  We begin with proving \eqref{boxcount1} as follows. Recall that $\mu \ge 1$ and let $2d \ge j\ge 1$ denote the maximal integer with $M_j(\Lambda_t) \le \mu$. Then Lemma \ref{Dav4} implies
  \begin{equation*}
    \# \{ v \in \Lambda_t : \, \norm{v} \le \mu \} \ll_d \mu^j \alpha_j(\Lambda_t) \leq \mu^{2\4 d} \alpha_d(\Lambda_t),
  \end{equation*}
  since we have $M_j(\Lambda_t) \ge \ldots \ge M_{d+1}(\Lambda_t) \gg 1$ if $j>d$ and $\mu < M_{j+1}(\Lambda_t)\le \ldots \le M_d(\Lambda_t) \ll_d 1$ if $j < d$. In the case $\mu < M_1(\Lambda_t)$ the inequality in \eqref{boxcount1} holds trivially. Moreover, this argument also proves \eqref{eq:alphamax}. Finally, the estimate \eqref{eq:alpha_psi} follows from the 
relation \eqref{eq:further_estimate_theta} combined with \eqref{boxcount1} for $\mu=d^{1/2}$.
\end{proof}

For arbitrary $t \in \R$ the following bounds hold independently of the Diophantine properties of $Q$.

\begin{lemma}
  \label{apriori}
  Denote by $\Delta$ the lattice $Q_+ ^{1/2} \Z^d$, then
  \begin{equation}
    \label{apriori-alpha}
    \textstyle
    \sup_{t \in \R} \alpha_d( D_{sQ} \4 U_{4tQ} \4 \Z^{2d}) \ll_d \varphi_Q(s),
  \end{equation}
  where $D_{sQ}$ and $U_{4tQ}$ are defined as in \eqref{df:action} and
  \begin{equation}
    \label{apriori-alpha00}
    \varphi_Q(s) \defi s^d \4 \abs{\det Q}^{-1/2}\4 \textstyle \prod_{j\,:\, M_{j}(\Delta) > s} (s^{-2} M_{j}(\Delta)^2), \,\, s >0.
  \end{equation}
  In particular, it follows that
  \begin{alignat}{2}
    \label{apriori-alpha01}
    \varphi_Q(s) & \ll_d s^d \4 \abs{\det Q}^{-1/2}, \q &  & \text{if} \q 
\abs{s} \ge q^{1/2},
  \end{alignat}
  and for small $t$ we get
  \begin{alignat}{2}
    \label{apriori-alpha1}
    \alpha_d( D_{sQ} \4 U_{4tQ} \4 \Z^{2d} ) & \ll_d \abs{\det Q}^{1/2}\4 (s^{-1}+\abs{t \4 s})^d, \q  & \text{if} \q q_0^{1/2}\abs{t\4 s} \ge 1, \\  
    \label{apriori-alpha2}
    \alpha_d(D_{sQ} \4 U_{4tQ} \4 \Z^{2d}) & \ll_d \abs{\det{Q}}^{-1/2} \max\{1,(\sqrt{q}/s)^d\} \abs{t \4 s}^{-d} , \q & \text{if} \q q^{1/2}\4\abs{t\4 s} \le 1.
  \end{alignat}
\end{lemma}

We emphasize that these estimates will be used for a wide range of $s>0$ (depending on the blow-up parameter $r \geq q^{1/2}$), see e.g.\ the proof of Lemma \ref{special-int2}, and for small $t$ as well (by which we mean $r^{-1} q_0^{-1/2} < t < T_{-}$ as stated in Theorem \ref{maintheorem}).

\begin{proof}
  In this proof we replace the definition of $\Lambda_t$, see \eqref{df:lambda_t}, by $\Lambda_t = D_{sQ} \4 U_{4tQ} \Z^{2d}$, i.e.\ $r$ has to be replaced by $s$. If $1/8 < M_1(\Lambda_t)$, then we have 
  \begin{equation}
    \label{equiv_a}
    \alpha_d(\Lambda_t) \asymp_d (M_1(\Lambda_t)\ldots M_d(\Lambda_t))^{-1} \ll_d \# \{ v\in \Lambda_t \, : \, \norm{v} \le 1/8\}.
  \end{equation}
  Otherwise, there exists an integer $j=1,\ldots,d$ with $M_j(\Lambda_t) \leq 1/8 < M_{j+1}(\Lambda_t)$, since $1 \leq M_{d+1}(\Lambda_t)$ holds by \eqref{Minima_size}. Now, taking $\mu=1/8$ in \eqref{boxcount} of Lemma \ref{Dav4} shows that
  \begin{equation*}
    \alpha_d(\Lambda_t) \asymp_d (M_1(\Lambda_t)\ldots M_d(\Lambda_t))^{-1} \! \ll (M_1(\Lambda_t)\ldots M_j(\Lambda_t))^{-1} \asymp_d \# \{ v\in \Lambda_t : \norm{v} \le 1/8\},
  \end{equation*}
  i.e.\ \eqref{equiv_a} holds also in the second case. Recalling again \eqref{eq:thetaestimate3}, we see that the right-hand side of \eqref{equiv_a} is the same as the number all lattice points $m,\bar{m} \in \Z^d$ satisfying
  \begin{equation}
    \label{ht}
    H_t[m,\bar{m}] = s^2 \4 Q_{+}^{-1}[m- 4t\4 Q\4 \bar{m}] + s^{-2}\4 Q_{+}[\bar{m}] \le 1/64,
  \end{equation}
  where the positive form $H_t[\cdot,\cdot]$ is defined as in \eqref{eq:thetaestimate3}, but here again $r$ has to be replaced by $s$.
  
  \textit{Proof of \eqref{apriori-alpha}.} If \eqref{ht} holds, then $\norm{Q_+^{1/2} \bar{m}} \le s/2$, which has again by Lemma \ref{Dav4} at most $\ll_d \prod_{j\,:\, M_{j}(\Delta)\le s} (s \4 M_{j}(\Delta)^{-1})$ integral solutions. Similarly, for fixed $\bar{m}$ the triangle inequality 
combined with \eqref{ht} implies
  \begin{equation*}
    \norm{sQ_+^{-1/2}(m_1-m_2)} \leq \sqrt{H_t[m_1,\bar{m}]}+ \sqrt{H_t[m_2,\bar{m}]} \leq 1.
  \end{equation*}
  Thus, for fixed $\bar{m}$, the number of pairs $(m, \bar{m})$ for which 
\eqref{ht} holds is bounded by the number of elements $v$ in the dual lattice  $\Delta'=Q_{+}^{-1/2}\Z^d$ to $\Delta$ such that $\norm{v} \leq s^{-1}$. Since the successive minima for this dual lattice are determined by Lemma \ref{Dav3}, we may use Lemma \ref{Dav4}, inequality \eqref{boxcount}, again to determine the upper bound
  \begin{equation*}
    \textstyle
    \ll_d  \prod_{j\,:\, M_{j}(\Delta') \leq s^{-1}} (s M_{j}(\Delta'))^{-1} \leq \prod_{j\,:\, M_{j}(\Delta)\ge s} (s^{-1} M_{j}(\Delta))
  \end{equation*}
  for this number as well. The product of both numbers yields the bound
  \begin{equation*}
    \textstyle
    \alpha_d(\Lambda_t) \ll_d  \# \{ v\in \Lambda_t : \norm{v} \le 1/2\} \ll_d  s^d \big( \prod_{j=1}^d M_{j}(\Delta) \big)^{-1} \big( \prod_{j\,:\, M_{j}(\Delta) \ge s}(s^{-2} M_{j}(\Delta)^2) \big).
  \end{equation*}
  Finally, using Lemma \ref{Dav4} in form of $(\prod_{j=1}^d M_{j}(\Delta))^{-1} \asymp_d \alpha_d(\Delta) = \abs{\det \4 Q}^{1/2}$ shows the claimed bound in \eqref{apriori-alpha}. Also the inequality \eqref{apriori-alpha01} follows immediately from \eqref{apriori-alpha00}.
  
  \textit{Proof of \eqref{apriori-alpha1}.} Assume $q_0^{1/2}\abs{t \4 s} 
\ge 1$ and $q_0\ge 1$. If $m=0$ we conclude that $\norm{\bar{m}} \le \abs{4t\4 s} \norm{Q_+^{1/2} \bar{m}} \le 1/8$. Hence $\bar{m}=0$. For any fixed $m \ne 0$ the triangle inequality implies that there is at most one element $\bar{m}\in \Z^d$ with \eqref{ht}. Furthermore, we get $(\norm{Q_+^{-1/2} m} - 1/(8 \4 s)) \le \norm{ 4t\4 Q_+^{1/2}\, \bar{m}}$ for that pair $(m,\bar{m})$. This implies
  \begin{equation*}
    1/8 \ge \sqrt{H_t(m,\bar m)} \ge s^{-1} \norm{ Q_+^{1/2} \4 \bar{m}} \ge \big( \norm{ Q_+^{-1/2} m} - 1/(8 \4 s) \big) / \abs{4t\4 s}
  \end{equation*}
  and hence $\norm{Q_+^{-1/2} m} \le (s^{-1}+\abs{4t\4 s})/8$. Thus
  \begin{equation*}
    \# \{ v \in \Lambda_t\,:\,\norm{v}^2 \le 1/4\} \ll_d (s^{-1}+\abs{t \4 s})^d \, \abs{\det Q}^{1/2}.
  \end{equation*}
  \textit{Proof of \eqref{apriori-alpha2}.} As in the previous case, \eqref{ht} implies by the triangle inequality that
  \begin{equation}
    \label{apriori-alpha4}
    \big\lvert \norm{Q_+^{-1/2} m} - \norm{ 4t\4 Q_+^{1/2}\,S\, \bar{m}} \big\rvert \le (8\4 s)^{-1}
  \end{equation}
  and together with $q^{1/2}\4 \abs{t \4 s} \le 1$ also $\abs{4t\4 s}\4 s^{-1} \norm{Q_+^{1/2}\bar{m}}\le \abs{4t\4 s}/8 \le  (2 q)^{-1/2}$. Moreover one of these inequalities is strict and therefore we have
  \begin{equation}
    \label{apriori-alpha5}
    q^{-1/2} \norm{m} \leq \norm{Q_+^{-1/2} m} < (2 \4 s)^{-1} +(2 \4 q^{1/2})^{-1}.
  \end{equation}
  If $s \geq q^{1/2}$, this leads to a contradiction unless $m=0$. Hence, the possible solutions for $\bar{m}$ in \eqref{apriori-alpha4} satisfy 
$\norm{Q_+^{1/2}\bar{m}} \le \abs{32 t \4 s}^{-1}$ which, as in the proof of \eqref{apriori-alpha}, has at most $\ll_d \abs{\det\4 Q}^{-1/2} \abs{t\4 s}^{-d}$ solutions. In the second case, i.e.\ if $s < q^{1/2}$, the inequality \eqref{apriori-alpha5} has at most $\ll_d (q^{1/2}/s)^d$ solutions for $m$. Now any possible $\bar{m}$ must satisfy
  \begin{equation*}
    \norm{Q_+^{1/2}\bar{m}} \leq \abs{32t s}^{-1} + \abs{4t}^{-1} \norm{Q_+^{-1/2} m} \leq  \abs{2 t s}^{-1}
  \end{equation*}
  again, which completes the proof of \eqref{apriori-alpha2} in view of \eqref{equiv_a}.
\end{proof}



\subsection{Approximation by Compact Subgroups}
\label{compact_groups}
In Section \ref{section-aol} we shall develop mean-value estimates for fractional moments of the $\alpha_d$-characteristic of the lattice $\Lambda_t$ introduced in \eqref{df:lambda_t}. In order to apply techniques from harmonic analysis, we will rewrite the family $\{\Lambda_t\}_{t \in \R}$ as an orbit of a single lattice by means of elements of the one-parameter 
subgroups $\mathrm D \tdefi \{d_r : r>0\}$\index{D@ $\mathrm D :=\{d_r : r>0\}$, diagonal subgroup of $\SL(2,\R)$} and $\mathrm U \tdefi \{u_t : 
t\in \R\}$\index{U@ $\mathrm U := \{u_t : t\in \R\}$, standard unipotent subgroup of $\SL(2,\R)$} of $\SL(2,\R)$, where
\begin{equation}
  \label{svo4}
  d_r \defi \left(\begin{array}{*{2}c} r & 0 \\
      0         & r^{-1}\end{array}\right),\qquad u_t \defi \left(\begin{array}{*{2}c}  1 & -t\\ 0 &  1\end{array}\right),
\end{equation}
and then approximate the subgroup $\mathrm U$ locally by the compact subgroup $\mathrm K= \mathrm{SO}(2)=\{ k_\theta : \theta\in[0,2\pi]\}$\index{K@ $\mathrm K:= \mathrm{SO}(2)=\{ k_\theta : \theta\in[0,2\pi]\}$, orthogonal subgroup of $\SL(2,\R)$} parameterized, as usual, by elements
\begin{equation}
  \label{eq:defkth}
  k_\theta \defi \begin{pmatrix} \cos{\theta} & - \sin{\theta} \\  \sin{\theta}  &  \cos{\theta} \end{pmatrix}.
\end{equation}
Let $S$\index{S@ $S$, orthogonal matrix such that $SQQ_+^{-1}S^T = Q_0$} be an orthogonal matrix such that $SQQ_+^{-1}S^T = Q_0$, where $Q_0$\index{Q@ $Q_0$ signature matrix corresponding to $Q$} denotes the signature matrix corresponding to $Q$, that is $Q_0 = \text{diag}(1,\dots, 1,-1, \dots ,-1)$. A short computation shows that\index{D@ $D_{rQ}$, diagonalizable matrix on $\R^{2d}$}\index{U@ $U_{4tQ}$, unipotent matrix on $\R^{2d}$}
\begin{align*}
  D_{rQ} U_{4tQ} =  \begin{pmatrix} S^T & \\ & S^T \end{pmatrix} d_r u_{4t} \begin{pmatrix} SQ_+^{-1/2} & \\ & SQ_+^{1/2} \end{pmatrix},
\end{align*}
where we embed $\SL(2,\R)$ into $\SL(2d,\R)$ according to the following action
\begin{align}
  \label{embedding}
  \begin{pmatrix} a& b \\ c & d \end{pmatrix} \longmapsto \begin{pmatrix}a \mathbbm 1_d & b \, Q_0 \\ c \,Q_0 & d \mathbbm 1_d \end{pmatrix}.
\end{align}
Define the $2d$-dimensional lattice\index{L@ $\Lambda_Q$, lattice in $\R^{2d}$ depending on $Q$ only}
\begin{align}
  \label{svo5}
  \Lambda_Q \defi \begin{pmatrix} SQ_+^{-1/2} & \\ & SQ_+^{1/2} \end{pmatrix} \Z^{2d},
\end{align}
then as claimed,\index{L@ $\Lambda_t := D_{rQ} \4 U_{4tQ} \Z^{2d}$, lattice on $\R^{2d}$}
\begin{align}
  \label{svo6}
  \Lambda_t = \begin{pmatrix} S^T & \\ & S^T \end{pmatrix} d_r u_{4t} \,\Lambda_Q.
\end{align}
Moreover, since $S$ is orthogonal and $\alpha_i$ is invariant under left multiplication by orthogonal matrices we observe for any $i = 1,\dots,2d$ that
\begin{align}
  \alpha_i(\Lambda_t) = \alpha_i(d_r u_{4t} \Lambda_Q).
\end{align}

\begin{lemma}
  \label{compact-group}
  With respect to the embedding of $\SL(2,\R)$ defined in \eqref{embedding} we have for $t \in \R$, $s \ge 1$ and any $2d$-dimensional lattice $\Lambda$ in $\R^{2d}$
  \begin{alignat}{2}
    \label{eq:kth}
    \alpha_j(d_s\4 u_t \4\Lambda)  \ll_d (1+t^2)^{\frac{j}{2}} \4 \alpha_j(d_s \4 k_\theta \Lambda),  \;\q  j=1, \ldots, 2d,
  \end{alignat}
  where $\theta= \arctan t$.
\end{lemma}

\begin{proof}
  Suppose the signature of $Q$ is $(p,q)$ and let $(v,w) \in \R^d \times \R^d$, thought of as a column vector with coordinates $v_1,\dots,v_d,w_1,\dots, w_d$, then
  \begin{equation}
    \label{eq:kth:proof}
    \norm{d_s u_t (v,w)}^2 = \sum_{i=1}^p \norm{d_s u_t (v_i,w_i)}^2 +  \sum_{i=p+1}^d \norm{d_s u_{-t} (v_i,w_i)}^2.
  \end{equation}
  Let $x,y \in \R$. Note that $y+t\4 x=(1 +t^2)\,y+t\, (x-t\4 y)$, which implies that
  \begin{equation*}
    (y+t\4 x)^2 \le 2\4 (1 +t^2)^2\,(y)^2 + 2\, t^2\,(x-t\4 y )^2,
  \end{equation*}
  and therefore we find
  \begin{equation}
    \label{eq:rho3}
    s^2\,(x-t\4 y)^2+s^{-2}\4 (y+t\4 x)^2 \le 2 \4 (1 +t^2)^2\,\bgl( s^2\,(x-t\4 y)^2+ s^{-2}\4{y}^2\bgr),
  \end{equation}
  provided that $s \geq 1$. Taking $\theta =\arctan t$ and noting that $\cos(\theta) = (t^2+1)^{-1/2}$, resp.\ $\sin(\theta) = t (t^2+1)^{-1/2}$, we see that \eqref{eq:rho3} can be written as
  \begin{equation*}
    \norm{d_s k_\theta (x,y)}^2 \le 2 \4 (1+t^2) \norm{d_s u_t (x,y)}^2,
  \end{equation*}
  and it is easy to see, along the same lines as before, that
  \begin{equation*}
    \norm{d_s k_\theta^T (x,y)}^2 \le 2 \4 (1+t^2)  \norm{d_s u_{-t} (x,y)}^2.
  \end{equation*}
  Hence, we obtain in view of \eqref{eq:kth:proof} that
  \begin{equation*}
    \norm{d_s k_\theta (v,w)}^2  \le 2 \4 (1+t^2) \norm{d_s u_t (v,w)}^2,
  \end{equation*}
  from which we deduce that $(1+t^2)^{i/2} M_i(d_s u_t \Lambda) \gg  M_i(d_s k_\theta \Lambda)$ for any $i =1, \dots ,2d$. The claim follows now 
from \eqref{eq:alpham}.
\end{proof}



\subsection{Irrational and Diophantine Lattices}
\label{subsection:diophantinity}
The purpose of this section is to relate the $\alpha_d$-characteristic of 
$\Lambda_t$ to the Diophantine approximation of $tQ$ by symmetric integral matrices. We begin by motivating the Definition \ref{def:dio_type}: Recall that $Q$ is said to be Diophantine of type $(\kappa,A)$, where $\kappa>0$ and $A>0$, if
\begin{equation*}
  \inf_{t \in [1,2]} \norm{M - mt Q} > A m^{-\kappa} \; \; \text{ for all } m \in \Z \setminus \{0\} \text{ and } M \in \text{Sym}(d,\Z)
\end{equation*}
or equivalently if we introduce the truncated rational approximation error\index{D@ $\delta_{tQ;R}$, rational approximation error of $tQ$ truncated at $R$}
\begin{equation}
  \label{diophant0}
  \delta_{tQ;R} \defi \min \Big\{\norm{M- m\4 t \4 Q}\, :\, m \in \Z, 0 < 
\abs{m} \le R, \, M\in \mathrm{Sym}(d,\Z)\Big\},\,\, R\ge 1,
\end{equation}
we require $Q$ to satisfy
\begin{equation}
  \label{diophant1}
  \inf_{t \in[1,2]} \delta_{tQ;R} > AR^{-\kappa} \, \, \text{ for all } R 
\geq 1.
\end{equation}

\begin{remark}
  As an aside, we remark that the property of $Q$ being Diophantine in the above sense is equivalent to the requirement that for some $\tilde{\kappa}>0$
  \begin{equation*}
    \norm{M-t Q}  > t^{-\tilde{\kappa}},\, \, \, \text{ for all } t \geq 2 \text{ and } M \in \text{Sym}(d,\Z),
  \end{equation*}
  which was introduced in \cite{eskin-margulis-mozes:1998} in the context of forms that are (EWAS). However, this formulation is not optimal because $\tilde{\kappa}$ must be chosen larger than $\kappa$ depending on $A$. Moreover, in most applications the constant $A$ cannot be determined explicitly due to non-effective methods in Diophantine approximation.
\end{remark}

The following lemma justifies calling such forms Diophantine:

\begin{lemma}
  \label{diophantineexample}
  Let $k$ be an integer in the range $1 \leq k \leq \frac{d(d+1)}{2}-1$ and let $Q$ be a form such that $k+1$ non-zero entries $y, x_1, \dots, x_k$ satisfy the property that
  \begin{equation*}
    \max_{i=1,\ldots,k} \abs{q \, x_i / y + p_i} > A q^{-\kappa}
  \end{equation*}
  for all $k$-tuples $(p_1/q,\dots,p_k/q)$ of rationals. Then $Q$ is Diophantine of type  $(\kappa,A')$, where $A'$ depends on $A, \kappa, y, x_1/y, \dots, x_k/y$ only (see \eqref{def:A_apo}).
\end{lemma}

\begin{proof}
  Let $M \in \text{Sym}(d,\Z)$, $m \in \Z \setminus \! \{0\}$ and $t \in [1,2]$. Denoting the entries in $M$ corresponding to the coordinates of $Q$ in which $y, x_1, \dots, x_k$ appear by $q, p_1, \dots, p_k$, we find the inequality
  \begin{equation*}
    \norm{M-m\,tQ} \geq \max \big\{ \max_{1 \leq i \leq k} \abs{p_i-m\,t x_i}, \abs{q 
- m\,t y}\big\}.
  \end{equation*}
  Suppose that the expression on the right-hand side is strictly less than $A'm^{-\kappa}$, where
  \begin{equation}
    \label{def:A_apo}
    A'= \min\{A\,  (4y)^{-\kappa} \,(1 + \max_{1 \leq i \leq k} \abs{x_i /y} ))^{-1},1/2\}.
  \end{equation}
  Note first that $\abs{m} \ge \abs{m \4 t y}/(2 y) > q/(4 y)$ and hence
  \begin{equation*}
    \bigg\lvert \frac{x_i}{y} \4 q - p_i \bigg\rvert \leq \bigg\lvert \frac{x_i}{y}\bigg\rvert \, \abs{q - m \4 t y} + \abs{mt \4 x_i -p_i} < A' m^{-\kappa}(1+ \abs{x_i /y}) < Aq^{-\kappa}
  \end{equation*}
  for all $i=1,\ldots,k$, which yields a contradiction.
\end{proof}

Recall that a number $\theta \in \R$ is called Diophantine of type $\kappa>0$ if there exists $c_\kappa>0$ such that $\abs{q \theta -a} \geq c_\kappa \abs{q}^{-\kappa}$ for every rational number $a/q$. In particular any form $Q$ for which one ratio of two of its entries is a Diophantine number, is Diophantine in the sense of Definition \ref{def:dio_type} and hence almost all forms are Diophantine in this sense. An example of Diophantine forms for which we can control the exponent $\kappa$ is the following: Suppose $Q$ is a form with $k+1$ entries $y, x_1, \dots,x_k$ such that $x_1/y, \dots, x_k/ y$ are algebraic and $1, x_1/y, \dots, 
x_k/ y$ are linearly independent over $\mathbb{Q}$, then Schmidt's Subspace Theorem together with Lemma \ref{diophantineexample} implies that for any $\eta>0$ the form $Q$ is Diophantine of type $(1/k +\eta,A')$, where $A'$ is a constant depending only on $\eta,A,y, x_1/y, \dots, x_k/y$. However, as is usually the case in Diophantine approximation, the constant $A$ and hence $A'$ is ineffective in the sense that these constants cannot 
be determined explicitly.

After the previous motivation, we shall state the main result of this section. In particular, we will see that larger values of $\beta_{t;r}$\index{B@ $\beta_{t;r} := \alpha_d(\Lambda_t) \4 r^{-d} \4 \abs{\det Q}^{1/2}$} (see \eqref{def:beta}) enforce smaller values of the truncated rational approximation error $\delta_{4tQ;R}$ as follows

\begin{lemma}
  \label{alpha_dio}
  Assume that $q_0\ge 1$. Then we have for all $t \in \mathbb{R}$ and $r\ge q^{1/2}$
  \begin{equation}
    \label{dualal}
    \delta_{4 \4 t\4Q; \beta_{t;r}^{-1}} \ll_d q \4 r^{-2} \4 \beta_{t;r}^{-1},
  \end{equation}
  where
  \begin{equation}
    \label{def:beta}
    \beta_{t;r} \defi \alpha_d(\Lambda_t) \4 r^{-d} \4 \abs{\det Q}^{1/2}.
  \end{equation}
  Note that this bound is non-trivial for $\beta_{t;r} > q \4 r^{-2}$ only, due to the uniform bound $\beta_{t;r}\ll_d 1$ for $r \geq q^{1/2}$ established in Lemma \ref{apriori}.
\end{lemma}

Before proving \eqref{dualal}, we shall state some important consequences.

\begin{corollary}
  \label{irr-dio}
  Consider any interval $[T_-,T_+]$ with $T_- \in (0,1]$ and $T_+ \geq 1$.
  \begin{enumerate}[label=\roman*)]
    \item If $Q$ is \textit{irrational}, then
          \begin{equation}
            \label{irrational}
            \lim_{r \rightarrow \infty} \big( \sup_{T_- \le t \le T_+} \alpha_d(\Lambda_t) \4 r^{-d} \, \big) = 0.
          \end{equation}
    \item If $Q$ is \textit{Diophantine} of type $(\kappa,A)$,
          then
          \begin{equation}
            \label{diophant2}
            \sup_{T_- \le t \le T_+} \alpha_d(\Lambda_t)\4 r^{-d} \ll_d \abs{\det{Q}}^{-1/2} (q \4 A^{-1} r^{-2} )^{\frac{1}{\kappa+1}} \4 \max\big\{ (T_-)^{-\frac{1}{\kappa+1}}, (T_+)^{\frac{\kappa}{\kappa+1}}\big\}\4 .
          \end{equation}
  \end{enumerate}
\end{corollary}

A variant of (i) in terms of the successive minima of $\Lambda_t$ can also be found in \cite{goetze:2004}, see Lemma 3.11, yielding an alternative 
proof of \eqref{irrational} when combined with \eqref{eq:alpham}.

\begin{proof}
  \textit{i)} We show the contraposition: Assume that there exists an $\eps >0$ and sequences $(r_j)_j$, $(t_j)_j$ such that $\lim_{j \to \infty} r_j = \infty$ and $\beta_{t_j;r_j} > \eps$. Passing to a subsequence we may assume that $\lim_{j \to \infty} t_j =t$ for some $t \in [T_-,T_+]$. Thus \eqref{dualal} yields $\lim_{j \to \infty} \delta_{4t_jQ; R_j^*}=0$ with $R^*_j \tdefi \beta_{t_j;r_j}^{-1}  < \eps^{-1}$. By definition, this means that $\lim_{j \to \infty} \norm{M_j - 4 t_j m_j Q}=0$ for some $M_j\in \mathrm{Sym}(d,\Z)$ and $m_j \in \Z$ with $\abs{m_j} \le \eps^{-1}$. Obviously both, $\norm{M_j}$ and $\abs{m_j}$, are bounded. Hence there exist integral elements $M$, $m$ and an infinite subsequence $j'$ of $j$ with $M_{j'}=M$, $m_{j'}=m$ and by construction $\lim_{j'} t_{j'}=t$. These limit values satisfy $\norm{M - 4 \4 m\4t\4 Q}=0$, i.e.\ $Q$ is a multiple of a rational form.
  
  \textit{ii)} First we note that for any $t \in [1,T_{+}]$ we have by \eqref{diophant1}
  \begin{equation*}
    \textstyle
    (\delta_{tQ;R})^{-1} \le \sup_{t' \in [1,2]} (\delta_{t'Q;4tR})^{-1} < A^{-1} (4 t R)^\kappa \leq A^{-1} (T_+)^\kappa (4R)^\kappa
  \end{equation*}
  and similarly for $t \in [T_{-},1]$
  \begin{equation*}
    (T_{-})^{-1} \delta_{4tQ;R} \gg  \lceil t^{-1} \rceil \delta_{tQ;4 R} 
\ge \delta_{(\lceil t^{-1} \rceil t)Q;4R} > A (4R)^{-\kappa}.
  \end{equation*}
  Thus, the relation \eqref{dualal}, established in Lemma \ref{alpha_dio}, implies for any $t \in [T_-,T_+]$ that
  \begin{equation*}
    \beta_{t;r} \ll_d q r^{-2} (\delta_{4tQ;\beta_{t;r}^{-1}})^{-1} \ll_d 4^\kappa q \4 
r^{-2} A^{-1} \max\{ (T_-)^{-1}, (T_+)^{\kappa} \} (\beta_{t;r})^{-\kappa} ,
  \end{equation*}
  where we used \eqref{dualal}. Therefore we conclude \eqref{diophant2} as claimed.
\end{proof}

%

\begin{proof}[Proof of Lemma \ref{alpha_dio}]
  We begin by recalling that $\Lambda_t = D_{rQ} \4 U_{4tQ} \4 \Z^{2d}$ 
(see \eqref{df:lambda_t}), where
  \begin{equation*}
    D_{rQ} = \begin{pmatrix} \,r Q_{+}^{-1/2} & 0\\ 0 & r^{-1}\, Q_{+}^{1/2}\end{pmatrix} \q \ \ \text{and} \q \ \ U_{4tQ} = \begin{pmatrix} \, \mathrm{I}_d & - 4 \4 t \4 Q \\ 0 &  \mathrm{I}_d \end{pmatrix}.
  \end{equation*}
  As noted in Remark \ref{remark_existence_alpha_l} the $\alpha_d$-characteristic of $\Lambda_t$ is attained at some sublattice,  that is we can write $\alpha_d(\Lambda_t) = \norm{ w_1 \wedge \ldots \wedge w_d}^{-1}$ by means of vectors $w_j \tdefi D_{rQ} U_{4t Q} l_j$ with linear independent points $l_1,\ldots,l_d \in \Z^{2d}$ depending on $t$. Here we use the 
standard Euclidean norm on the exterior product $\largewedge^d\R^{2d}$. Moreover, we write $l_j=(m_j,n_j)$, where $m_j,n_j \in \Z^d$ and the coordinates of $(m_j,n_j)$ are the coordinates of the vectors $m_j$ and $n_j$ in the corresponding order. Additionally, we introduce the $d\times d$ integer matrices $N$ and $M$ with columns $n_1, \ldots , n_d$ and $m_1,\ldots,m_d$ as well. Using this notation, we may write
  \begin{equation}
    \label{rewrite-wedge}
    w_1 \wedge \ldots \wedge w_d = (D_{rQ} U_{4tQ}) \begin{pmatrix} M \\ N \end{pmatrix} e_1 \wedge \ldots \wedge e_d.
  \end{equation}
  First, we shall prove that
  \begin{equation}
    \label{detN}
    \alpha_d(\Lambda_t) > q \4 d_Q \4 r^{d-2} \q \q \text{implies} \q \q \beta_{t;r}^{-1} >\abs{\det(N)} > 0.
  \end{equation}
  Note that the left-hand side of \eqref{detN} can be rewritten as $\beta_{t;r} > q \4 r^{-2}$ and we may assume that this inequality holds, since 
otherwise the bound \eqref{dualal} is trivial.

  Let us show that $\operatorname{rank}(N)=d$. To this end, we write $k=d-\operatorname{rank}(N)$. According to elementary divisor theory (for matrices with entries in a principal ideal domain) there exist $P, P' \in \GL(d,\Z)$ such that $P'N P$ is a diagonal matrix with positive entries of the form $\text{diag}(0,\dots,0,a_{k+1},\dots,a_d)$ with $a_i \mid a_{i+1}$, $a_i \in \N$. In particular $NP$ is a matrix whose first $k$ columns are zero. Moreover, since $\det{P} = \pm 1$, we conclude that
  \begin{equation*}
    \begin{pmatrix} M P \\ N P \end{pmatrix} e_1 \wedge \ldots \wedge e_d 
=\pm \begin{pmatrix} M \\ N \end{pmatrix} e_1 \wedge \ldots \wedge e_d,
  \end{equation*}
  and hence we can assume from now on that $N = (0,\ldots,0,n_{k+1},\ldots,n_{d})$ with linearly independent vectors $n_{k+1},\ldots,n_d \in \Z^d$. 
Since $l_1,\ldots,l_d$ constitute a basis of a $d$-dimensional lattice, we note that $m_1,\ldots,m_k$ are necessarily linearly independent. Now we 
shall express $w_1\wedge\ldots \wedge w_d$ in terms of the standard basis 
$e_{I} \wedge e_{J}$ indexed by pairs of subsets $I \subset \{1,\ldots, d\}$ and $J \subset \{d+1, \ldots, 2d\}$ with $\abs{I}+\abs{J}=d$, i.e.\ we write
  \begin{equation*}
    w_1 \wedge \ldots \wedge w_d = \sum_{I,J} \omega_{I,J} e_I \wedge e_J.
  \end{equation*}
  Let $I = \{i_1,\ldots,i_m\}$ and $J = \{j_1,\ldots,j_{d-m}\}$, then 
the coefficients $\omega_{I,J}$ are given by
  \begin{equation}
    \label{coeffjk}
    \omega_{I,J} \defi \det{\begin{pmatrix} A_I & * \\ 0 & B_J \end{pmatrix}},
  \end{equation}
  where
  \begin{align*}
    \label{coeff2}
    A_I & \defi \begin{pmatrix}
      \langle r Q_+^{-\frac{1}{2}} m_1, e_{i_1} \rangle & \dots & \langle 
r Q_+^{-\frac{1}{2}} m_k, e_{i_1} \rangle \\
      \vdots & & \vdots \\
      \langle r Q_+^{-\frac{1}{2}} m_1, e_{i_m} \rangle & \dots & \langle 
r Q_+^{-\frac{1}{2}} m_k, e_{i_m} \rangle \\
    \end{pmatrix}  \\
    B_J & \defi \begin{pmatrix}
      \langle r^{-1} Q_+^{\frac{1}{2}} n_{k+1}, e_{j_1} \rangle     & \dots & \langle r^{-1} Q_+^{\frac{1}{2}} n_{d}, e_{j_1} \rangle     \\
      \vdots                                                        &     
  & \vdots                                                      \\
      \langle r^{-1} Q_+^{\frac{1}{2}} n_{k+1}, e_{j_{d-m}} \rangle & \dots & \langle r^{-1} Q_+^{\frac{1}{2}} n_{d}, e_{j_{d-m}} \rangle \\
    \end{pmatrix}.
  \end{align*}
  Since the matrix in \eqref{coeffjk} is of block-type, we find
  \begin{equation}
    \label{betabound}
    \begin{aligned}
      \alpha_d(\Lambda_t)^{-2} & = \norm{w_1 \wedge \ldots \wedge w_d}^2  
                                                                      \\
       & \ge \sum_{\abs{I} =k} \sum_{\abs{J} = d-k } \omega_{I,J}^2 = \Big( \sum_{\abs{I} =k} (\det{A_I})^2 \Big) \Big( \sum_{\abs{J} = d-k } (\det{B_J})^2 \Big) \\
                               & = r^{4k-2d} \norm{Q_+^{-\frac{1}{2}} (m_1 
\wedge \ldots \wedge m_k)}^2 \4 \norm{Q_{+}^{\frac{1}{2}} (n_{k+1} \wedge \ldots \wedge n_d)}^2.
    \end{aligned}
  \end{equation}
  Without loss of generality assume that the eigenvalues of $Q$ are indexed such that $\abs{q_1} \leq \dots \leq \abs{q_d}$. Since $q_0 \geq 1$, note that 
the minimal eigenvalue of the $k$-th exterior power of $Q_{+}^{-1/2}$ is given by $\abs{q_{d-k+1} \ldots q_d}^{-1/2}$ and that of the $(d{-}k)$-th exterior power of $Q_+^{1/2}$ is precisely $\abs{q_1 \ldots q_{d-k}}^{1/2}$. Hence, since $m_1,\ldots, m_k$ and $n_{k+1}, \ldots, n_d$ are linearly independent and integral, we obtain the following lower bound
  \begin{equation*}
    \alpha_d(\Lambda_t)^{-1} \ge r^{2k-d} \left( \frac{\abs{q_1 \ldots q_{d-k}}}{\abs{q_{d-k+1} \ldots q_d}}\right)^{1/2} \ge q^{-1}  \abs{\det{Q}}^{1/2} 
r^{2-d}.
  \end{equation*}
  where we used that $r \ge q^{1/2}$. In view of \eqref{detN}, this strict inequality yields a contradiction unless $k=0$. Thus, we proved that $k=0$, i.e.\ $\abs{\det N} > 0$. Now \eqref{betabound} also implies $\beta_{t;r}^{-1} \ge  \abs{\det N}$. Hence, the upper bound for $\abs{\det N}$ in \eqref{detN} holds as well. \par
  Finally, we shall prove \eqref{dualal}. Since $N$ is invertible, we can rewrite $w_1 \wedge \ldots \wedge w_d$ by
  \begin{equation}
    \label{lasteq0}
    (D_{rQ} \4 U_{4tQ}) \begin{pmatrix}  M N^{-1} \\ \mathbbm{1}_d \\ \end{pmatrix}  N \4  ( e_1 \wedge \ldots \wedge e_d ) = (\det{N}) (D_{rQ} \4 U_{4tQ}) \begin{pmatrix}  M N^{-1} \\ \mathbbm{1}_d \\ \end{pmatrix} e_1 \wedge \ldots \wedge e_d,
  \end{equation}
  i.e.\ we parametrized the subspace spanned by $l_1,\ldots,l_d$. Introduce also the $2d {\times} d$ matrix
  \begin{equation*}
    W \defi (D_{rQ} \4 U_{4tQ}) \begin{pmatrix}  M N^{-1} \\ \mathbbm{1}_d \\ \end{pmatrix} = \begin{pmatrix}  r Q_{+}^{-\frac{1}{2}}(M N^{-1}-4tQ) \\ r^{-1} Q_{+}^{\frac{1}{2}} \\ \end{pmatrix}
  \end{equation*}
  and note that $W^T W$ is a positive definite symmetric $d \times d$ matrix. Thus, there exists an orthogonal matrix $V \in O(d)$ such that $D \tdefi V^T W^T W V$ is diagonal with positive entries. Since $(\det{V}) (e_1 \wedge \ldots \wedge e_d) = V (e_1 \wedge \ldots \wedge e_d)$ it follows that
  \begin{equation}
    \label{lasteq1}
    \begin{aligned}
      \norm{W (e_1\wedge & \ldots \wedge e_d)}^2 = \norm{WV (e_1\wedge \ldots \wedge e_d)}^2                                                           
               \\
                     & = \langle D (e_1\wedge \ldots \wedge e_d), (e_1\wedge \ldots \wedge e_d) \rangle =\prod_{i=1}^d \norm{D e_i} = \prod_{i=1}^d \norm{ W v_i}^2,
    \end{aligned}
  \end{equation}
  where $v_1,\ldots,v_d$ denote the columns of $V$. Next observe that
  \begin{equation}
    \label{lasteq2}
    \max_{1\le i \le d} \norm{W v_i} \ge \max_{1 \le i \le d} \norm{rQ_{+}^{-\frac{1}{2}}(MN^{-1}-4t Q) v_i} \gg_d r q^{-\frac{1}{2}} \norm{M N^{-1}-4t Q}.
  \end{equation}
  Now let $i_0$ be a subscript for which $\norm{Wv_i}$ is maximal. Similar to the proof of \eqref{betabound} we may write $ W ( \wedge_{i \ne i_0} v_i) = \sum \omega_{I,J} e_{I} \wedge e_J$, where the sum is taken over subsets $I \subset \{1,\ldots,d\}$ and $J \subset \{d+1,\ldots,2d\}$ with $\abs{I}+\abs{J} = d-1$, and find that
  \begin{equation}
    \label{lasteq3}
    \norm{W (\wedge_{i \ne i_0} v_i)}^2 \ge \sum_{\abs{I}=0, \abs{J} = d-1} \omega_{I,J}^2 = \norm{r^{-1} Q_+^{\frac{1}{2}} (\wedge_{i \ne i_0} v_i)}^2 \ge r^{-2(d-1)} q^{-1} \abs{\det{Q}}.
  \end{equation}
  Combining \eqref{lasteq0} together with \eqref{lasteq1}--\eqref{lasteq3} yields
  \begin{align*}
    \alpha_d(\Lambda_t)^{-1} & = \abs{\det(N)} \, \norm{W v_{i_0}} \,{\textstyle \prod_{i \neq i_0}} \norm{Wv_i} = \abs{\det(N)} \, \norm{W v_{i_0}} \, \norm{W (\wedge_{i \neq i_0} v_i)} \\ & \gg_d r^{-(d-2)} q^{-1} \abs{\det{Q}}^\frac{1}{2} \,\abs{\det{N}}\, \norm{MN^{-1} -4 t Q}.
  \end{align*}
  Since $(\det N)\4 N^{-1}$ is an integral matrix, the last line together 
with \eqref{detN} implies
  \begin{equation*}
    \min\{\norm{ \bar{M} - 4 \4 m\4 t \4 Q} \, : \, 0< \abs{m} \le \beta_{t;r}^{-1}  ,\, m, \bar{M} \text{ integral} \} \ll_d q \4 r^{-2} \4 \beta_{t;r}^{-1},
  \end{equation*}
  and, since $Q$ is symmetric, we may take $\bar{M}$ symmetric as well, which proves \eqref{dualal}.
\end{proof}



\section{\texorpdfstring{Averages Along Translates of Orbits of $\mathrm{SO}(2)$}{Averages Along Translates of Orbits of SO(2)}}
\label{section-aol}

\subsection{Application of Geometry of Numbers}
\label{R}
In view of the bound \eqref{crucialbound} we need to estimate the error term $I_\theta$, that is \eqref{I1}. Proceeding as in \eqref{bound3.43} combined with the estimates $\abs{\theta_v(t)} \ll_d \abs{ \det Q}^{-1/4}\4 
r^{d/2}\4 \psi(r,t)^{1/2}$ and $\psi(r,t) \ll_d \alpha_d(\Lambda_t)$, obtained in Lemma \ref{thetaestimate2} respectively \eqref{eq:alpha_psi} of Corollary \ref{alphamax}, leads to
\begin{equation}
  \label{eq:5.1:intro}
  I_\theta \ll_d r^{d/2} \4 \abs{\det{Q}}^{-1/4} \4 \norm{\wh{\zeta}}_1 \int_{\abs{t} > \imbound} \abs{\wh{g}_w(t)} \4 \alpha_d(\Lambda_t)^{1/2} \, \dif t,
\end{equation}
where $\Lambda_t$ denotes the lattice defined in \eqref{df:lambda_t} and $g_w$ the smoothed indicator function of $[a,b]$ with $0 < w < (b-a)/4$, see Corollary \ref{co1}. Since Lemma \ref{l2} provides estimates for $\norm{\wh{\zeta}}_1$ in the case of both admissible and non-admissible regions $\Omega$, it remains to estimate the integral in \eqref{eq:5.1:intro}. 
We shall start with bounding this integral over an interval $I$ of length 
at most $1/q$. For this, we introduce the maximum value over $I$ of the $\alpha_d$-characteristic for the lattice $\Lambda_t$ via\index{G@ $\gamma_{[a,b],\beta}(r)$, Diophantine factor for $Q$ on $[a,b]$ with exponent $\beta$}
\begin{equation}
  \label{defgamma}
  \gamma_{I, \beta}(r) \defi \sup\bigl\{ \big( r^{-d}\4 \alpha_d(\Lambda_t)\big)^{\frac{1}{2}-\beta}: \, \, t\in I \bigr\}
\end{equation}
and the following family of lattices\index{L@ $\Lambda_{Q,t}:= d_{q^{1/2}} \4 u_{4t}\4 \Lambda_Q$, lattice on $\R^{2d}$}
\begin{equation}
  \label{latticelambdaqt}
  \Lambda_{Q,t}\tdefi d_{q^{1/2}} \4 u_{4t}\4 \Lambda_Q,
\end{equation}
where $\Lambda_Q$ is as defined in \eqref{svo6}. Here $\gamma_{I, \beta}(r)$ depends on the Diophantine properties of $Q$ and tends to zero for growing $r \rightarrow \infty$ by Lemma \ref{irr-dio} for irrational $Q$.

\begin{lemma}
  \label{special_int}
  Let $r \ge q^{1/2}$, $0<\beta  \le 1/2$ and fix an interval $I=[\tau_1,\tau_2]$ of length at most $1/q$. Then we have
  \begin{equation}
    \label{eq:alpha_int1}
    \int_I  \alpha_d(\Lambda_t)^{1/2} \4 \abs{\wh{g}_w(t)}  \, \dif t
    \ll_d \wh{g}_I \4 r^{\frac{d}{2} - \beta \4 d} \4 \gamma_{I,\beta}(r) 
\frac{1}{q} \int_{-\pi}^\pi  \alpha(d_{r_*} \4 k_\theta \4 \Lambda_{Q,4\tau_1})^{\beta}\4 \frac{\dif \theta}{2 \4 \pi},
  \end{equation}
  where $r_* \tdefi r \4 q^{-1/2}$\index{R@ $r_* := r \4 q^{-1/2}$} and 
$\wh{g}_I \tdefi \max \{ \abs{\wh{g}_w(t)} : t \in I\}$\index{G@ $\wh{g}_I  := \max \{\abs{\wh g_w(t)} : t \in I\}$, maximum of $\abs{\wh g_w(t)}$ on an interval $I$}.
\end{lemma}

\begin{proof}
  Using the trivial bound $\alpha_d(\Lambda_{t}) \le r^{d- 2 \4 \beta \4 d} \gamma_{I, \beta}(r)^2 \4 \alpha_d(\Lambda_{t})^{2\beta}$ and estimating $\abs{\wh{g}_w}$ by its maximum $\wh{g}_I$ on $I$ yields
  \begin{equation}
    \label{finalbound1}
    \int_I \alpha_d(\Lambda_t)^{1/2} \4 \abs{\wh{g}_w(t)} \, \dif t \le \wh{g}_I \4 r^{\frac{d}{2} - d\4\beta} \gamma_{I,\beta}(r) \int_I \alpha_d(\Lambda_{t})^{\beta}\4 \dif t.
  \end{equation}
  Since the group $\mathrm{D}$ normalizes $\mathrm{U}$, a computation shows that $d_r\4 u_{4t} = d_r \4 u_{4(t-\tau_1)} \4 u_{4 \tau_1}=d_{r_*}\4 u_{\tau} \4 d_{q^{1/2}} \4 u_{4\tau_1}$, where $\tau \tdefi 4 \4 (t-\tau_1)\4 q$. Changing variables from $t$ to $\tau$ we obtain in terms of the lattices $\Lambda_{Q,s}$, defined in \eqref{latticelambdaqt},
  \begin{equation}
    \label{finalbound2}
    \int_I \alpha_d(\Lambda_t)^\beta \dif t
    = \int_{\tau_1}^{\tau_2} \alpha_d(d_{r_*}\4 u_{\tau}\4 d_{q^{1/2}} u_{4\tau_1}\4\Lambda_Q)^{\beta} \dif t
    \ll \frac{1}{q} \int_{0}^4 \alpha_d(d_{r_*}\4 u_\tau \4 \Lambda_{Q,4\tau_1})^\beta \dif \tau.
  \end{equation}
  Finally, we estimate the last average with the help of Lemma \ref{compact-group} by the average over the group $\text{K} = \mathrm{SO}(2)$. Changing variables $\theta(s) = \arctan(\tau)$, $\tau \in [0,4]$, and noting that $\abs{\theta} < \pi$ and $\dif \tau = (1+\tau^2) \, \dif \theta$, we get by \eqref{eq:kth} of Lemma \ref{compact-group} that
  \begin{equation*}
    \int_{0}^4 \alpha_d(d_{r_*}\4 u_{\tau}\4 \Lambda_{Q,4\tau_1})^{\beta} 
\4 \dif \tau
    \ll \int_0^4 \alpha_d(d_{r_*} \4 k_{\theta(\tau)} \4 \Lambda_{Q,4\tau_1})^{\beta} \4 \dif \tau
    \ll \int_{-\pi}^{\pi} \alpha_d(d_{r_*}\4 k_\theta \4 \Lambda_{Q,4\tau_1})^{\beta}\4 \frac {\dif \theta} {2 \4 \pi}.
  \end{equation*}
  Now note that $\alpha_d(\Lambda) \le \alpha(\Lambda)$ holds for any lattice $\Lambda$ in $\R^{2d}$. Thus, the last inequality together with \eqref{finalbound1} and \eqref{finalbound2} completes the proof.
\end{proof}

In the following paragraphs we shall develop explicit bounds for averages 
over the group $\mathrm K$ of type $\int_{\mathrm K} \alpha_d(d_r\4 k \4 \Lambda)^{\beta} \, \dif k$.



\subsection{\texorpdfstring{Operators $A_g$ and Functions $\tau_\lambda$ on $\SL(2,\R)$}{Operators A and Functions tau on SL(2,R)}}
\label{s1}
Let $\mathrm  G = \SL(2,\R)$\index{G@ $\mathrm  G = \SL(2,\R)$}. We consider the following two subgroups of $\mathrm  G$\index{T@ $\mathrm  T $, Borel subgroup of $\SL(2,\R)$}:
\begin{equation*}
  \mathrm K = \mathrm{SO}(2) = \left\{ k_\theta \,:\, 0\leq\theta < 2\pi\right\} \q \text{and} \q \mathrm  T = \left\{\begin{pmatrix} \, a & 
b \\ 0 & a^{-1}\end{pmatrix} :  a > 0 , \, b\in\R\right\},
\end{equation*}
where $k_\theta$ is defined in \eqref{eq:defkth}. According to the Iwasawa decomposition, any $g\in G$ can be uniquely represented as a product of 
elements from $\mathrm{K}$ and $\mathrm{T}$, that is\index{I@ Iwasawa decomposition of $g \in \SL(2,\R)$: $g = k(g)t(g), \, k(g)\in \mathrm K, \, t(g)\in \mathrm  T$}
\begin{equation*}
  g = k(g)t(g), \q k(g)\in \mathrm K, \, t(g)\in \mathrm  T.
\end{equation*}
Now let
\begin{equation*}
  d_a \defi \begin{pmatrix} a & 0 \\ 0 & a^{-1} \end{pmatrix}  \ \text{for} \ a > 0 \,\, \text{and} \,\, \mathrm  D^+ = \{ d_a : a\geq 1\}.
\end{equation*}
According to the Cartan decomposition\index{C@ Cartan decomposition of $g 
\in \SL(2,\R)$: $g=k_1(g)d(g)k_2(g)$}, we have
\begin{equation*}
  \mathrm{G=KD^+K}, \,\, g=k_1(g)d(g)k_2(g), \,\, g\in \mathrm  G, k_1(g), k_2(g)\in \mathrm  K,\,\, d(g)\in \mathrm  D^+.
\end{equation*}
In this decomposition $d(g)$ is determined by $g$, and if $g\notin \mathrm  K$ then $k_1(g)$ and $k_2(g)$ are also determined by $g$ up to a factor of $\pm 1$ on $k_1$ and $k_2$. It is clear that $\norm{g} = \norm{d(g)}$, where $\norm{\,\cdot \,}$ denotes the operator norm induced by the standard Euclidean norm on $\R^2$. Note that, in the simple case $g= d_a$, this norm is given by $\norm{d_a}=a$. Since $d_a$ is the conjugate of $d_{a^{-1}}$ by $k_{\pi /2}$, we see that $g^{-1}\in \mathrm  Kg \mathrm K$ or equivalently, $d(g) = d(g^{-1})$ for any $g\in \mathrm  G$. Therefore, $\norm{g} = \norm{g^{-1}}$, $g\in \mathrm  G$. \par We say that a function $f$ on $\mathrm  G$ is \textit{left $\mathrm  K$-invariant} (resp.\ \textit{right $\mathrm  K$-invariant}, resp.\ \textit{bi-$\mathrm  K$-invariant}) if $f(\mathrm  Kg) = f(g)$ (resp.\ $f(g\mathrm  K) = f(g)$, resp.\ $f(\mathrm  Kg \mathrm  K) = f(g)$). Any bi-$\mathrm K$-invariant function on $\mathrm G$ is completely determined by its restriction to $\mathrm D^+$. Hence for any bi-$\mathrm  K$-invariant function $f$ on $\mathrm G$, there is a function $f^*$\index{1@ $f^*$, radial realization of a bi-$\mathrm K$-invariant function on $\mathrm G$} on $[1,\infty )$ such that $f(g) =  f^*(\norm{g})$, $g\in \mathrm  G$.

For any $\lambda\in\R$ we define a character $\chi_\lambda$ of $\mathrm  T$\index{X@ $\chi_\lambda$, character of $\mathrm  T$}  by
\begin{equation*}
  \chi_\lambda \begin{pmatrix} \, a & b \\ 0 & a^{-1}\end{pmatrix} = a^{-\lambda}
\end{equation*}
and the function $\varphi_\lambda :\mathrm  G\to\R^+$\index{P@ $\varphi_\lambda :\mathrm  G\to\R^+$, corresponding to the character $\chi_\lambda$} by
\begin{equation*}
  \varphi_\lambda (g) = \chi_\lambda (t(g)), \q g\in G. 
\end{equation*}
The function $\varphi_\lambda$ has the property
\begin{equation}
  \label{7.3}
  \varphi_\lambda (kgt) = \chi_\lambda (t) \varphi_\lambda (g),\q g\in \mathrm  G, \, k\in \mathrm  K, \, t\in \mathrm  T,
\end{equation}
and it is completely determined by this property and the condition $\varphi_\lambda (1) = 1$. \par
For $g\in \mathrm  G$ and a continuous action of $\mathrm  G$ on a topological space $X$, we define the operator $A_g$\index{A@ $A_g$, mean-value operator on $\mathrm G$} on the space of continuous functions on $X$ by
\begin{equation}
  \label{7.4}
  (A_gf)(x) = \int_{\mathrm K} f(gkx) \, \dif\sigma (k), \q x\in X,
\end{equation}
where $\sigma$ is the normalized Haar measure on $\mathrm  K$, or, using the parametrization of $\mathrm K$, by
\begin{equation*}
  (A_gf)(x) = \frac{1}{2\pi} \int_0^{2\pi} \, f(g k_\theta x) \, \dif\theta,\q x\in X.
\end{equation*}
The operator $A_g$ is a linear map into the space of left $\mathrm{K}$-invariant functions on $X$. If $X=\mathrm{G}$ and $\mathrm{G}$ acts on itself by left translations, then $A_g$ commutes with right translations. From these two remarks, or using a direct computation, we get that $A_g\varphi_\lambda$ has the property \eqref{7.3}. Hence $\varphi_\lambda$ is an 
eigenfunction for $A_g$ with the eigenvalue\index{T@ $\tau_\lambda$, spherical function on $\mathrm G$}
\begin{equation}
  \label{7.6}
  \tau_\lambda (g) \defi (A_g \varphi_\lambda)(1) = \int_{\mathrm K} \varphi_\lambda (gk) \, \dif\sigma (k) = \int_{\mathrm K} \chi_\lambda (t(gk)) \, \dif \sigma (k).
\end{equation}
We see from \eqref{7.6} that $\tau_\lambda$ is obtained from $\varphi_\lambda$ by averaging over right translations by elements of $\mathrm K$. But $\varphi_\lambda$ is left $\mathrm K$-invariant and $A_g$ commutes with 
right translations. Hence the function $\tau_\lambda$ is bi-$\mathrm K$-invariant and it is an eigenfunction for $A_g$ with the eigenvalue $\tau_\lambda (g)$, that is
\begin{equation}
  \label{7.7}
  (A_g\tau_\lambda)(h) = \tau_\lambda (g) \tau_\lambda (h) \q \text{for 
all} \q h\in \mathrm G.
\end{equation}
We have that
\begin{equation}
  \label{7.8}
  \varphi_\lambda (g) = \norm{ge_1}^{-\lambda},\q g\in G, \, e_1 = (1,0),
\end{equation}
where $\norm{\cdot}$ denotes the usual Euclidean norm on $\R^2$. Indeed
\begin{equation*}
  \varphi_\lambda (g) = \chi_\lambda (t(g)) = \norm{t(g)e_1}^{-\lambda} = \norm{k(g)t(g)e_1}^{-\lambda} = \norm{g e_1}^{-\lambda}.
\end{equation*}
From \eqref{7.6} and \eqref{7.8} we get
\begin{equation}
  \label{7.9}
  \begin{aligned}
    \tau_\lambda (g) & = \int_K \norm{gke_1}^{-\lambda} \, \dif \sigma(k)
    = \frac{1}{2\pi} \int_0^{2\pi} \norm{gk(\theta )e_1}^{-\lambda} \, \dif \theta                                                               
                    \\
                     & = \frac{1}{2\pi} \int_0^{2\pi} \norm{g(\cos \theta, \sin\theta)}^{-\lambda} \, \dif\theta = \int_{S^1} \norm{gu}^{-\lambda} \, \dif \ell (u),
  \end{aligned}
\end{equation}
where $S^1$ is the unit circle in $\R^2$ and $\ell$ denotes the normalized rotation invariant measure on $S^1$. One can easily see that $\norm{gu}^{-2}$, $g\in G$, $u\in S^1$, is equal to the Jacobian at $u$ of the diffeomorphism $v\mapsto gv/\norm{g v}$ of $S^1$ onto $S^1$. On the other hand, it follows from the change of variables formula that
\begin{equation*}
  \int_M J^\lambda_f = \int_M \, J^{1-\lambda}_{f^{-1}}, \q \lambda\in \R,
\end{equation*}
where $f \colon M \to M$ is a diffeomorphism of a compact differentiable manifold $M$ and $J_f$ (resp.\ $J_{f^{-1}}$) denotes the Jacobian of $f$ (resp.\ $f^{-1}$). Now using \eqref{7.9} we get
\begin{equation}
  \label{7.10}
  \tau_\lambda (g) = \tau_{2-\lambda} (g^{-1}) = \tau_{2-\lambda} (g), \, g\in G, \, \lambda \in \R.
\end{equation}
The second equality in \eqref{7.10} is true because $\tau_\lambda$ is bi-$\mathrm K$-invariant and $g^{-1}\in \mathrm Kg \mathrm K$. Since, obviously, $\tau_0 (g) = 1$, it follows that
\begin{equation}
  \tau_2 (g) = \tau_0 (g) = 1. \label{7.11}
\end{equation}
Since $t^{-\lambda}$ is a strictly convex function of $\lambda$ for any $t > 0, t\neq 1$, it follows from \eqref{7.9} that $\tau_\lambda (g)$ is a 
strictly convex function of $\lambda$ for any $g\in G$. From this, \eqref{7.10} and \eqref{7.11} we deduce that
\begin{alignat}{2}
   & \tau_\eta (g) < \tau_\lambda (g) \q                            &  & \text{for any} \, \, g\notin \mathrm K \,\, \text{and} \,\, 1 \le \eta < \lambda \le 2, \nonumber \\
  \label{7.13}
   & \tau_\eta (g) < 1 \,\, \text{and} \,\, \tau_\lambda (g) > 1 \q &  & \text{for any} \, \, g\notin \mathrm K,\,\, 0 < \eta < 2, \lambda > 2, \, 
\text{ and }      \\
  \label{7.14}
   & \tau_\eta (g) < \tau_\lambda (g) \q                            &  & \text{for any} \,\, g\notin \mathrm K, \; \lambda\geq 2, \, 0 < \eta < \lambda.
\end{alignat}
Since the function $\tau_\lambda (g)$ is bi-$\mathrm K$-invariant, it depends only on the norm $\norm{g}$ of $g$. Thus, we can write
\begin{equation}
  \label{7.15}
  \tau_\lambda (g) = \tau_\lambda^* (\norm{g}), \q g \in \mathrm G,
\end{equation}
where for $a\geq 1$
\begin{equation}
  \label{7.16}
  \tau_\lambda^* (a) = \tau_\lambda (d_a) = \int_{\mathrm K} \norm{d_a k e_1}^{-\lambda} \, \dif \sigma (k) = \frac{1}{2\pi} \int_0^{2\pi} \frac{\dif \theta}{(a^{2}\cos^{2}\theta +a^{-2}\sin^{2}\theta)^{\lambda /2}}.
\end{equation}
In view of \eqref{7.7} and the definition of $A_g$, we get
\begin{equation}
  \label{7.17}
  \int_K \tau_\lambda^* ( \norm{gkd_a} )\, \dif\sigma (k) = \tau_\lambda (g) \tau_\lambda^* (a), \q g\in \mathrm G, a\geq 1.
\end{equation}
Since $\Vert g\Vert = \Vert g^{-1}\Vert$ for all $g\in \mathrm G$,
\begin{equation*}
  \frac{a}{\Vert g\Vert} \leq \Vert gkd_a\Vert \leq a\Vert g\Vert
\end{equation*}
for all $k\in \mathrm K$ and $g\in \mathrm G$. From this, \eqref{7.13} and \eqref{7.17} we deduce that, for any $\lambda > 2$, the continuous function $\tau_\lambda^* (a), a\geq 1$, does not have a local maximum. Hence $\tau_\lambda^*$ is strictly increasing for all $\lambda > 2$ or, equivalently,
\begin{equation}
  \label{7.18}
  \tau_\lambda (g) < \tau_\lambda (h) \q \text{if} \,\, \Vert g\Vert < \Vert h\Vert, \, g, h\in G, \,\, \lambda > 2.
\end{equation}
Using \eqref{7.10} and \eqref{7.16} yields
\begin{equation}
  \label{7.19}
  \tau_\lambda^*(a) = \tau_{2-\lambda}^*(a)  = \frac{1}{2\pi} \int_0^{2\pi} (a^2 \cos^2 \theta + a^{-2} \sin^2 \theta )^{\frac{\lambda}{2}-1} \, \dif \theta.
\end{equation}
Since $a^2\cos^2\theta\leq a^2\cos^2 \theta + a^{-2} \sin^2 \theta \leq a^2$, we deduce from \eqref{7.19} the estimates
\begin{equation}
  \label{7.20}
  c(\lambda )a^{\lambda-2} \leq \hat\tau_\lambda (a) \leq a^{\lambda -2}, 
\,\, a\geq 1, \,\, \lambda \geq 2,
\end{equation}
where
\begin{equation}
  \label{harishcvalue}
  c(\lambda ) = \frac{1}{2\pi} \int_0^{2\pi} \, \abs{\cos{\theta}}^{\lambda-2} \, \dif \theta = \frac{2}{\pi} \int_0^{\pi/2} \cos(\theta)^{\lambda-2} \, \dif \theta = \frac{\mathrm{B} \big(\tfrac{\lambda-1}{2},\tfrac{1}{2} \big)}{\pi} = \frac{\Gamma(\frac{\lambda-1}{2})}{\Gamma(\frac{\lambda}{2}) \sqrt{\pi}},
\end{equation}
$\mathrm{B}$ denotes the beta function and we use the identity $\mathrm{B}(x,y) = \Gamma(x)\Gamma(y)/\Gamma(x+y)$ as well as $\Gamma(1/2)=\sqrt \pi$. From \eqref{7.19} we also conclude that for any $\lambda > 2$ the 
ratio $\frac{\tau_{\lambda}^*(a)}{a^{\lambda-2}}$ is a strictly decreasing function of $a\geq 1$ and
\begin{equation}
  \label{harishc}
  \lim_{a\to\infty} \,\, \frac{\tau_{\lambda}^*(a)}{a^{\lambda -2}} = c(\lambda).
\end{equation}

\begin{remark}
  The function $\tau_\lambda$ can be viewed as a spherical function on the upper-half plane $\mathbb H$ (see \cite{helgason:1984} Chapter IV Proposition 2.9) and all spherical functions on $\mathbb H$ are of this form for some $\lambda \in \C$. In particular, it is not difficult to see that $\tau_\lambda$ can also be represented as
\begin{align*}
  \tau_\lambda(g) = \frac{1}{2 \pi}\int_0 ^{2\pi} \big( \cosh(2 \log \norm{g}) + \sinh(2 \log \norm{g}) \sin(\theta) \big)^{\lambda/2 -1} \dif \theta.
\end{align*}
Moreover, for $\text{Re}(\lambda)>1$ it is well-known that $c(\lambda)$, which is usually referred to as Harish-Chandra's $c$-function, as defined 
in \eqref{harishc} exists and its value is given by \eqref{harishcvalue} (see \cite{helgason:1984} Introduction Theorem 4.5 or \cite{lang:1985} Chapter V \S 5).
\end{remark}

\begin{lemma}
  \label{Lemma 7.1}
  Let $g\in \mathrm G, g\notin \mathrm K$, $\lambda > 2$, $0 < \eta < \lambda$, $b\geq 0$, $B > 1$, and let $f$ be a left $\mathrm K$-invariant positive continuous function on $\mathrm G$. Assume that
  \begin{equation}
    \label{7.22}
    A_gf \leq \tau_\lambda (g)f + b \4 \tau_\eta
  \end{equation}
  and that
  \begin{equation}
    \label{7.23}
    f(yh) \leq B f(h) \q \text{if}\q h, y\in \mathrm G \ \, \text{and} \ \, \norm{y} \leq \norm{g}.
  \end{equation}
  Then for all $h\in G$
  \begin{equation*}
    (A_hf)(1) = \int_{\mathrm K} f(hk)\, \dif \sigma (k) \leq s\tau_\lambda (h),
  \end{equation*}
  where
  \begin{equation}
    \label{7.25}
    s = B \left( f(1) +\frac{b}{\tau_{\lambda}(g)-\tau_\eta (g)}\right).
  \end{equation}
\end{lemma}

\begin{proof}
  We define
  \begin{equation*}
    f_{\mathrm K} (h)\defi \int_{\mathrm K} f(hk) \, \dif \sigma (k), \q h\in \mathrm G.
  \end{equation*}
  Since $A_g$ commutes with right translations, and $\tau_\eta$ is right $\mathrm K$-invariant, it follows from \eqref{7.22} that $A_g f_{\mathrm K}\leq \tau_\lambda (g) f_{\mathrm K} + b\tau_\eta$. If $h$ and $y$ are as in \eqref{7.23}, then $f(yhk) \le B f(hk)$ for every $k\in \mathrm K$ and therefore $f_{\mathrm K}(yh) \leq B f_{\mathrm K}(h)$. On the other hand, it is clear that
  \begin{equation*}
    f_{\mathrm K}(h) = (A_hf_{\mathrm K})(1) = (A_h f)(1).
  \end{equation*}
  Thus we can replace $f$ by $f_{\mathrm K}$ and assume that $f$ is bi-$\mathrm K$-invariant. Then we have to prove that $f\leq s\tau_\lambda$. Assume the contrary, then $f(h) > s'\tau_\lambda (h)$ for some $h\in G$ and 
$s' > s$. In view of \eqref{7.14} and \eqref{7.25}, $s' > s\geq Bf(1)$. From this, \eqref{7.18} and \eqref{7.23} we get that $\norm{h} > \norm{g}$ 
and
  \begin{equation}
    \label{7.26}
    f(yh) > \frac{s'}{B} \, \tau_\lambda (yh) \q \text{if} \q \norm{y} \leq \norm{g} \ \ \text{and} \ \ \norm{yh} \leq \norm{h}.
  \end{equation}
  Using the Cartan decomposition, we see that any $x\in \mathrm G$ with $\frac{\norm{h}}{\norm{g}} \leq \norm{x} \leq \norm{h}$ can be written as $x = k_1 yh k_2$, where $k_1, k_2\in \mathrm K$, $\norm{y}\leq \norm{g}$ and $\norm{yh} \leq \norm{h}$. But the functions $f$ and $\tau_\lambda$ 
are bi-$\mathrm K$-invariant. Therefore it follows from \eqref{7.26} that
  \begin{equation}
    \label{7.27}
    f(x) > \frac{s'}{B} \, \tau_\lambda (x) \q \text{if} \q \frac{\norm{h}}{\norm{g}} \leq \norm{x} \leq \norm{h}.
  \end{equation}
  Let
  \begin{align*}
     & a_1\defi \frac{s'}{B} > f(1) + \frac{b}{\tau_\lambda(g)-\tau_\eta(g)}, \,\, a_2 \defi \frac{b}{\tau_{\lambda}(g)-\tau_\eta(g)}, \,\, \text{and} \\
     & \omega  \defi f - a_1\tau_\lambda + a_2\tau_\eta.
  \end{align*}
  In view of \eqref{7.7} and \eqref{7.22}, we see that
  \begin{equation}
    \label{7.28}
    \begin{aligned}
      A_g\omega - \tau_\lambda (g) \omega & = A_g (f-a_1 \tau_\lambda + 
a_2 \tau_\eta) - \tau_\lambda (g)(f-a_1 \tau_\lambda + a_2\tau_\eta)      
                                            \\
                                          & = \left[A_g f -\tau_\lambda(g)f] - a_1 [A_g\tau_\lambda - \tau_\lambda (g)\tau_\lambda\right] + a_2 \left[A_g\tau_\eta - \tau_\lambda (g)\tau_\eta \right] \\
                                          & \leq b\tau_\eta + a_2 \left[\tau_\eta (g) \tau_\eta - \tau_\lambda (g) \tau_\eta \right] = 0.
    \end{aligned}
  \end{equation}
  Since $\tau_\lambda (1) = \tau_\eta (1) = 1$, we have
  \begin{equation}
    \label{7.29}
    \omega (1) = f(1) - a_1 + a_2 < 0.
  \end{equation}
  It follows from \eqref{7.14} that $a_2\geq 0$. Using additionally \eqref{7.25} and \eqref{7.27}, we get that
  \begin{equation}
    \label{sec:5:dreieck}
    \begin{aligned}
      \omega (x) & = f(x) - a_1\tau_\lambda (x) + a_2 \tau_\eta (x) \geq f(x) - a_1 \tau_\lambda (x) \\ &> \left(\frac{s'}{B} - a_1 \right) \tau_\lambda (x) = 0 \,\, \q \text{if}  \q \,\, \frac{\norm{h}}{\norm{g}} \, \leq \norm{x} \leq \norm{h}.
    \end{aligned}
  \end{equation}
  Let $v\in \mathrm G$, satisfying $\norm{v} \le \norm{h}$, be a point where the continuous function $\omega$ attains its minimum on the set $\{ x\in \mathrm G: \norm{x} \leq \norm{h} \}$. It follows from \eqref{7.29} and \eqref{sec:5:dreieck} that
  \begin{equation*}
    \omega (v) < 0 \q \text{and} \q \norm{v} \leq \frac{\norm{h}}{\norm{g}}.
  \end{equation*}
  Because of $\tau_\lambda (g) > 1$ and $\norm{gkv} \leq \norm{g} \norm{v}$ for all $k\in \mathrm K$ we conclude
  \begin{equation*}
    (A_g\omega)(v) = \int_{\mathrm K} \omega (gkv) \, \dif\sigma (k) \geq \omega (v) > \tau_\lambda (g) \omega (v).
  \end{equation*}
  Thus, we get a contradiction with \eqref{7.28}.
\end{proof}

As a special case ($\eta = 2$ and $b =0$) of Lemma \ref{Lemma 7.1}, we have the following

\begin{corollary}
  \label{Lemma 7.2}
  Let $g\in \mathrm G$, $g\notin \mathrm K$, $\lambda > 2$, $B > 1$, and let $f$ be a left $\mathrm K$-invariant positive continuous function on $\mathrm G$ satisfying the inequality \eqref{7.23}. Assume that
  \begin{equation*}
    A_gf\leq \tau_\lambda (g) f.
  \end{equation*}
  Then for all $h\in G$
  \begin{equation*}
    (A_hf)(1) = \int_{\mathrm K} f(hk) \, \dif \sigma(k) \le B f (1) \tau_\lambda (h).
  \end{equation*}
\end{corollary}

\begin{lemma}
  \label{Lemma 7.3}
  Let $g\in \mathrm G$, $g\notin \mathrm K$, $2 < \lambda < \mu$, $B > 1$, $M > 1$, $n\in \N^{+}$ and let $f_i$,  $0\leq i \leq n$, be left $\mathrm K$-invariant positive continuous functions on $\mathrm G$. We denote $\min\{i, n-i\}$ by $\bar i$ and $\sum_{0\leq i\leq n} f_i$ by $f$. Assume 
that
  \begin{alignat}{2}
     & f_i(yh) &  & \le B f_i (h) \q \text{if} \ \ 0\leq i\leq n,\; \; h, 
y\in \mathrm G \; \; \text{and}\; \; \norm{y} \le \norm{g}, \nonumber \\
     & A_gf_i  &  & \leq \tau_\lambda (g) f_i + M \,\,\max_{0 < j\leq\bar 
i} \sqrt{f_{i-j}f_{i+j}}, \q 0 \leq i\leq n, \label{7.36}
  \end{alignat}
  so in particular $A_g f_0 \le \tau_\lambda (g)f_0$ and $A_gf_n \leq \tau_\lambda (g) f_n$.
  Then there is a constant $C = C(g, \lambda ,\mu, B, M, n)$ such that for all $h\in \mathrm G$,
  \begin{equation}
    (A_hf)(1) = \int_{\mathrm K} \, f(hk) \, \dif \sigma (k) \leq C f(1) \tau_\mu (h). \label{7.38}
  \end{equation}
\end{lemma}

\begin{proof}
  For any $0 < \eps \leq 1$ and $0\leq i \leq n$ we define
  \begin{equation*}
    f_{i,\eps} = \eps^{q(i)} f_i \q \text{where}\,\, q(i) \defi i(n-i).
  \end{equation*}
  Using the inequality \eqref{7.36} for all $i$, $0\leq i\leq n$, we see that
  \begin{align*}
    A_gf_{i,\eps} = \eps^{q(i)} A_g f_i & \leq \eps^{q(i)} \tau_\lambda 
(g) f_i + \eps^{q(i)} M \max_{0 <j \leq \bar i} \sqrt{\eps^{-q(i-j)} f_{i-j,\eps} \eps^{-q(i+j)} f_{i+j,\eps}} \\
                                        & = \tau_\lambda (g) f_{i,\eps} 
+ M \max_{0<j\leq\bar i} \eps^{q(i)-\frac{1}{2}[q(i-j)+q(i+j)]} \sqrt{f_{i-j,\eps} f_{i+j,\eps}}.
  \end{align*}
  Direct computation shows that
  \begin{equation*}
    q(i) - \frac{1}{2} [q(i-j) + q(i+j)] = j^2.
  \end{equation*}
  Hence for all $i$, $0\leq i\leq n$,
  \begin{equation}
    \label{7.41}
    A_g f_{i,\eps} \leq \tau_\lambda (g) f_{i,\eps} + \eps M \max_{0<j\leq\bar i} \, \sqrt{f_{i-j,\eps} f_{i+j,\eps}}.
  \end{equation}
  Let $f_\eps \tdefi \sum_{0\leq i\leq n} f_{i,\eps}$. Summing \eqref{7.41} over all $i$, $0\leq i\leq n$, and using the inequalities $f_\eps > \sqrt{f_{i-j,\eps} \, f_{i+j,\eps}}$, which are satisfied for any $1\leq i\leq n-1$, $0 < j \leq \bar i$, we get
  \begin{equation}
    \label{7.42}
    A_g f_\eps = \sum_{0\leq i\leq n} A_g f_{i,\eps} \leq \tau_\lambda (g) f_\eps + \eps M (n-1)f_\eps =\left(\tau_\lambda (g) + \eps M (n-1)\right) f_\eps.
  \end{equation}
  Write
  \begin{equation*}
    \eps_0 = \min\left\{ 1, \frac{\tau_\mu (g) - \tau_\lambda (g)}{M(n-1)}\right\}
  \end{equation*}
  in order to get from \eqref{7.42} that
  \begin{equation*}
    A_g f_{\eps_{0}} \leq \tau_\mu (g) f_{\eps_{0}}.
  \end{equation*}
  Since $f_\eps$ also satisfies \eqref{7.23}, we can apply Corollary \ref{Lemma 7.2} to $f_{\eps_{0}}$ and get that
  \begin{equation*}
    (A_hf)(1) < \eps^{-n^{2}}_0 (A_h f_{\eps_{0}}) (1) \leq \eps^{-n^{2}}_0 f_{\eps_{0}} (1) \tau_\mu (h) \leq \eps^{-n^{2}}_0 B f(1)\tau_\mu (h)
  \end{equation*}
  for all $h\in \mathrm G$. Hence \eqref{7.38} is true with $C = \eps^{-n^{2}}_0 B$.
\end{proof}

\begin{proposition}
  \label{Prop-7.1}
  Let $g\in \mathrm G$, $g\notin \mathrm K$, $d\in\N^+$, $B>1$, $M>1$. For every $0\leq i\leq 2d$, let $\lambda_i\geq 2$ and let $f_i$ be a left $\mathrm K$-invariant positive continuous function on $\mathrm G$. We denote $\min\{ i, 2d-i\}$ by $\bar{i}$ and $\sum_{0\leq i\leq 2d} f_i$ by $f$. Assume that
  \begin{equation*}
    \lambda_d > \lambda_i \q \text{for any} \,\, i\neq d.
  \end{equation*}
  \begin{equation}
    \label{7.46}
    f_i(yh)\leq B f_i (h) \q \text{if} \q 0 \leq i \leq 2d, \ h,y \in \mathrm G \ \text{and} \ \norm{y} \leq \norm{g},
  \end{equation}
  \begin{equation}
    \label{7.47}
    A_gf_i \leq \tau_{\lambda_{i}}(g) f_i + M \max_{0<j\leq\bar i} \sqrt{f_{i-j} f_{i+j}}, \q 0\leq i\leq 2d,
  \end{equation}
  in particular,
  \begin{equation*}
    A_gf_0 \leq \tau_{\lambda_{0}} (g)f_0 \q \text{and} \q \ A_g f_{2d}\leq \tau_{\lambda_{2d}} (g) f_{2d}.
  \end{equation*}
  Then, using the notation $\ll$ (which until the end of the proof of this proposition means that the left hand side is bounded from above by the right-hand side multiplied by a constant which depends on $g, \lambda_0, \ldots , \lambda_{2d}, B$ and $M$, and does not depend on $f_0,\ldots ,f_{2d}$), we have that
  \begin{enumerate}
    \item For all $h\in \mathrm G$ and $0\leq i\leq 2d,\; i\neq d$,
          \begin{equation*}
            (A_hf_i)(1) = \int f_i (hk) \, \dif\sigma(k) \ll f(1) \tau_\eta (h),
          \end{equation*}
          where
          \begin{equation}
            \label{7.50}
            \eta = \lambda_d - 3^{-(d+1)} (\lambda_d - \eta') < \lambda_d, \q \eta' = \max\{\lambda_i :0\leq i\leq 2d, i\neq d\}.
          \end{equation}
    \item For all $h\in \mathrm G$
          \begin{equation*}
            (A_hf_d)(1) = \int_{\mathrm K} f_d(hk) \, \dif\sigma (k) \ll f(1) \tau_{\lambda_{d}} (h).
          \end{equation*}
    \item For all $h\in \mathrm G$
          \begin{equation*}
            (A_hf)(1) = \int_{\mathrm K} f(hk) \, \dif\sigma (k) \ll f(1) \norm{h}^{\lambda_{d}-2}.
          \end{equation*}
  \end{enumerate}
\end{proposition}

\begin{proof}
  (a) Let
  \begin{equation*}
    f_{i,\mathrm K}(h) \defi \int_{\mathrm K} f_i(hk) \, \dif\sigma (k), \q h\in \mathrm G.
  \end{equation*}
  The Cauchy-Schwarz inequality implies
  \begin{align*}
    \int_{\mathrm K} \sqrt{f_{i-j} (hk) f_{i+j} (hk)} \, \dif \sigma (k)
     & \leq \sqrt{\int_{\mathrm K} f_{i-j}(hk) \, \dif \sigma (k)} \, \sqrt{\int_{\mathrm K} f_{i+j} (hk) \, \dif \sigma (k)} \\
     & = \sqrt{f_{i-j,\mathrm K} (h) f_{i+j,\mathrm K} (h)}.
  \end{align*}
  Hence
  \begin{align*}
    \int_{\mathrm K} \max_{0<j\leq\bar i} \sqrt{f_{i-j} (hk) f_{i+j}(hk)} 
\, \dif \sigma (k)
     & \leq \sum_{0<j\leq\bar i} \int_{\mathrm K} \sqrt{f_{i-j} (hk) f_{i+j}(hk)} \, \dif \sigma (k) \\
     & \leq \sum_{0<j\leq\bar i} \sqrt{f_{i-j, \mathrm K}(h) f_{i+j, \mathrm K}(h)}                  \\
     & \leq d \max_{0<j\leq\bar i} \sqrt{f_{i-j, \mathrm K} (h) f_{i+j,\mathrm K} (h)}.
  \end{align*}
  On the other hand, we have
  \begin{equation*}
    (A_g f_{i,\mathrm K})(h) = \int_{\mathrm K} (A_g f_i)(hk) \, \dif\sigma (k)
  \end{equation*}
  and according to \eqref{7.47}
  \begin{equation*}
    (A_g f_i)(hk)\leq\tau_{\lambda_{i}} (g) f_i(hk) + M \max_{0<j\leq\bar 
i} \sqrt{f_{i-j} (hk) f_{i+j}(hk)}.
  \end{equation*}
  Therefore
  \begin{equation*}
    A_g f_{i,\mathrm K} \leq\tau_{\lambda_{i}} (g) f_{i,\mathrm K} + dM \max_{0<j\leq\bar i} \sqrt{f_{i-j,\mathrm K}f_{i+j,\mathrm K}}.
  \end{equation*}
  But $f_{\mathrm K}(1) = f(1)$,
  \begin{equation*}
    f_{i,\mathrm K} (h) = (A_hf_{i,\mathrm K}) (1) = (A_hf_i)(1)
  \end{equation*}
  and, as easily follows from \eqref{7.46}, we have
  \begin{equation*}
    f_{i,\mathrm K} (yh)\leq Bf_{i,\mathrm K} (h)
  \end{equation*}
  if $h,y\in \mathrm G$, and $\norm{y} \leq \norm{g}$. Thus, replacing $f_i$ by $f_{i,\mathrm K}$ and $M$ by $dM$, we can assume that the functions $f_i$ are bi-$\mathrm K$-invariant. Then we have to prove that
  \begin{equation}
    \label{7.53}
    f_i \ll f(1) \tau_\eta \q \text{for all} \,\, 0\leq i\leq 2d, \; i\neq d.
  \end{equation}
  Let $\eta' = \max\{ \lambda_i : 0\leq i\leq 2d, i\neq d\}$, as in \eqref{7.50}. We define $\mu_i$, $0\leq i\leq 2d$, by
  \begin{align}
    \label{7.54}
    \mu_d & = \lambda_d + 3^{-(d+1)} (\lambda_d -\eta') \q \text{and}   
           \\
    \label{7.55}
    \mu_i & = \mu_d -3^{-\bar i} (\lambda_d - \eta'), \ 0\leq i\leq 2d, 
\ i\neq d.
  \end{align}
  Since \eqref{7.14} implies $\tau_{\lambda_{i}} (g)\leq \tau_{\mu_{d}} (g)$, it follows from \eqref{7.14} and Lemma \ref{Lemma 7.3} that
  \begin{equation}
    \label{7.56}
    f_i \ll f(1)\tau_{\mu_{d}}, \q 0\leq i\leq 2d.
  \end{equation}
  One can easily check that $\eta > \mu_i > \lambda_i\geq 2$ and therefore $\tau_\eta \geq \tau_{\mu_{i}}$ for all $0\leq i\leq 2d, i\neq d$. Thus, to prove \eqref{7.53}, it is enough to show that
  \begin{equation}
    \label{7.57}
    f_i \ll f(1) \tau_{\mu_{i}} \q \text{for all} \ 0\leq i\leq 2d, \ i\neq d.
  \end{equation}
  We will prove \eqref{7.57} for $i\leq d-1$ by using induction in $i$; the proof in the case $i\geq d+1$ is similar. For $i=0$ we have $\tau_{\mu_{0}}(g) > \tau_{\lambda_{0}}(g)$ because of \eqref{7.14} and thus it is enough to use Corollary \ref{Lemma 7.2}. Let $1\leq m\leq d-1$ and assume that \eqref{7.57} is proved for all $i < m$. Using \eqref{7.56} for all $0 < j\leq m$ we find that
  \begin{equation}
    \label{7.58}
    \sqrt{f_{m-j}f_{m+j}} \ll f(1)\sqrt{\tau_{\mu_{m-j}} \tau_{\mu_{d}}} \leq f(1) \sqrt{\tau_{\mu_{m-1}} \tau_{\mu_{d}}} \ll f(1) \tau_{(\mu_{m-1}+\mu_{d})/2}.
  \end{equation}
  Note that the second inequality in \eqref{7.58} follows from \eqref{7.14} and \eqref{7.55}, and the third one follows from \eqref{7.15} and \eqref{7.20}.\par
  Combining \eqref{7.47} and \eqref{7.53} we get
  \begin{equation*}
    A_gf_m \leq\tau_{\lambda_{m}} (g) f_m + Cf(1) \tau_{(\mu_{m-1}+\mu_{d})/2},
  \end{equation*}
  where $C \ll 1$. On the other hand, we have $\lambda_m < \mu_m$ and $(\mu_{m-1} + \mu_d)/2 < \mu_m$ by \eqref{7.54} and \eqref{7.55}. Now, to prove that $f_m \ll f(1)\tau_{\mu_{m}}$, it remains to apply Lemma \ref{Lemma 7.1} combined with \eqref{7.14}.\par
  (b) As in the proof of (a), we can assume that the functions $f_i$ are bi-$K$-invariant. Then we get from \eqref{7.47} and \eqref{7.53} that
  \begin{equation*}
    A_gf_d\leq\tau_{\lambda_{d}} f_d + Df(1) \tau_\eta,
  \end{equation*}
  where $D \ll 1$. Since $\eta < \lambda_d$, Lemma \ref{Lemma 7.1} implies that $f_d \ll f(1)\tau_{\lambda_{d}}$ which proves (b).\par
  (c) Follows from (a), (b), \eqref{7.14}, \eqref{7.15} and \eqref{7.20}.
\end{proof}



\subsection{\texorpdfstring{Quasinorms and Representations of $\SL(2,\R)$}{Quasinorms and Representations of SL(2,R)}}
\label{s2}
We say that a continuous function $v\mapsto \abs{v}$ on a real topological vector space $V$ is a \textit{quasinorm} if it satisfies the following properties
\begin{enumerate}[label=(\roman*)]
  \item $\abs{v} \geq 0$ and $\abs{v} = 0$ if and only if $v=0$,
  \item $\abs{\lambda v} = \abs{\lambda} {\cdot} \abs{v}$ for all $\lambda\in\R$ and $v\in V$.
\end{enumerate}
If $V$ is finite dimensional, then any two quasinorms on $V$ are equivalent in the sense that their ratio lies between two positive constants.

\begin{lemma}
  \label{Lemma8.4}
  Let $\rho$ be a (continuous) representation of $\mathrm G = \SL(2,\R)$ in a real topological vector space $V$, let $\abs{\, \cdot \,}$ be a $\rho(\mathrm K)$-invariant quasinorm on $V$ and let $v\in V, v\neq 0$, be an eigenvector for $\rho$ corresponding to the character $\chi_{-r}, r\in\R$, that is
  \begin{equation*}
    \rho\begin{pmatrix} a & b \\ 0 & a^{-1}\end{pmatrix} v = a^rv.
  \end{equation*}
  Then for any $g\in \mathrm G$ and $\beta\in\R$
  \begin{equation}
    \label{8.1}
    \abs{\rho(g)v}^{-\beta} = \varphi_{\beta r}(g) \abs{v}^{-\beta}
  \end{equation}
  and
  \begin{equation}
    \label{8.2}
    \int_{\mathrm K} \frac{\dif\sigma(k)}{\abs{\rho(gk)v}^\beta} = \tau_{\beta r} (g) \abs{v}^{-\beta}.
  \end{equation}
\end{lemma}

\begin{proof}
  Using the $\mathrm K$-invariance of $\vert\cdot\vert$ we get that
  \begin{align*}
    \abs{\rho (g)v}^{-\beta} & = \abs{\rho (k(g))\rho (t(g))v}^{-\beta} 
=\abs{\rho (t(g))v}^{-\beta} =\abs{\chi_{-r} (t(g))v}^{-\beta} = \chi_{\beta r}(t(g)) \abs{v}^{-\beta} \\
                             & = \varphi_{\beta r} (g) \abs{v}^{-\beta}.
  \end{align*}
  The equality \eqref{8.2} follows from \eqref{8.1} and from the definition of $\tau_{\beta r}(g)$.
\end{proof}

Let $\norm{z}$ denote the norm of $z\in\C^2$ corresponding to the standard Hermitian inner product on $\C^2$, that is
\begin{equation*}
  \norm{z}^2 = \norm{x}^2 + \norm{y}^2 \q \text{where} \ z = x + iy, \, x, y\in\R^2.
\end{equation*}

\begin{lemma}
  \label{Lemma 8.5}
  For any $z\in\C^2$, $z \neq 0$, $g\in G$ and $\beta > 0$, we have
  \begin{equation}
    \label{8.3}
    F(z) = F_{g,\beta}(z) \defi \Vert z \Vert^\beta \int_{\mathrm K} \frac{\dif \sigma (k)}{\Vert gkz\Vert^{\beta}} \leq \tau_\beta (g).
  \end{equation}
\end{lemma}

\begin{proof}
  Since the measure $\sigma$ on $\mathrm K$ is translation invariant, we have
  \begin{equation}
    \label{8.4}
    F(kz) = F(z) \,\, \text{for any} \,\, k\in \mathrm K.
  \end{equation}
  Also for all $\lambda\in \C$, $\lambda\neq 0$, and $z\in\C^2, z\neq 0$,
  \begin{equation}
    \label{8.5}
    F(\lambda z) = F(z),
  \end{equation}
  because $\norm{\lambda v} = \abs{\lambda} {\cdot} \norm{v}$, $v\in\C^2$, and because $\mathrm G = \SL(2,\R)$ acts $\C$-linearly on $\C^2$. Any non-zero vector $x\in\R^2$ can be represented as $x=\lambda ke_1$ with $\lambda\in\R$, $k\in \mathrm K$, $e_1 = (1,0)$. Then, using \eqref{7.9} from Section \ref{s1}, we get from \eqref{8.4} and \eqref{8.5} that
  \begin{equation}
    \label{8.6}
    F(x) = F(e_1) = \tau_\beta (g) \,\, \text{for all} \,\, x\in\R^2, 
x\neq 0.
  \end{equation}
  Let now $z = x+iy$, $x,y\in\R^2$, $z\neq 0$. We write $e^{i\theta} z = x_\theta + iy_\theta$, $x_\theta, y_\theta\in\R^2$. Then $\frac{\norm{x_\theta}}{\norm{y_\theta}}$ is a continuous function of $\theta$ with values in $\R_{\geq0}\cup\{\infty\}$. But $e^{i\pi /2}z = iz = -y+ ix$ 
and therefore $\frac{\norm{x_{\pi /2}}}{\norm{y_{\pi /2}}} = \left(\frac{\norm{x_0}}{\norm{y_0}}\right)^{-1}$. Hence there exists $\theta$ such that $\norm{x_\theta} = \norm{y_\theta}$. Replacing then $z$ by $e^{i\theta}z$ and using \eqref{8.5} we can assume that $\norm{x_\theta} = \norm{y_\theta}$. Now using the convexity of the function $t\to t^{-\beta /2}$, $t > 0$, and the identity \eqref{8.6} we get that
  \begin{equation}
    \label{8.7}
    \begin{aligned}
      \int_{\mathrm K} \frac{\dif\sigma (k)}{ \norm{gkz}^\beta} & =
      \int_{\mathrm K} \frac{\dif\sigma (k)}{(\norm{gkx}^2+\norm{gky}^2)^{\beta /2}}                                                               
 \\
      & \leq \frac{2^{-\beta /2}}{2} \left[\int_K \frac{\dif\sigma (k)}{\norm{gkx}^\beta}
        + \int_{\mathrm K} \frac{\dif\sigma (k)}{\norm{gky}^\beta}\right]
      = \frac{2^{-\beta /2}}{2} \left[ \frac{\tau_\beta (g)}{\norm{x}^\beta} + \frac{\tau_\beta (g)}{\Vert y\Vert^\beta}\right]                  
   \\
      & = 2^{-\beta /2} \tau_\beta (g) \frac{1}{\Vert x\Vert^\beta}
      =2^{-\beta /2} \tau_\beta (g)\cdot \frac{1}{\Vert z\Vert^\beta\cdot 2^{-\beta /2}}
      = \frac{\tau_\beta (g)}{\Vert z\Vert^\beta}.
    \end{aligned}
  \end{equation}
  Clearly the last inequality \eqref{8.7} implies \eqref{8.3}.
\end{proof}

\noindent Let us recall some basic facts of the finite-dimensional representation theory of $\mathrm{G} = \SL(2,\R)$. Let $W$ be a finite-dimensional complex vector space, there is a correspondence between complex-linear representations of $\mathfrak{sl}(2,\C)$ on $W$ and representations of $\mathrm G$ on $W$, under which invariant subspaces and equivalences are preserved (see \cite{knapp:2001}  Proposition 2.1). It is well-known that any finite-dimensional representation of $\mathfrak{sl}(2,\C)$ is fully reducible, that is, it can be decomposed into the direct sum of irreducible representations (see \cite{knapp:2002} Corollary 1.70). Moreover, for each $m \geq 1$ there exists up to equivalence a unique irreducible complex-linear representation of $\mathfrak{sl}(2,\C)$ on a complex vector space of dimension $m$ (see \cite{knapp:2002} Corollary 1.63).
Hence, any finite-dimensional representation of $\mathrm G$ is fully reducible and any two irreducible finite-dimensional representations of the same degree must be isomorphic.
Let $\mathcal{P}_m$ denote the $(m+1)$-dimensional complex vector space of complex polynomials in two variables homogeneous of degree $m$, and let 
$\psi_m$ denote the regular representation of $\mathrm G = \SL(2,\R)$ on $\mathcal{P}_m$ defined by $(\psi_m(g)P)(z) = P(g^{-1}z)$, for $g\in \mathrm G, z\in \C^2$ and $P\in\mathcal{P}_m$. It is well-known that the representation $\psi_m$ is irreducible for any $m$ (see \cite{kowalski:2014} Example 2.7.11) and hence it is, up to isomorphism, the unique irreducible finite-dimensional representation of $\mathrm G$ of degree $m$. We define
\begin{equation*}
  I(\rho ) = \{ \, m\in \N^+ \, : \, \psi_m \ \text{is isomorphic to a subrepresentation of} \ \rho\ \}.
\end{equation*}

\begin{proposition}
  \label{Proposition8.6}
  Let $\rho$ be a representation of $\mathrm G = \SL(2,\R)$ on a finite-dimensional space $W$. Then there exists a $\rho(\mathrm K)$-invariant quasinorm $\abs{\;\cdot \;} = \abs{\;\cdot\;}_\rho$ on $W$ such that for 
any $w\in W, w\neq 0$, $g\in \mathrm G$ and $\beta > 0$,
  \begin{equation*}
    \int_{\mathrm K} \frac{\dif\sigma (k)}{\abs{\rho (gk)w}^\beta}
    \leq \max_{m\in I(\rho)} \{\tau_{\beta m}(g)\} \frac{1}{\abs{w}^\beta}.
  \end{equation*}
\end{proposition}

\begin{proof}
  Let $W =\bigoplus_{i=1}^n W_i$ be the decomposition of $W$ into the 
direct sum of $\rho(\mathrm G)$-irreducible subspaces, and let $\pi_i \colon W\to W_i$ denote the natural projection. Suppose that we constructed for each $i$ a $\mathrm K$-invariant quasinorm $\vert\cdot\vert_i = \vert\cdot\vert_{\rho_{i}}$ on $W_i$ such that for any $w\in W_i, w\neq 0, g\in \mathrm G$, and $\beta > 0$,
  \begin{equation}
    \label{8.9}
    \int_{\mathrm K} \frac{\dif\sigma(k)}{\abs{\rho_i (gk)w}^\beta_{i}} \leq \tau_{\beta m(i)} (g) \, \frac{1}{ \abs{w}^\beta_i },
  \end{equation}
  where $\rho_i$ denotes the restriction of $\rho$ to $W_i$ and $m(i)\in I(\rho )$ is defined by the condition that $\psi_{m(i)}$ is isomorphic to 
$\rho_i$. Then we define $\vert w\vert = \vert w\vert_\rho$ by
  \begin{equation}
    \label{8.10}
    \vert w \vert = \max_{1\leq i\leq n} \vert\pi_i (w)\vert_i, \,\, w\in W.
  \end{equation}
  Clearly $\vert\cdot\vert_\rho$ is a $\mathrm K$-invariant quasinorm. Let us fix now $w\in W, w\neq 0$. Then
  \begin{align*}
    \int_{\mathrm K} \frac{\dif\sigma (k)}{\abs{\rho (gk)w}^\beta}
     & \le \min_{1\leq i \leq n} \int_{\mathrm  K} \frac{\dif\sigma (k)}{\abs{\pi_i (\rho (gk)w)}^\beta_i}
    = \min_{1\leq i \leq n} \int_{\mathrm K} \frac{\dif\sigma (k)}{\abs{\rho_i(gk)\pi_i(w)}^\beta_i}                                             
                \\
     & \leq \min_{1\leq i \leq n} \tau_{\beta m(i)} (g) \frac{1}{\abs{\pi_i(w)}^\beta_i} \le \max_{m\in I(\rho )} \{\tau_{\beta m}(g)\} \frac{1}{\abs{w}^\beta}.
  \end{align*}
  Thus, it is enough to prove the proposition for representations $\psi_m$. For this, let $P\in \mathcal{P}_m, P\neq 0$. We consider $P$ as a polynomial on $\C^2$ and decompose $P$, using the fundamental theorem of algebra, into the product of $m$ linear forms
  \begin{equation*}
    P = \ell_1\cdot\ldots\cdot\ell_m, \q \text{where} \q  \ell_i (z_1,z_2) = a_iz_1 + b_iz_2, \q a_i,b_i, z_1,z_2\in\C.
  \end{equation*}
  There is a natural $\mathrm K$-invariant norm on the space of linear forms on $\C^2$:
  \begin{equation*}
    \norm{\ell}^2 = \vert a\vert^2 + \vert b\vert^2, \,\, \ell (z_1,z_2) = az_1 + bz_2.
  \end{equation*}
  Now we define a quasinorm on $\mathcal{P}_m$ by the equation
  \begin{equation}
    \label{8.12}
    \abs{P} = \norm{\ell_1} \cdot \ldots\cdot \norm{\ell_m}.
  \end{equation}
  This definition is correct because the factorization \eqref{8.12} is unique up to the order of factors and the multiplication of $\ell_i$, $1\leq i\leq n$, by constants. We denote by $\tilde\psi_1$ the extension of $\psi_1$ to the space of linear forms on $G$. It is isomorphic to the standard representation of $\mathrm G$ on $\C^2$. Then using Lemma \ref{Lemma 8.5} and the generalized H\"older inequality, we get that
  \begin{equation}
    \label{8.13}
    \begin{aligned}
      \int_{\mathrm  K} \frac{\dif\sigma (k)}{\vert\psi_m(gk)P\vert^\beta}
      = \int_{\mathrm K} \frac{\dif\sigma (k)} {\prod_{i=1}^m \norm{\tilde\psi_1(gk) \ell_i}^\beta}
       & \leq \prod_{i=1}^m \left(\int_{\mathrm K} \frac{\dif\sigma (k)}{\Vert\tilde\psi_1 (gk)\ell_i\Vert^{\beta m}} \right)^{1/m}              
 \\
       & \leq \prod_{i=1}^m \left(\frac{\tau_{\beta m}(g)}{ \norm{\ell_i}^{\beta m}}\right)^{1/m} = \frac{\tau_{\beta m}(g)}{\vert P\vert^\beta}.
    \end{aligned}
  \end{equation}
  Since $I(\psi_m) = \{ m\}$, \eqref{8.13} implies \eqref{8.9} for $\rho = \psi_m$.
\end{proof}

We recall from Section \ref{s1}, see \eqref{7.13} and \eqref{7.14}, that $\tau_\mu (g) < 1$ and $\tau_\eta (g)< \tau_\lambda (g)$ for any $g\notin 
\mathrm  K$, $0 < \mu < 2$, $\lambda \geq 2$ and $0 < \eta < \lambda$. Using this, we deduce from the previous Proposition \ref{Proposition8.6} the following corollary.

\begin{corollary}
  \label{corollary8.7}
  Let $\rho$ be a representation of $\mathrm  G = \SL(2,\R)$ in a finite dimensional space $W$, and let $m$ be the largest number in $I(\rho )$. 
Then there exists a $\rho(\mathrm  K)$-invariant quasinorm $\vert\cdot\vert = \vert\cdot\vert_\rho$ on $W$ such that
  \begin{enumerate}[label=(\roman*)]
    \item if $\beta > 0$ and $\beta m\geq 2$ then for any $w\in W$, $w\neq 0$, and $g\in \mathrm  G$
          \begin{equation*}
            \int_{\mathrm  K} \frac{\dif\sigma (k)}{\vert\rho (gk)w\vert^\beta} \leq \tau_{\beta m} (g) \frac{1}{\vert w\vert^\beta},
          \end{equation*}
    \item if $\beta > 0$ and $\beta m < 2$ then for any $w\in W$, $w\neq 0$, and $g\in \mathrm  G$, $g\notin \mathrm  K$,
          \begin{equation*}
            \int_{\mathrm  K} \frac{\dif\sigma (k)}{\vert\rho (gk)w\vert^\beta} < \frac{1}{\vert w\vert^\beta}.
          \end{equation*}
  \end{enumerate}
\end{corollary}



\subsection{\texorpdfstring{Functions $\alpha_i$ on the Space of Lattices 
and Estimates for $A_h\alpha_i$}{Functions alpha on the Space of Lattices 
and corresponding Estimates}}
\label{s3}
Let $\rho$ be a representation of $\mathrm G = \SL(2,\R)$ on $\R^n$ and for each $1 \leq i \leq n$ let $| \, \cdot \, |_i$ be a $(\wedge^i\rho )(\mathrm  K)$-invariant quasinorm on the exterior product $\largewedge^i \R^n$. Throughout this section the underlying quasinorms in the definition of the lattice functions $\alpha_i$ and $\alpha$ are taken to be with respect to this particular choice of quasinorms (see \eqref{9.2} and \eqref{9.3}).
For every compact subset $A \subset \mathrm G$ note that
\begin{align*}
   & \sup\left\{\frac{\vert(\wedge^i\rho )(h)v\vert_i}{\vert v\vert_i} : h\in A, v\in \largewedge^i\R^n, v\neq 0\right\} \\
   & =\sup \{\vert (\wedge^i\rho )(h)v\vert_i : h\in A, v\in \largewedge^i \R^n, \vert v\vert_i = 1\}
\end{align*}
is finite for every $i$, $1\leq i\leq n$. Hence, if we fix $g\in \mathrm  G, g\notin \mathrm K$, then there exists some $B > 1$ such that for
any $i$, $1\leq i\leq n$, and $v \in \largewedge^i \R^n$, $v\neq 0$,
\begin{equation}
  \label{9.4}
  B^{-1} < \frac{\abs{(\wedge^i\rho )(y)v}_i}{\abs{v}_i} < B \q \text{if} 
\ y\in \mathrm  G \ \text{and} \ \norm{y} \leq \norm{g},
\end{equation}
where $\norm{h} = \norm{h^{-1}}$ denotes the norm of $h \in \mathrm  G = \SL(2,\R)$ with respect to the standard Euclidean norm on $\R^2$. Now, let $\Delta$ be a lattice in $\R^n$ and $L$ a $\Delta$-rational subspace. For any $h \in \SL(2,\R)$ observe that $h L$ is an $h\Delta$-rational subspace and if $v_1, \dots, v_i$ is a basis of $\Delta \cap L$ then $h v_1, \dots, h v_i$ is a basis of $h\Delta \cap h L$. This observation together with \eqref{9.4} implies that
\begin{equation}
  \label{9.5}
  B^{-1} < \frac{d_{y\Delta}(yL)}{d_\Delta (L)} < B \q \text{if} \ y\in \mathrm  G \ \text{and} \ \norm{y} \leq \norm{g}.
\end{equation}
Hence, for any $i \in \{ 0, \ldots, n\}$ it follows that
\begin{equation}
  \label{9.6}
  \alpha_i (y\Delta) < B \alpha_i (\Delta ) \q \text{if} \ y\in  \mathrm  
G \ \text{and} \ \norm{y} \leq \norm{g}.
\end{equation}
For any $\beta > 0$ and $1\leq i\leq n$ we define the functions $F_{i,\beta}$ on $\largewedge^i \R^n \setminus \{ 0\}$ by
\begin{equation*}
  F_{i,\beta} (w) \defi \int_K \frac{\abs{w}_i^\beta}{\abs{ (\wedge^i\rho 
)(gk)w}_i^\beta} \, \dif \sigma (k), \quad w\in \wedge^i\R^n, w\neq 0.
\end{equation*}
It is clear that the functions $F_{i,\beta}$ are continuous and that $F_{i,\beta}(\lambda w) = F_{i,\beta}(w)$ for any $\lambda\in\R$, $\lambda\neq 0$. Let $c_{0,\beta}\tdefi 1$ and for $1 \leq i \leq n$
\begin{equation}
  \label{9.8}
  c_{i,\beta} \defi \sup\{ F_{i,\beta}(w): w\in\largewedge^i\R^n, w\neq 0\} = \sup\{ F_{i,\beta} (w): w\in\largewedge^i\R^n, \vert w\vert_i = 1\}.
\end{equation}
We note that $c_{n,\beta} = 1$, since the image of any continuous homomorphism $\SL(2,\R) \rightarrow \GL(n,\R)$ is contained in $\SL(n,\R)$ and thus $\abs{ (\wedge^n\rho )(gk)w}_n = \abs{\det(\wedge^n\rho (gk))} \abs{w}_n = \abs{w}_n$.

\begin{lemma}
  \label{Lemma 9.7}
  For any $i$, $0\leq i\leq n$,
  \begin{equation}
    \label{9.9}
    A_g \alpha^\beta_i\leq c_{i,\beta} \alpha^\beta_i + C^\beta B^{2\beta} \max_{0 < j\leq\bar i} \sqrt{\alpha^\beta_{i-j} \alpha^\beta_{i+j}},
  \end{equation}
  where $\bar{i} = \min\{ i,n-i\}$, the constant $C\geq 1$ is from Lemma \ref{Lemma 9.6} and the operator $A_g$ is defined by \eqref{7.4} from Section \ref{s1}.
\end{lemma}

\begin{proof}
  Let $\Delta$ be a lattice in $\R^n$. We have to prove that
  \begin{equation}
    \label{9.10}
    \int_{\mathrm K} \alpha_i (gk\Delta)^\beta \, \dif\sigma (k) \leq c_{i,\beta} \alpha_i(\Delta)^\beta +C^\beta B^{2\beta} \max_{0<j\leq\bar i} \sqrt{\alpha_{i-j} (\Delta )^\beta \alpha_{i+j} (\Delta )^\beta}.
  \end{equation}
  According to Remark \ref{remark_existence_alpha_l} there exists a $\Delta$-rational subspace $L$ of dimension $i$ such that
  \begin{equation}
    \label{9.11}
    \frac{1}{d_\Delta (L)} = \alpha_i (\Delta ).
  \end{equation}
  Let us denote the set of $\Delta$-rational subspaces $M$ of dimension $i$ with $d_\Delta (M) < B^2 d_\Delta (L)$ by $\Psi_i$. For a $\Delta$-rational $i$-dimensional subspace $M\notin \Psi_i$ we get from \eqref{9.5} that
  \begin{equation*}
    d_{gk\Delta} (gkM) > d_{gk\Delta} (gkL).
  \end{equation*}
  If $\Psi_i = \{ L\}$, then it follows from this and the definitions of $\alpha_i$ and $c_{i,\beta}$ that
  \begin{equation}
    \label{9.12}
    \int_{\mathrm  K} \alpha_i (gk\Delta )^\beta \, \dif\sigma (k) \leq c_{i,\beta} \alpha_i (\Delta )^\beta.
  \end{equation}
  Assume now that $\Psi_i\neq\{ L\}$. Let $M\in\Psi_i$, $M\neq L$. Then $\dim (M+L) = i+j$, $0 < j\leq\bar i$. Now we obtain by \eqref{9.5}, \eqref{9.11} and Lemma \ref{Lemma 9.6} for any $k\in K$ that
  \begin{align*}
    \alpha_i (gk\Delta ) < B \alpha_i (\Delta ) = \frac{B}{d_\Delta (L)} \leq \frac{B^2}{\sqrt{d_\Delta (L)d_\Delta (M)}} & \leq \frac{CB^2}{\sqrt{d_\Delta (L\cap M)d_\Delta (L+M)}} \\  & \leq CB^2 \sqrt{\alpha_{i-j} (\Delta ) \alpha_{i+j} (\Delta)}.
  \end{align*}
  Hence, if $\Psi_i\neq \{ L\}$,
  \begin{equation}
    \label{9.14}
    \int_{\mathrm  K} \alpha_i (gk\Delta )^\beta \, \dif\sigma (k) \leq C^\beta B^{2\beta} \max_{0<j\leq\bar i} \sqrt{\alpha_{i-j} (\Delta )^\beta 
\alpha_{i+j}(\Delta )^\beta}.
  \end{equation}
  Combining \eqref{9.12} and \eqref{9.14}, we get \eqref{9.10}.
\end{proof}

\begin{theorem}
  \label{GM}
  Let $d\in\N^+$ and let $\rho_d$ be a representation of $\mathrm G =\SL(2,\R)$ isomorphic to the direct sum of $d$ copies of the standard $2$-dimensional representation. Let $\beta$ be a positive number such that $\beta d > 2$. Then there is a constant $R$, depending only on $\beta$ and the choice of the $\mathrm  K$-invariant quasinorms $| \, \cdot \, |_i$ involved in the definition of $\alpha_i$, such that for any $h\in \mathrm  G$ and any lattice $\Delta$ in $\R^{2d}$
  \begin{equation*}
    (A_h \alpha^\beta)(\Delta ) = \int_{\mathrm K} \alpha (hk\Delta )^\beta \, \dif\sigma (k)
    \leq R \4 \alpha (\Delta )^\beta \Vert h\Vert^{\beta d-2}.
  \end{equation*}
\end{theorem}

\begin{proof}
  As in Section \ref{s2}, we define for a finite dimensional representation $\rho$ of $\mathrm  G$
  \begin{equation*}
    I(\rho ) =\{ m\in\N^+: \psi_m \,\, \text{is isomorphic to a subrepresentation of} \,\, \rho\},
  \end{equation*}
  where $\psi_m$ denotes the regular representation of $\mathrm  G$ in the space of complex homogeneous polynomials in two variables homogeneous of degree $m$. Let $m_i$ be the largest number in $I(\wedge^i\rho_d)$, $1\leq i\leq 2d$. It is well known that
  \begin{equation}
    \label{9.16}
    m_i = \bar{i} \defi \min\{ i, 2d-i\}.
  \end{equation}
  We fix $g\in \mathrm  G, g\notin \mathrm K$. It follows from \eqref{9.16} and from  Corollary \ref{corollary8.7} 
  that we can choose quasinorms $\vert\cdot\vert_i$ on $\largewedge^i\R^{2d}$ in such a way that for $w\in\largewedge^i \R^{2d}$, $w\neq 0$,
  \begin{equation*}
    \int_{\mathrm K} \frac{\vert w\vert^\beta_i}{\vert (\wedge^i \rho_d)(g)w\vert^\beta_i} \, \dif \sigma (k) \leq \begin{cases} \tau_{\beta\bar i}(g) & \text{if} \ \beta\bar i\geq 2 \\ 1 & \text{if} \ \beta\bar i < 2.\end{cases}
  \end{equation*}
  Hence
  \begin{equation}
    \label{9.19}
    c_{i,\beta} \leq \tau_{\beta \bar{i}} (g) \q \text{if} \ \ \beta \bar{i} \geq 2 \q \text{and} \q  c_{i,\beta}\leq 1 \q \text{if} \ \ \beta\bar{i} < 2.
  \end{equation}
  where $c_{i,\beta}$, $1\leq i\leq 2d$, is defined by \eqref{9.8} and $c_{0,\beta} = 1$. As a remark, we notice that $c_{i,\beta} = \tau_{\beta\bar i} (g)$ if $\beta \bar{i} \geq 2$.\par According to Lemma \ref{Lemma 9.7}, the functions $\alpha^\beta_i$, $ 0\leq i\leq 2d$, satisfy the following system of inequalities
  \begin{equation}
    \label{9.21}
    A_g \alpha^\beta_i \leq c_{i,\beta} \alpha^\beta_i + C^\beta B^{2\beta} \max_{0 < j\leq\bar i} \sqrt{\alpha^\beta_{i-j} \alpha^\beta_{i+j}}.
  \end{equation}
  Let
  \begin{equation}
    \label{9.22}
    \lambda_i \defi \max\{ 2, \beta \bar{i}\}, \q  0 \leq i \leq 2d.
  \end{equation}
  Since $\tau_2(g) = 1$, see \eqref{7.11} in Section \ref{s1}, it follows from \eqref{9.19}-\eqref{9.22} that
  \begin{equation}
    \label{9.23}
    A_g \alpha^\beta_i\leq\tau_{\lambda_{i}}(g) \alpha^\beta_i + C^\beta B^{2\beta} \max_{0<j\leq\bar i}\sqrt{\alpha^\beta_{i-j}\alpha^\beta_{i+j}}, \q 0\leq i\leq 2d.
  \end{equation}
  Now we fix a lattice $\Delta$ in $\R^{2d}$ and define functions $f_i$, $ 0\leq i\leq 2d$, on $\mathrm G$ by
  \begin{equation*}
    f_i (h) = \alpha_i (h\Delta )^\beta, \q h\in \mathrm  G.
  \end{equation*}
  Then it follows from \eqref{9.23} that
  \begin{equation*}
    A_g f_i\leq \tau_{\lambda_{i}} (g) f_i + C^\beta B^{2\beta} \max_{0 < 
j\leq\bar i} \sqrt{f_{i-j}f_{i+j}}, \q 0\leq i\leq 2d.
  \end{equation*}
  On the other hand, in view of \eqref{9.6},
  \begin{equation*}
    f_i(yh) \leq B^\beta f_i(h), \q \text{if} \q 0\leq i\leq 2d,\; h,y \in \mathrm G\q \text{and} \q \norm{y} \leq \norm{g}.
  \end{equation*}
  Since $\beta d > 2$, we have that $\beta d = \lambda_d > \lambda_i$ for any $i\neq d$. Now we can apply Proposition \ref{Prop-7.1} (c) in order to get that
  \begin{equation}
    \label{9.27}
    \begin{aligned}
      (A_h \alpha^\beta )(\Delta ) < (A_h \sum_{0\leq i\leq 2d} \alpha^\beta_i)(\Delta) & = (A_h \sum_{0\leq i\leq 2d} f_i)(1) \ll  (\sum_{0\leq 
i\leq 2d} f_i (1)) \4 \Vert h\Vert^{\lambda_{d}-2}                        
          \\
& = (\sum_{0\leq i\leq 2d} \alpha_i(\Delta )^\beta )\Vert 
h \Vert^{\lambda_{d}-2} \leq 2d \alpha (\Delta )^\beta\4 \Vert h\Vert^{\beta d-2}.
    \end{aligned}
  \end{equation}
  The inequality \eqref{9.27} proves the theorem for our specific choice of the quasinorms $\vert\cdot\vert_i$. Now it remains to notice that any two quasinorms on $\largewedge^i \R^n$ are equivalent.
\end{proof}




\section{Proofs of Theorem \ref{maintheorem} and Theorem \ref{th:gaussian_weights}}
\label{conclus}
\noindent In this section we shall prove our main theorem, giving effective estimates on the lattice remainder. But, before doing this, we have to establish mean-value estimates for the $\alpha_d$-characeristics of $\Lambda_t$ by applying Theorem \ref{GM} combined with Lemma \ref{special_int}.

\begin{corollary}
  \label{special_int1}
  Let $r \ge q^{1/2}$, $I=[t_0,t_0+1]$ with $t_0 \in \R$, $0< \beta \le 1/2$ with $\beta d >2$ and $\wh{g}_I \tdefi \max \{\abs{\wh{g}_w(t)}\, :\, t \in I\}$. Using the notation \eqref{defgamma}, we have
  \begin{equation}
    \label{eq:alpha-int2}
    \int_I  \alpha_d(\Lambda_t)^{1/2} \4 \abs{\wh{g}_w(t)} \, \dif t \ll_{\beta,d} q \4 \abs{\det{Q}}^{-\beta/2} \4 \wh{g}_I \4 \gamma_{I, \beta}(r) \4  r^{\frac{d}{2}-2},
  \end{equation}
  where $\gamma_{I, \beta}(r) =1$ if $\beta =1/2$. Note that we need at least $d \ge 5$.
\end{corollary}

Based on our variant of Weyl's inequality (see Lemma \ref{thetaestimate2} and Corollary \ref{alphamax}) the $\alpha$-characteristic enters with a power $1/2$ in \eqref{eq:alpha-int2}. While saving a maximum of the $\alpha$-characeristic, it will enter still with an exponent $0 < \beta \leq 1/2$ for its average (compare Lemma \ref{special_int}). Since the crucial averaging recursion (Theorem \ref{GM}) fails unless $\beta d >2$, the proof essentially needs $d >4$ and thus $d \geq 5$.

\begin{proof}
  In order to apply Lemma \ref{special_int}, we cover $I$ by intervals $I_j = [s_j,s_{j+1}]$ of length at most $1/q$, where $s_j = t_0 + j/q$ with $j \in J \tdefi \{0, \ldots, \lceil q \rceil \}$. This implies
  \begin{equation}
    \label{eq:recall:cor6.1}
    \begin{aligned}
      \int_I  \alpha_d(\Lambda_t)^{1/2} \4 \abs{\wh{g}_w(t)} \, \dif t
       & \le r^{\frac{d}{2}-\beta d} \4 \wh{g}_I  \4 \gamma_{I,\beta}(r) \frac{1}{q} \sum_{j \in J} \int_{-\pi}^\pi  \alpha(d_{r_*} \4 k_\theta \4 
\Lambda_{Q,s_j})^{\beta}\4 \frac{\dif \theta}{2 \4 \pi} \\
       & \ll r^{\frac{d}{2}-\beta d} \4  \wh{g}_I  \4 \gamma_{I,\beta}(r) 
\max_{j \in J} \int_{-\pi}^\pi  \alpha(d_{r_*} \4 k_\theta \4 \Lambda_{Q,s_j})^{\beta}\4 \frac{\dif \theta}{2 \4 \pi}.
    \end{aligned}
  \end{equation}
  Now, we shall apply Theorem \ref{GM} with $h=d_{r_*}$, $r_*=r/q^{1/2}$ and the lattices $\Lambda_{Q,s_j} = d_{q^{1/2}} \4 u_{s_j} \Lambda_Q$, as defined in \eqref{latticelambdaqt}, and obtain
  \begin{equation*}
    \max_{j\in J}\int_{-\pi}^{\pi} \alpha(d_{r_*} \4 k_\theta \4 \Lambda_{Q,s_j})^{\beta}\4 \frac{d \theta}{2 \4 \pi} \ll_{\beta,d}  \max_{j\in J} 
\alpha(\Lambda_{Q,s_j} )^\beta \norm{d_{r_*}}^{\beta d-2} \ll_d r^{\beta d-2}\4\big(q \4 d_Q^\beta \big),
  \end{equation*}
  where we have used $\norm{d_{r_*}}= r_* = r/q^{1/2}$ and \eqref{apriori-alpha01} in form of
  \begin{equation*}
    \alpha(\Lambda_{Q,s_j} ) \ll_d \alpha_d (\Lambda_{Q,s_j} ) \ll_d \abs{\det Q}^{-1/2}\4 q^{d/2}.
  \end{equation*}
  Note that we have applied Corollary \ref{alphamax} with $r = q^{1/2}$ 
and $t= s_j$ in order to get $\alpha(\Lambda_{Q,s_j} ) \asymp_d\alpha_d(\Lambda_{Q,s_j})$. Finally, in view of \eqref{eq:recall:cor6.1}, this concludes the proof of \eqref{eq:alpha-int2}.
\end{proof}

In order to bound the lattice point remainder for `wide shells', that is $b-a> q^{1/2}$, we need to extend the averaging result, established in Corollary \ref{special_int1}, for small values of $t$. To do this, we recall the bound
\begin{equation}
  \label{gw}
  \abs{\wh{g}_w(t)} \ll \min\{ \abs{b-a}, \abs{t}^{-1}\} \4\exp \{ - \abs{t \4 w}^{1/2}\}
\end{equation}
for the integrand $\wh{g}_w(t)$ in \eqref{eq:alpha_int1}, provided that $0 < w < (b-a)/4$. Note that it is of size $b-a$ for $\abs{t} \le 1/(b-a)$ 
and changes rapidly if $\abs{b-a}>1 $ grows with $r$.

\begin{lemma}
  \label{special-int2}
  If $r \ge q^{1/2}$, $\beta d >2$ and $0< w < \abs{b-a}/4$, then
  \begin{equation}
    \label{IJ56}
    \int_{\imbound}^{q^{-1/2}} \alpha_d(\Lambda_t)^{1/2} \4 \abs{\wh{g}_w(t)} \, \dif t \ll_{\beta,d} q^{\beta d+1/2} \4 \abs{\det{Q}}^{-\beta/2} \4 \gamma_{I,\beta}(r) \4 r^{\frac{d}{2}-2},
  \end{equation}
  where $I = [\imbound, q^{-1/2}]$.
\end{lemma}

\begin{proof}
  Proceeding first as in the proof of Lemma \ref{special_int} and changing variables to $s= t^{-1}$ it is plain to see that
  \begin{align*}
    \int_{\imbound}^{q^{-1/2}} \alpha_d(\Lambda_t)^{1/2} \4 \abs{\wh{g}_w(t)} \, \dif t \ll_d \gamma_{I,\beta}(r)\,r^{d/2-\beta d} \int_{q^{1/2}}^{r q_0 ^{1/2}}  \alpha_{d} (d_r\4 u_{4 s^{-1}}\4\Lambda_Q)^{\beta} \4 \abs{\wh{g}_w(s^{-1})}  \, \frac{\dif s}{s^2}.  \end{align*}
  Let $N = \lceil r(q_0/q)^{1/2} \rceil$, then the integral on the right-hand side is bounded by $\sum_{j=2} ^N I_j$, where
  \begin{align*}
    I_j \defi \int_{q^{1/2}(j-1)} ^{q^{1/2}j} \alpha_d(d_r u_{4s^{-1}}\Lambda_Q)^\beta  \4 \abs{\wh{g}_w(s^{-1})}  \, \frac{\dif s}{s^2}.
  \end{align*}
  For $2 \leq j \leq N$ write $t_j = q^{-1/2}j^{-1}$, then using that
  \begin{align*}
    d_r u_{4s^{-1}} = d_r \4 u_{4(s^{-1} - t_j)} \4 u_{4 t_j} = d_{4 r j^{-1}} \4 u_{4^{-1}j^2(s^{-1} - t_j)}\4 d_{4^{-1} j} \4 u_{4 t_j}
  \end{align*}
  together with the change of variables $v = 4^{-1} j^2 (s^{-1}- t_j)$ yields
  \begin{align*}
    I_j & \leq \frac{4}{j^2} \int_0 ^{1} \alpha_d(d_{4 r  j^{-1}}  \4 u_{v} \4 d_{4^{-1} j} \4 u_{4t_j} \Lambda_Q )^\beta \4 \abs{\wh{g}_w(4 v j^{-2}+t_j)}  \,  \dif v \\ & \ll_d
    \frac{q^{1/2}}{j} \int_0 ^{1} \alpha_d(d_{4 r  j^{-1}}  \4 u_{v} \4 d_{4^{-1} j} \4 u_{4t_j} \Lambda_Q )^\beta \4 \dif v,
  \end{align*}
  where the last inequality is a consequence of $\abs{\wh{g}_w(t)} \ll \abs{t}^{-1}$. Hence, since $ 4 rj^{-1} \geq 1$ and $q^{1/2}j t_j =1$, we deduce from Lemma \ref{compact-group}, Theorem \ref{GM} and \eqref{apriori-alpha2} of Lemma \ref{apriori} that
  \begin{align*}
    I_j & \ll_d \frac{q^{1/2}}{j} \int_{\mathrm K} \alpha_d (d_{4 r  j^{-1}}  \4 k \4 d_{4^{-1} j} \4 u_{4t_j} \Lambda_Q )^\beta \4 \dif \sigma(k) 
\\ & \ll_d r^{\beta d -2}  \abs{\det Q}^{-\beta/2} q^{\beta d /2+ 1/2} j^{1-\beta d }  \max \{1 ,(4q^{1/2} j^{-1})^{\beta d} \}.
  \end{align*}
  Summing the last inequality over $2 \leq j \leq N$, we observe that it suffices to show that the following estimate holds
  \begin{align*}
    \textstyle
    \sum_{j=2} ^N j^{1-\beta d }  \max \{1 ,(4q^{1/2} j^{-1})^{\beta d} 
\} \ll_{\beta,d}r^{\beta d -2}  \abs{\det Q}^{-\beta/2} q^{\beta d+ 1/2}.
  \end{align*}
  Indeed, split the previous sum according to whether $j \le 4q^{1/2}$ or 
$j > 4q^{1/2}$. The sum over $j >  4 q^{1/2}$ can be bounded by
  \begin{equation*}
    \textstyle
    \label{eq:L6.2:1}
    r^{\beta d -2}  \abs{\det Q}^{-\beta/2} q^{\beta d /2+ 1/2} \sum_{j = 
\lceil 4q^{1/2} \rceil} ^N j^{1-\beta d }  \ll_{\beta,d} r^{\beta d -2} 
 \abs{\det Q}^{-\beta/2} \4 q^{3/2},
  \end{equation*}
  and the sum over $2 \leq j \le 4 q^{1/2} $ by
  \begin{equation*}
    \label{eq:L6.2:2}
    \textstyle
    r^{\beta d -2}  \abs{\det Q}^{-\beta/2} q^{\beta d+ 1/2} \sum_{j = 2} ^{\lfloor 4 q^{1/2} \rfloor}  j^{1-2\beta d }  \ll_{\beta,d} r^{\beta d -2}  \abs{\det Q}^{-\beta/2} q^{\beta d+ 1/2}.\qedhere
  \end{equation*}
\end{proof}




\begin{proof}[Proof of Theorem \ref{maintheorem}]
 In view of \eqref{crucialbound}, it remains to estimate $I_\theta$. By \eqref{eq:5.1:intro}, with $K_0 \tdefi [\imbound,1]$ and $K_j \tdefi (j,j+1]$, $j \ge 1$, we have
\begin{equation}
  \label{final_I}
  I_\theta \ll_d \abs{\det{Q}}^{-\frac{1}{4}} \norm{\wh{\zeta}}_1 \Big( I_{\theta,0} + \sum_{j=1}^{\infty} I_{\theta,j}\Big), \ \  \text{where} \ \  I_{\theta,j} \defi \int_{K_j} \abs{\wh{g}_w(t)} \4 \alpha_d(\Lambda_t)^{\frac{1}{2}} \4 \dif t.
\end{equation}
For fixed $r \geq q^{1/2}$ we may choose
\begin{equation}
  \label{choice_wab}
  0< w < (b-a)/4, \q 1 \ge T_{-} \ge \imbound, \q T_+ \geq 1  \q \text{and} \q \frac{d}{2} > \beta d  >2.
\end{equation}
For notational simplicity, we write $C_Q \tdefi q \4 \abs{\det{Q}}^{-1/4-\beta/2}$.\\[2mm]
\textit{Step 1: Estimate of $I_{\theta,0}$.} We consider the case $b-a \le q$ first. Here we apply Corollary \ref{special_int1} to bound the integral over $K_0$ combined with $\wh{g}_{K_0} \ll s_{[a,b]_{\pm w}}(t) \ll b-a$, compare \eqref{h-estimate1} and \eqref{h-estimate2}. Note that we didn't use the restriction $b-a \le q$ at all. For wide shells, i.e.\ in the case $b-a > q$, we use Lemma \ref{special-int2} for $t\in K_0$, $\imbound \le \abs{t} \le q^{-1/2}$  and Corollary \ref{special_int1} for the other $t$ in $K_0$ together with $\wh{g}_{[q^{-1/2},1]} \ll q^{1/2}$. Furthermore, for both cases of $b-a$, split $K_0=K_{00} \cup K_{01}$, where $K_{00} \tdefi [\imbound, T_{-}]$ and $K_{01}\tdefi 
(T_{-},1]$. Then \eqref{apriori-alpha1} of Lemma \ref{apriori} yields
\begin{equation}
  \label{K0}
  \gamma_{K_{00},\beta}(r) \ll_d \bigl(\abs{\det Q}^{\frac{1}{2}} \, T_{-}^{d}\bigr)^{\frac{1}{2}-\beta} = T_{-}^{\frac{d}{2}-2-\delta}\4 \abs{\det{Q}}^{ \frac{1}{4} -\frac{\beta}{2} },
\end{equation}
with the notation \eqref{defgamma}. Using $C_Q q^{(2\beta d -1)/2} =\bar{C}_Q$, we may bound $I_{\theta,0}$ as\index{1@ $(b-a)_q := (b-a)I(b-a 
\le q) + q^{(2\beta d -1)/2} I(b-a > q)$}
\begin{align}
  \label{final_I1}
  I_{\theta,0} & \ll_d C_Q\4 (b-a)_q \4
  \bigl(\abs{\det{Q}}^{ \frac{1}{4} -\frac{\beta}{2}} \4 T_{-}^{\frac{d}{2}-2-\delta} + \gamma_{K_{01},\beta}(r)\bigr) \4 r^{d-2}, \q\text{where} \\
  \label{defab}
  (b-a)_q      & \defi (b-a)I(b-a \le q) + q^{(2\beta d -1)/2} I(b-a > q).
\end{align}
As a side remark, we note that the above splitting of the interval $K_0 = [\imbound,1]$ is required for our later applications - especially, Corollary \ref{irr-dio} is only valid for fixed intervals $[T_{-},T_{+}]$.\\[2mm]
\textit{Step 2: Estimate of $I_{\theta,j}$ for $j \ge 1$.} Similar as before, applying Corollary \ref{special_int1} (with $\beta=1/2$), while noting that $\gamma_{I, \beta}(r) =1$ if $\beta =1/2$, yields 
\begin{equation}
  \label{final_IJ}
  I_{\theta,j} \ll_d \wh{g}_{K_j} \4 q \4 \abs{\det{Q}}^{-1/2} \4 r^{d-2}.
\end{equation}
We recall the bound \eqref{gw} for $\wh{g}_w$ and the choices of $T_+$ and $w$ in \eqref{choice_wab} in order to get
\begin{equation*}
  \sum_{j=T_+}^{\infty} \wh{g}_{K_j} \ll \int_{T_+}^\infty  \frac{\exp\{-\abs{sw}^{1/2}\}}{s} \, \dif s \ll \frac{1}{\sqrt{T_+w}} \exp \{ -\abs{T_+ \4 w }^{1/2} \}.
\end{equation*}
Thus, we obtain
\begin{equation}
  \textstyle
  \label{final_BigT}
  \sum_{j = T_+}^\infty I_{\theta,j} \ll_d  r^{d-2} \4 q \4 \abs{\det{Q}}^{-1/2} \4 (T_+ \4 w)^{-1/2} \exp \{ -\abs{T_+ \4 w}^{1/2}\}.
\end{equation}
Furthermore, for $b-a > 1$ we can use $\abs{\wh{g}_{K_j}}\ll j^{-1}$ to bound the remaining sum. Whereas for $b-a \le 1$ we use $\abs{\wh{g}_{K_j}} \ll b-a$ for $1 \leq j \leq S-1$ and $\abs{\wh{g}_{K_j}}\ll j^{-1}$ for 
$S \leq j \leq T_+-1$ and minimize the resulting expression in $S$. In both cases this leads to
\begin{equation}
  \textstyle
  \label{final_SmallT}
  \sum_{j=1}^{T_+-1} \wh{g}_{K_j} \ll 1+ \log ( (b-a)^* \4 T_+),
\end{equation}
where\index{1@ $(b-a)^* \defi (b-a)I(b-a \le 1) + I(b-a > 1)$}
\begin{equation*}
  (b-a)^* \defi (b-a)I(b-a \le 1) + I(b-a > 1).
\end{equation*}
Hence, using \eqref{final_I} combined with \eqref{final_I1}, \eqref{final_BigT} and \eqref{final_SmallT} with \eqref{final_IJ}, we get
\begin{equation}
  \label{t_rest}
  \begin{aligned}
    I_\theta \ll_d \norm{\wh \zeta}_1\4 r^{d-2}\4 C_Q\4\Bigl((b & -a)_q\4 
( c_Q T_{-}^{\frac{d}{2}-2-\delta}+ \gamma_{[T_{-},1],\beta}(r)) \\
& +\gamma_{(1,T_+],\beta}(r) \4 (1+ \log((b-a)^* \4 T_+)) + c_Q^{-1} \4 \tfrac{\exp(- (T_+w)^{1/2})}{(T_+ \4 w)^{1/2}} \Bigr),
  \end{aligned}
\end{equation}
where $c_Q = \abs{\det{Q}}^{\frac{1}{4}-\frac{\beta}{2}}$. Together with the inequality \eqref{crucialbound} we obtain
\begin{equation}
  \label{finalb0}
  \begin{aligned}
    \Delta_r(\specialv) & \defi \ \Big\lvert \sum_{m\in\Z^d} I_{[a,b]}(Q[m]) \specialv_r (m) - \int_{\R^d} I_{[a,b]}(Q[x]) \specialv_r (x) \, \dif x \Big\rvert                                                              \\
& \ \ll_{\beta,d} \4 r^{d-2} \big( \4 \norm{\wh \zeta}_1\4 C_Q \4 \rho_{Q,b-a}^{w}(r) + w \4 \norm{\specialv}_Q \4 \big) \hspace{-1pt} + d_Q \4 r^{d/2}\4 \norm{\wh{\zeta}}_{*,r} \log \hspace{-1pt} 
\Big(1 \hspace{-1pt}+ \hspace{-1pt} \tfrac{\abs{b-a}}{q_0^{1/2} r}\Big),
  \end{aligned}
\end{equation}
where\index{R@ $\rho_{Q,b-a}^{w}(r)$}
\begin{equation*}
  \begin{aligned}
    \rho_{Q,b-a}^{w}(r) \! \defi \! \inf \big\{ (b-a)_q \4 ( c_Q T_{-}^{\frac{d}{2}-2-\delta} \! + \gamma_{[T_{-},1],\beta}(r)) \! + \gamma_{(1,T_+],\beta}(r) \4 (1+\log((b-a)^* \4 T_+)) \  &        \\
    + c_Q^{-1} \4 (T_+ w)^{-1/2} \4 \e^{- (T_+w)^{1/2}} : T_{-} \in [\imbound,1] , \ T_+ \geq 1 & \big\}
  \end{aligned}
\end{equation*}
under the condition $0 < w < (b-a)/4$. This completes the proof of Theorem \ref{maintheorem}.
\end{proof}

\begin{proof}[Proof of Theorem \ref{th:gaussian_weights}]
 We have only to apply Theorem \ref{maintheorem} to the Gaussian weights $\specialv(x) = \exp\{-2 \4 Q_{+}[x]\}$  noting that $\zeta(x) = \exp 
\{ - Q_{+}[x] \}$ satisfies the integrability condition \eqref{zeta-decay}. This yields
 \begin{equation*}
   R(\specialv_r I_{E_{a,b}}) \ll_{Q,\beta,d} \big\{ w \4\norm{\specialv}_Q + \norm{\wh\zeta}_1 \4 \rho_{Q,b-a}^{w}(r)\big\} r^{d-2} r^{d/2} \norm{\hat \zeta}_{*,r} \log\big(1+ \tfrac{\abs{b-a}} {q_0^{1/2} r}\big).
 \end{equation*}
 In view of \eqref{def:v_Q:2} and \eqref{def:v_Q:1}, we see that $\norm{\specialv}_Q \ll d_Q$. Here we used that $\varphi_{\specialv}(v,\sqrt{u^2-v}) = \exp\{-2u^2\}$, if $Q$ is indefinite; and $\varphi_{\specialv}(v) 
= \exp\{-2v^2\}$ if $Q$ is positive definite. Moreover, a simple calculation shows that $\norm{\wh\zeta}_1 \ll_d 1$ and by following the arguments in the proof of \eqref{zetaperiod-notadmiss} we get $\norm{\hat \zeta}_{*,r} \ll_{d} q^{d/4} ( (q/q_0)^{d/2} + d_Q q^{d/2})$ as well.
\end{proof}




\section{Lattice Point Deficiency for Admissible Regions and Applications}
\label{section:parallelepipeds}
\noindent Before we can apply Theorem \ref{maintheorem}, we have to construct smooth bump functions, approximating the indicator function of special parallelepiped regions, and also to control the additional error produced by this smoothing step: In the following Lemma \ref{l3} we shall bound the volume of $\eps$-boundaries of $r\Omega\cap E_{a,b}$ and in Lemma \ref{l2} we estimate integrals of the Fourier transform of the region $\Omega$. For wide shells the lattice point counting remainders will reflect the Diophantine properties of $Q$ more directly when using counting regions $\Omega$ which are `admissible' convex polyhedra. 

\subsection{Smoothing of Special Parallelepiped Regions}
\label{subsection:smoothing_regions}
Here we confine ourselves to study a specially oriented parallelepiped $\Omega = B^{-1} [-1,1]^d$ with
\begin{equation}
  \label{eq:sec7:Def:Omega}
  \quad Q_+ \leq B^T B \leq c_B Q_+
\end{equation}
for a suitable $B \in \GL(d,\R)$ and a positive constant $c_B \geq 1$ depending on $B$. In this case, the Minkowski functional of $\Omega$ is given by $M(x)= \max( \langle g_{i,\pm} ,x \rangle \4 : \4 i=1,\ldots,d \4 )$, where $g_{i,\pm} = \pm B^T e_i$ are $2d$ outward normal vectors of the faces of $\Omega$. Note that the inequalities in \eqref{eq:sec7:Def:Omega} imply the norm equivalence
\begin{equation}
  \label{M-bound}
  d^{-1/2} \4 \norm{Q_+^{1/2}x} \le M(x) \le  (c_B)^{1/2} \4 \norm{Q_+^{1/2}x}.
\end{equation}
We now approximate $I_{\Omega}$ by smooth weight functions. For this, introduce\index{O@ $\Omega_{\pm \eps}:= (1\pm\eps)\Omega$, $\epsilon$-thickening resp. thinning of $\Omega$}\index{O@ $(\partial\Omega)_{\eps}:= \Omega_{\eps}\setminus \Omega_{-\eps}$, $\epsilon$-thickened boundary of $\Omega$}\index{V@ $\specialv_{\pm\eps} := I_{\Omega_{\pm\eps}}*k_{B,\eps}$, $\epsilon$-smoothing of $\Omega$}
\begin{equation}
  \label{omega_eps}
  \Omega_{\pm \eps}\defi (1\pm\eps)\Omega, \q (\partial\Omega)_{\eps}\defi \Omega_{\eps}\setminus \Omega_{-\eps} \q \text{and}\q \specialv_{\pm\eps} \defi I_{\Omega_{\pm\eps}}*k_{B,\eps},
\end{equation}
where $k_{B,\eps}(A) = k_{\eps}(BA)$ for any $A\in\mathcal{B}^d$ and $k_{\eps}$ denotes the rescaled measure on $\R^d$ introduced in the beginning of Subsection \ref{subsection:3.1}. Moreover, we need the technical restriction $0 < \eps \leq \eps_0$ with $\eps_0 \tdefi 1/15$. Since Lemma \ref{lemma:1} can be adapted to this situation, taking $\specialv_{\pm\eps,r}(x) \tdefi \specialv_{\pm\eps}(x/r)$, we get for the lattice point remainder\index{R@ $R(I_{E_{a,b}} \4 \specialv_r), R(g \4 \specialv_r),R( I_{E_{a,b} \cap r\Omega})$, lattice point remainder} \eqref{smooth-remainder}
\begin{equation}
  \label{Omega-smooth0}
  \abs{R( I_{E_{a,b} \cap r\Omega})} \leq \max_{\pm} \abs{R(I_{E_{a,b}} \specialv_{\pm\eps,r})} +R_{\eps,r},
\end{equation}
where, in view of \eqref{smooth-basic}, the remainder term is given by
\begin{equation}
  \label{smooth-error}
  R_{\eps,r} \defi \int_{\R^d} I_{(\partial\Omega)_{2\eps}}(x /r) \4 I_{[a,b]}(Q[x]) \, \dif x.
\end{equation}
For hyperbolic shells the latter term \eqref{smooth-error} will be absent, but for elliptic shells we shall find that
\begin{equation}
  \label{omega-smooth}
  \abs{R( I_{E_{a,b} \cap r\Omega})} \le \max_{\pm} \abs{R( I_{E_{a,b}} \4 \specialv_{\pm \eps,r})} + d_Q \4 (b-a)\4 \eps\4 r^{d-2}.
\end{equation}
This estimate will be proven in the following Lemma \ref{l3}, but first we need to introduce some notations: For a measurable, non-negative, bounded weight function $\specialv$ on $\R^d$ we shall define 
the spherical mean by
\begin{equation}
    \label{eq:spherical_mean_v}
    \varphi_{\specialv}(r_1,r_2) \defi \int_{S^{p-1}\times S^{q-1}} \specialv(Q_+^{-1/2} U^{-1}(r_1\4\eta_1,r_2\4 \eta_2))\,\dif\sigma(\eta_1)\4 \dif \sigma(\eta_2),
\end{equation}
where $r_1,r_2 \geq 0$, $\sigma$ denotes the unique normalized Haar measure on the sphere $S^{p-1}$ resp.\ $S^{q-1}$, $(p,q)$ denotes the signature of $Q$ (with $p+ q =d$) and $U$ a rotation in $\R^d$ such that $U Q U^{-1}$ is diagonal matrix whose first $p$ entries are positive and the latter $q$ are negative. Note that in the case of positive definite forms $Q$ (i.e.\ $q=0$), the double integral must be replaced by a single one.

\begin{lemma}
  \label{l3}
  Let $\varphi_{\specialv}$ be defined as in \eqref{eq:spherical_mean_v}. If $Q$ is indefinite, define also
      \begin{equation}
       \label{def:v_Q:1}
       \norm{\specialv}_Q \defi d_Q \sup_{v \in r^{-2} \partial_{w}[a,b]} \left\lvert \int_{0}^{\infty}I( u^2 \ge v)\4 u^{p-1}\varphi_{\specialv}(u,\sqrt{u^2-v}) (u^2-v)^{(q-2)/2} \4 \dif u \right\rvert
  \end{equation}
  and suppose that the latter integral exists. Otherwise, if $Q$ is positive definite, define
  \begin{equation}
    \label{def:v_Q:2}
    \norm{\specialv}_Q \defi d_Q \sup_{v \in r^{-2} \partial_{w}[a,b]} \4 \abs{v^{d-1} \4 \varphi_{\specialv}(v)}
  \end{equation}
  and assume that the latter supremum is bounded. Under these conditions, 
writing $\partial_{w}[a,b] \tdefi [a-2w,a+2w] \cup[b-2w,b+2w]$, we have for $0 < w < (b-a)/4$
  \begin{equation}
    \label{Q-smooth1}
    \int I_{\partial_{w}[a,b]}(Q[x]) \4 \specialv(x/r) \, \dif x \ll_d   w \4 \norm{\specialv}_Q \4 r^{d-2}.
  \end{equation}
  Assuming additionally $\max\{\abs{a},\abs{b}\} \leq c_0 r^2$ with $c_0 = (c_B)^{-1} /5$, the estimates
  \begin{align}
    \label{smoothing}
    R_{\eps,r} & \ \ll_d \  d_Q \4 (b-a) \4 \eps \4 r^{d-2} \\
    \label{volume_bound}
    \volu H_r  & \ \gg_d \ d_Q \4 (\sqrt{c_B})^{-(d-2)} \4 (b-a) \4 r^{d-2}
  \end{align}
  hold for indefinite forms $Q$, provided that $\eps \in (0,\eps_0]$. Moreover, for the special choice $\specialv = \specialv_{\pm\eps}$, as defined in \eqref{omega_eps}, we have
  \begin{equation}
    \label{special:v_eps}
    \norm{\specialv_{\pm\eps}}_Q \ \ll_d \abs{\det{Q}}^{-1/2},
  \end{equation}
  whereby the condition $\max\{\abs{a},\abs{b}\} \leq c_0 r^2$ can be dropped if $Q$ is positive definite.
\end{lemma}

The lower bound \eqref{volume_bound} can be also found in \cite{bentkus-goetze:1999}, see Lemma 8.2. Moreover, Lemma 3.8 in \cite{eskin-margulis-mozes:1998} provides an asymptotic formula for the volume of $H_r$.

\begin{proof}
  For a bounded measurable function $g$ on $\R$ with compact support we introduce
  \begin{equation*}
    R_g \defi \int_{\R^d} g(Q[x]) \4 \specialv(x/r) \, \dif x.
  \end{equation*}
  Let $S_Q=Q\4 Q_{+}^{-1}, L_Q=Q_{+}^{1/2}$ and let $U$ denote the rotation stated in the lemma. In particular, $UQU^{-1}$ and $UL_QU^{-1}$ are diagonal. Changing variables via $x= r L_Q^{-1} U^{-1} \4 y $ in $\R^d$ with $y\in\R^p\times\R^q$, $d=p+q$ and using polar coordinates, $y=\4(r_1\4 \eta_1,r_2\eta_2)$, where $r_1,r_2 >0$ and $\eta_1\in S^{p-1}$, $\eta_2\in S^{q-1}$, that is $\norm{\eta_1}=\norm{\eta_2}=1$, we may write $Q[x]= r^2(r_1^2-r_2^2)$ and obtain by Fubini's theorem
  \begin{equation}
    \label{vol-trf}
    R_g = r^d d_Q \int_0^{\infty}\int_0^{\infty} r_1^{p-1}\4 r_2^{q-1} g(r^2(r_1^2-r_2^2))\4 \varphi_{\specialv}(r_1,r_2)\,\dif r_1\4 \dif r_2,
  \end{equation}
  where $\varphi_{\specialv}(r_1,r_2)$ is defined as in \eqref{eq:spherical_mean_v} for suitable weight functions $\specialv$. (As already noted, in the case of positive definite forms $Q$, the double integral in \eqref{vol-trf} must be replaced by a single one.)  Next, we change variables via $v\tdefi r_1^2-r_2^2$ and $u\tdefi r_1$, so that $r_1^2+r_2^2=2u^2-v$ and $r_2=\sqrt{u^2-v}$. Thus, we get
  \begin{equation}
    \label{V1_bound}
    R_g = r^d \frac{d_Q}{2} \int_{\R} g(r^2\4 v) \int_{0}^{\infty}I( u^2 \ge v)\4 u^{p-1}\varphi_{\specialv}(u,\sqrt{u^2-v}) (u^2-v)^{(q-2)/2} \4 \dif u \, \dif v.
  \end{equation}
  In order to prove \eqref{Q-smooth1}, we choose $g=I_{\partial_{w}[a,b]}$ in \eqref{V1_bound}. Since the length of $r^{-2} \, \mathrm{supp} \, g$ is at most $\ll \abs{w} r^{-2}$, we get $R_g \ll_d \abs{w} r^{d-2} \norm{\specialv}_Q$, where $\norm{\specialv}_Q$ is defined as in \eqref{def:v_Q:1} if $Q$ is indefinite, resp.\ as in \eqref{def:v_Q:2} if $Q$ is positive definite.
  
  Next we prove \eqref{volume_bound}: Taking $g= I_{[a,b]}$, $\specialv(x)= I_{\Omega}(x) = I(M(x) \leq 1)$ and using
  \begin{equation}
    \label{eq:ScalingOfMinkowski}
    \norm{y} d^{-1/2} \leq M(L_Q^{-1} U^{-1}y) \leq \norm{y} (c_B)^{1/2}
  \end{equation}
  gives the lower bound
  \begin{align*}
    \varphi_\specialv(r_1,r_2) & \geq \int_{S^{p-1}\times S^{q-1}} I(\norm{(r_1 \eta_1,r_2 \eta_2)} \leq (c_B)^{-1/2}) \, \dif \sigma(\eta_1) \, \dif 
\sigma(\eta_2) \\
                              & \gg_d I(2u^2+\abs{v} \leq (c_B)^{-1}).
  \end{align*}
  Thus, we find
  \begin{align*}
    \volu H_{r} & \gg_d r^d d_Q \int_{r^{-2}a}^{r^{-2}b} \int_0^\infty I(u^2 \geq v) \4 I(2u^2+\abs{v} \leq (c_B)^{-1}) \4 u^{p-1} (u^2-v)^{(q-2)/2}\4 
\, \dif u \, \dif v        \\
                & \gg_d r^d d_Q \int_{r^{-2}a}^{r^{-2}b} I(\abs{v}\leq c_0) \int_0^\infty \, I( \tfrac{5}{4} c_0 \leq u^2 \leq 2 c_0) u^{p-1} (u^2-v)^{(q-2)/2}  \dif u \, \dif v \\
                & \gg_d r^{d-2} (b-a) d_Q (\sqrt{c_0})^{d-2}.
  \end{align*}
  \textit{Proof of \eqref{smoothing}.} In \eqref{V1_bound} we choose $g=I_{[a,b]}$ and $\specialv=I_{(\partial \Omega)_{2\eps}}$ with $0<\eps \leq \eps_0$. By the properties of the polyhedron $\Omega$, see \eqref{M-bound}, we have $I_{ (\partial \Omega)_{2\eps}}(x) \leq I(M(x)\in J_{1,2\eps})$, where $J_{1,2\eps}\tdefi [1-2\eps, 1+ 2 \eps]$. Let $g_1,\ldots,g_{2d}$ denote the $2d$-tuple of normal vectors defining $\Omega$ and let $f_m = U L_Q^{-1} g_m$, $m=1,\ldots,2d$, be the transformed vectors. Since
  \begin{equation*}
    \textstyle
    I(M(L_Q^{-1}U^{-1}\,y ) \in J_{1,2\eps}) \leq \sum_{m=1}^{2d} I(\langle y ,f_m \rangle \in J_{1,2\eps})
  \end{equation*}
  we may bound $\varphi_{\specialv}(r_1,r_2)$ in \eqref{V1_bound} as follows
  \begin{equation*}
    \textstyle
    \varphi_{\specialv}(r_1,r_2) \leq \sum_{m=1}^{2d} \varphi_{\specialv,m}(r_1,r_2),
  \end{equation*}
  where
  \begin{equation*}
    \varphi_{\specialv,m}(r_1,r_2) \defi \int_{S^{p-1} \times S^{q-1}} I \big[\langle (r_1\eta_1,r_2\eta_2),f_m \rangle \in J_{1,2\eps} \big]\, \dif \eta_1\, \dif \eta_2.
  \end{equation*}
  Recall $\abs{v} \leq c_0$, $v= r_1^2-r_2^2$, $u=r_1$ and $r_2=\sqrt{u^2-v}$. The inequality \eqref{eq:ScalingOfMinkowski} implies
  \begin{equation*}
    (1+2\eps)^2 d \ge r_1^2+r_2^2=2u^2 - v \geq (1-2\eps)^2 (c_B)^{-1}.
  \end{equation*}
  Therefore $\varphi_{\specialv}(u,\sqrt{u^2-v})=0$ if
  \begin{equation*}
    0 \le u < 2^{-\frac{1}{2}} \sqrt{5 c_0 (1-2\eps)^2 -c_0} \quad\text{or}\quad u> C_\Omega \defi  \frac{(1+2\eps)}{\sqrt{2}} \sqrt{d+c_0}.
  \end{equation*}
  Because of
  \begin{equation*}
    2^{-\frac{1}{2}} \sqrt{5 c_0 (1-2\eps)^2-c_0} \geq c_\Omega \defi \frac{\sqrt{310c_0}}{15}
  \end{equation*}
  and $u^2-v \ge 17c_0/45>0$, we get
  \begin{equation}
    \label{bV1_bound}
    \begin{aligned}
      R_g & \ll r^d d_Q \int_{r^{-2}a}^{r^{-2} b} \Bigg( \int_{c_\Omega}^{C_\Omega} u^{p-1} (u^2-v)^{\frac{q-2}{2}} \varphi_{\specialv} (u,\sqrt{u^2-v})  \, \dif u \Bigg) \dif v \\
      & \le r^d d_Q \sum_{m=1}^{2d} \int_{r^{-2}a}^{r^{-2}b} \Bigg( \int_{c_\Omega}^{C_\Omega}  u^{p-1} (u^2-v)^{\frac{q-2}{2}} \varphi_{\specialv,m}(u,\sqrt{u^2-v}) \, \dif u \Bigg) \dif v.
    \end{aligned}
  \end{equation}
  By interchanging the variables $r_1$ and $r_2$ we can suppose that $q \ge 2$. Thus, since $u \ll_d 1$ and $\sqrt{u^2-v} \ll_d 1$, we see that
  \begin{equation}
    \label{bV1_bound:2}
    \int_{c_\Omega}^{C_\Omega}  u^{p-1} (u^2-v)^{\frac{q-2}{2}} \varphi_{\specialv,m}(u,\sqrt{u^2-v}) \, \dif u  \ll_d \int_{c_\Omega}^{C_\Omega} \varphi_{\specialv,m}(u,\sqrt{u^2-v}) \, \dif u.
  \end{equation}
  We claim that
  \begin{equation}
    \label{l3.1}
    R_g \ll_d  d_Q  \4 \eps \4 (b-a)r^{d-2}
  \end{equation}
  holds. In view of \eqref{bV1_bound} and \eqref{bV1_bound:2}, the estimates
  \begin{equation*}
    R_m \defi \int_{c_\Omega}^{C_\Omega}  \varphi_{\specialv,m}(u,\sqrt{u^2-v})\, \dif u \ll_d \eps \4 c_\Omega
  \end{equation*}
  for all $m=1,\dots,2d$ will prove the bound \eqref{l3.1}.\par Thus let $F_m(u)\tdefi \langle (u\4 \eta_1,(u^2-v)^{1/2} \4 \eta_2),f_m \rangle$ 
for fixed $\abs{v} \le c_0$ and $(\eta_1,\eta_2)$. If
  \begin{equation}
    \label{l3.2}
    \Big\lvert\frac{\partial}{\partial u} F_m(u)\Big\rvert \ge c_1 >0
  \end{equation}
  for all $c_\Omega \leq u\le C_\Omega$ with $F_m(u)\in [1-2\eps,1+2\eps]$ uniformly in $(\eta_1,\eta_2)$ and $v$, then
  \begin{equation*}
    \int_{c_\Omega}^{C_\Omega} I(F_m(u)\in [1-2\eps,1+2\eps])\, \dif u\ll\frac{\eps}{c_1}
  \end{equation*}
  and hence $R_m \ll_d c_1^{-1} \eps$ for all $m=1,\dots,2d$. Note that
  \begin{equation*}
    \frac{\partial}{\partial u} F_m(u) = \frac{1}{u} \Bigg(F_m(u) + \frac{v}{\sqrt{u^2-v}} \langle (0,\eta_2), f_m \rangle \Bigg)
  \end{equation*}
  and because of $\norm{L_Q^{-1} B^T} = \norm{B \4 L_Q^{-1}} \le \sqrt{c_B}$ we see that
  \begin{equation*}
    \Big\lvert\frac{\partial}{\partial u} F_m(u)\Big\rvert \ge \frac{1}{u}\Big( \abs{F_m(u)} - \frac{c_0}{ \sqrt{17 c_0/45} } \norm{f_m} \Big) \ge \frac{1}{u} \Big( 1-2 \eps - \frac{4}{5} \Big) \gg c_\Omega^{-1}.
  \end{equation*}
  Note, that here it is important that $\eps >0$ is not too large, i.e.\ $\eps \in (0,\eps_0]$. Thus, \eqref{l3.2} holds and the assertion \eqref{l3.1} is proved. This yields the claimed bound for $R_{\eps,r}$, compare \eqref{smooth-error}. 
  
  Finally, we prove \eqref{special:v_eps}. Here we have $v= v_{\pm\eps}$ and $v_{\pm\eps}(x) \leq I(M(x) \leq 1 + 2 \eps)$. In view of \eqref{eq:ScalingOfMinkowski}, we find that the $u$-integral in \eqref{def:v_Q:1} can be restricted to $2u^2 \leq 2d +v$. Hence
  \begin{equation*}
    \norm{\specialv_{\pm\eps}}_Q \ll_d d_Q \sup_{v \in r^{-2} \partial_{w}[a,b]}   (1+\abs{v})^{(d-3)/2} \int_0^\infty I(v \leq u^2 \leq d+v/2) \, 
\dif u \ll d_Q,
  \end{equation*}
  because $\abs{v} \le r^{-2}(\abs{a}+\abs{b}) \le 2 c_0 \le 1$. Since $\varphi_\specialv$ is supported in $\norm{\cdot}$-ball of radius $2d^{1/2}$, we get also in the case of positive definite forms that \eqref{def:v_Q:2} is bounded by $\ll_d d_Q$.
\end{proof}



\subsection{Fourier Transform of Weights for Polyhedra}
\label{ftpoly}
Here we continue to estimate the remainder terms in \eqref{omega-smooth}. 
Since the bounds for $R(g^Q_w \4 \specialv_{-\eps,r})$ are exactly the same as for $R(g^Q_w \4 \specialv_{+\eps,r})$ we shall consider the latter only. We 
shall now modify the weight $\specialv_{\eps}$, defined in \eqref{omega_eps}, as follows. Define $\varphi=I_{[-2,2]} \ast k$, where $k$ is again 
the probability measure from Subsection \ref{subsection:3.1}. Of course, $\varphi$ is smooth and $\varphi(u)=1$ if $\abs{u} \leq 1$ and $\varphi(u)=0$ if $\abs{u} \geq 3$. Let $s_d \tdefi d(1+2\eps_0)^2$. Now, by construction $\varphi(Q_{+}[x] s_d^{-1})$ is identical to $1$ on the support of the $\eps$-smoothed indicator of $\Omega_{\eps} = B^{-1}[-(1+\eps),(1+\eps)]^d$, that is $\specialv_{\eps}(x)$. Hence we may rewrite the weights $\zeta$ of \eqref{zeta-def} via
\begin{equation}
  \label{sec7:CutOff:SmoothingRegion}
  \zeta_{\eps}(x)= \specialv_{\eps}(x) \exp\{Q_{+}[x]\} =\specialv_{\eps}(x) \psi(x)
\end{equation}
using the $C^{\infty}$ function $\psi(x)\tdefi \exp\{ Q_{+}[x] \} \varphi 
(Q_{+}[x]  s_d^{-1})$ of bounded support, whose Fourier transform can easily be estimated, see \eqref{l2.4}. In particular, the weights $\zeta_{\eps}$ satisfy the integrability condition \eqref{zeta-decay}, i.e.\ $\sup_{x\in \R^d}\big(\abs{\zeta_{\eps}(x)} + \abs{\wh{\zeta}_{\eps}(x)}\big) (1+\norm{x})^{d+1} < \infty$.

\begin{lemma}
  \label{l2}
  The following estimate holds
  \begin{equation}
    \label{l2.2}
    \int_{\R^d} \, \abs{ \wh{\zeta_{\eps} }(v)} \, \dif v \ll_d  \int \abs{\wh{I}_{[-1,1]^d}}(v)  \4 {\textstyle \prod_{j=1}^d } \exp\{-  \abs{\eps v_j}^{1/2}\} \, \dif v \ll_d  (\log \eps^{-1})^d.
  \end{equation}
\end{lemma}

\begin{remark}
  In the general case, when $\Omega$ has finite Minkowski surface measure 
$c_{\Omega}$ only, defined via $\mathrm{meas}(\partial_{\eps} \Omega) \le 
 c_{\Omega} \eps$, we have\index{I@ $\norm{\wh{I_{\Omega}}}_{1,\eps} 
:= \int_{\R^d} \abs{\wh{I}_{\Omega}(v)} \exp \{-  \norm{\eps v}^{1/2} \} \, \dif v \ll_{\Omega} \eps^{-(d+1)/2}$}
  \begin{equation*}
    \norm{\wh{I_{\Omega}}}_{1,\eps} \defi \int_{\R^d} \abs{\wh{I}_{\Omega}(v)} \exp \{-  \norm{\eps v}^{1/2} \} \, \dif v \ll_d c_\Omega \eps^{-d}
  \end{equation*}
  as can be deduced from the bound in Theorem 2.9 of \cite{brandolini-et.al:1997}, that is
  \begin{equation*}
   \frac{1}{\v(u \leq \norm{v} \leq 2 u)} \int_{\{u \leq \norm{v} \leq 2 u\}} \abs{\wh{I}_{\Omega}(v)} \, \dif v \leq c_\Omega (2+u)^{-(d+1)/2}.
  \end{equation*}
  This estimate is sharp as shown by the explicit example of an unit ball, see \cite{brandolini-et.al:1997} for more details. That paper contains also bounds on the average $\eta \mapsto \abs{\wh{I}_{\Omega}(s \eta)}$ over the unit sphere $S^{d-1}$ for polyhedra, which are usually of smaller order than pointwise bounds. In fact, the pointwise decay of $\wh{I}_\Omega(v)$ may depend crucially on the direction of $v$. In our setting (finding $L^1$-estimates for specially oriented parallelepipeds $\Omega$) more elementary arguments can be used.
\end{remark}

\begin{proof}
  Note that by definition
  \begin{equation}
    \label{l2.0}
    \int_{\R^d} \abs{\wh{\zeta_{\eps}}(v)} \, \dif v= \int_{\R^d} \abs{\wh{\specialv_{\eps} \psi}(v)}\, \dif v = \int_{\R^d} \Big\lvert\int_{\R^d} \wh{\specialv_{\eps}} (v-x)\wh{\psi}(x)\, \dif x \Big\rvert\4 \dif 
v \leq \norm{\wh{\specialv_{\eps}}}_1 \4 \norm{\wh{\psi}}_1.
  \end{equation}
  Since
  \begin{equation*}
    \wh{\psi}(x) = \abs{\det{Q}}^{-1/2} \int_{\R^d} \exp[ v^2] \4 \varphi( v^2 s_d^{-1}) \4 \e^{- 2 \pi \iu \langle v, Q_{+}^{-1/2} x \rangle} \, \dif v
  \end{equation*}
  we easily conclude that
  \begin{equation}
    \label{l2.4}
    \abs{\wh{\psi}(x)} \le \abs{\det{Q}}^{-1/2} c(d,k)(1+Q_{+}^{-1}[x])^{-k}, \, x\in\R^d, \q \text{and thus} \q \norm{\wh{\psi}}_1 \le c(d).
  \end{equation}
  Defining $Z \tdefi (B^{-1})^T$ and changing variables shows also that
  \begin{equation}
    \label{Appendix:Fourier:Region}
    \wh{I}_{\Omega_{\eps}}(v) = (1+\eps)^d \wh{I}_{\Omega}((1+\eps)v) = (1+\eps)^d \abs{\det{B}}^{-1} \4 \wh{I}_{[-1,1]^d}((1+\eps) Z v)
  \end{equation}
  and
  \begin{equation}
    \label{Appendix:Fourier:Smooth}
    \textstyle
    \abs{\wh{k}_{B,\eps}(v)} \leq \exp\{-   \eps^{1/2} \sum_{j=1}^d \abs{(Z v)_j}^{1/2} \}.
  \end{equation}
  Thus we get for $\specialv_{\eps}=I_{\Omega_{\eps}}*k_{B,\eps}$
  \begin{equation}
    \label{l2.5}
    \norm{\wh{\specialv}_{\eps}}_1 = \norm{\wh{I}_{\Omega_{\eps}} \wh{k}_{B,\eps}}_1 \ll_d \int_{\R^d} \4 \abs{\wh{I}_{[-1,1]^d}((1+\eps)v)} \4 {\textstyle \prod_{j=1}^d } \exp\{-  \abs{\eps v_j}^{1/2} \} \, \dif v.
  \end{equation}
  Finally, using $\wh{I}_{[-1,1]^d}(v) =  \prod_{j=1}^d \sin( 2 \pi v_j)/( \pi v_j)$ together with \eqref{l2.5} gives the estimate
  \begin{equation}
    \label{final-omega}
    \norm{\wh{\specialv}_{\eps}}_1 \ll_d \Big( \int_0^\infty \frac{1}{u+\eps} \4 \e^{- \sqrt{u}} \, \dif u \Big)^d \ll_d \Big(1 + \int_0^1 \frac{1}{u+\eps} \, \dif u \Big)^d  \ll_d \log(\eps^{-1})^d.
  \end{equation}
  We now obtain the estimate \eqref{l2.2} from \eqref{l2.0} combined with 
\eqref{l2.4} and \eqref{final-omega}.
\end{proof}



\subsection{Lattice Point Remainders for Admissible Parallelepipeds}
\label{specialp}
Now we restrict the parallelepiped $\Omega = B^{-1} [-1,1]^d$, as defined in \eqref{eq:sec7:Def:Omega}, such that its faces are in a general position relative to the standard lattice $\Z^d$. This ensures that the lattice point remainder for $r \Omega$  is  of `abnormally' small error \textit{uniformly} in $r$. To construct it, we may alternatively construct lattices $B \4\Z^d$ such that the faces of $[-1, 1]^d$ have this property. Following Skriganov \cite{skriganov:1994}, we call a lattice  $\Gamma \subset \R^d$ of full rank, and likewise $\Omega$, `admissible' if
\begin{equation}
  \label{Def:Admissible}
  \textstyle
  \Nm \Gamma \defi \inf_{\gamma \in \Gamma \setminus \{ 0 \}} \4 \abs{ \Nm \gamma } >0,
\end{equation}
where $\Nm \gamma = \abs{\gamma_1 \cdots \gamma_d}$ in standard coordinates $\gamma=(\gamma_1, \ldots, \gamma_d)$. 

\begin{remark}
 \label{remark:app_a:set_of_admis}
 The set of all admissible lattices is dense in the space of lattices (see \cite{skriganov:1998}). Hence, for any $\eta >0$, if $D_\eta$ denotes the set of diagonal matrices with entries in $[1,1+\eta)$, then $\text{O}(d)D_\eta \text{O}(d) \Gamma$ contains an admissible lattice. In particular, if $\Gamma = Q_+^{1/2} \Z^d$, then there exist orthogonal matrices $k,l \in \text{O}(d)$ and a diagonal matrix $d \in D_\eta$ such that  $B \Z^{d}$ is admissible, where $B = k d  l \, Q_+^{1/2}$ satisfies property \eqref{eq:sec7:Def:Omega} with a constant $c_B$ depending only on $\eta$.
\end{remark}

\begin{remark}
  This definition is a special case of `admissible lattices' for \textit{star-bodies}, see Chapter IV.4 in \cite{Cassels:1959}. Here, the star-body is given by $\{F <1\}$ with the \textit{distance function} $F(x) = \abs{x_1 \cdots x_d}^{1/d}$.
\end{remark}

As shown in Lemma 3.1 of \cite{skriganov:1994}, the dual lattice $\Gamma^*=Z\Z^d$ of $\Gamma$, where $Z^TB=\mathrm{Id}$, is admissible as well. Another property of admissible lattices is that there exists a cube $[-r_0,r_0]^d$ containing a fundamental domain $F$ of $\Gamma$ such that $r_0>0$ depends only by means of the invariants $\det{\Gamma}$ and $\Nm \Gamma$.

\begin{example}
  \label{example:app_a:admis_latt}
  Well known examples are provided by the \textit{Minkowski embedding} of 
a \textit{totally real} algebraic number field $\mathbb{F}$ of degree $d$ 
into $\R^d$. Given all embeddings $\sigma_1,\ldots,\sigma_d$ of $\mathbb{F}$, the Minkowski embedding $\sigma \colon \mathbb{F} \rightarrow \R^d$ is defined by $\sigma = (\sigma_1,\ldots,\sigma_d)$. In this case $\Nm \sigma(\alpha) = \abs{N_{\mathbb{F}/\mathbb{Q}}(\alpha)}$ is the field norm of any $\alpha \in \mathbb{F}$, where we interpret multiplication by 
$\alpha$ as a $\mathbb{Q}$-linear map. Thus, the image of the ring of integers $\mathcal{O}_\mathbb{F}$ is an admissible lattice $\Gamma$ with $\Nm \Gamma \geq 1$. For more information, see Chapter 2.3 in \cite{borevich-schafarevich:1966}.
\end{example}

\begin{remark}
  We also note that for any natural number $n \in \mathbb{N}$ we may choose a real number field of degree $n$ which is normal over the rational numbers. In fact, let $m \in \mathbb{N}$ be chosen such that $2n \mid \varphi(m)$ and let $\xi_m$ be a primitive $m$-th root of unity. Then $\mathbb{Q}(\xi_m + \xi_m^{-1})$ is a real number field of degree $\varphi(m)/2$, 
which is also normal and its Galois group $G$ is abelian. Since $G$ contains a subgroup $H$ of order $\varphi(m)/(2n)$, the fixed field of $H$ is real, normal and of degree $n$. Thus, there exists an admissible region $\Omega$ satisfying \eqref{eq:sec7:Def:Omega} with $c_B \asymp_d q/q_0$ and $\Nm(B) \asymp_d q^{d/2}$.
\end{remark}

\begin{lemma}
  \label{omega-remainder}
  Assume that the lattice $\Gamma=B\Z^d$ is admissible and $B$ satisfies \eqref{eq:sec7:Def:Omega}. For $0< \eps \leq \eps_0$ and $r \geq 1$ we get for the parallelepiped $\Omega=B^{-1} [-1,1]^d$ and the corresponding weights $\zeta_{\eps}(x)= \specialv_{\eps}(x) \psi(x)$ introduced in 
Subsection \ref{ftpoly}
  \begin{equation}
    \label{zetaperiod}
    I_{\zeta}\defi \int_{\norm{v}_{\infty} > r/2} \frac{\abs{\wh \zeta_{\eps}(v)}} {( q^{1/2} r^{-1}+\norm{ r^{-1}\4 v }_{\Z^d})^{d/2}}\4 \dif v \ll_d  q_0^{-d/4} \4 d_Q \4 \abs{\det{B}} \lambda_{r,\eps}^{d-1} \4 \frac{\bar{\lambda}_{r,\eps,\Gamma}}{\Nm(\Gamma)},
  \end{equation}
  where $\lambda_{r,\eps} \tdefi \min \{ \log(r+1), \log (\eps^{-1}) \}$ and $\bar{\lambda}_{r,\eps,\Gamma} \tdefi \max \{ \lambda_{r,\eps},\log(2+\tfrac{1}{\Nm(\Gamma) \4 r\eps})\}$. For any inadmissible parallelepiped 
$\Omega$ only the estimate
  \begin{equation}
    \label{zetaperiod-notadmiss}
    I_{\zeta} \ll_d d_Q \4 q^{d/2} \4 c_B^{(d+1)/2} \eps^{-d}
  \end{equation}
  holds. Additionally, we also have $d_Q \abs{\det{B}} \leq (c_B)^{d/2}$.
\end{lemma}

\begin{proof}
  We start by making the change of variables $w=r^{-1} Z\4v$ in \eqref{zetaperiod} and then splitting $I_\zeta$ into integrals over cells $C^*:= 
Z [-\frac{1}{2},\frac{1}{2})^d$, where $\Gamma^* \tdefi Z\Z^d$ denotes the dual lattice to $\Gamma$, that is $Z= (B^T)^{-1}$, in order to get
  \begin{equation}
    \label{1dint}
    I_{\zeta} = \hspace*{-1mm} \sum_{ \gamma^* \in \Gamma^* \setminus \{0\} } \hspace*{-1mm}  I_{\zeta}(\gamma^*), \ \  \text{where} \ \  I_{\zeta}(m) \defi  r^d \abs{\det{B}} \int_{C^*} \frac{\abs{\wh \zeta_{\eps}(Z^{-1}r(\gamma^* +v))}}{( q^{\frac{1}{2}} r^{-1} + \norm{Z^{-1}v}_\infty)^{\frac{d}{2}}} \, \dif v.
  \end{equation}
  Note that $\Gamma^*$ satisfies $\norm{Z} \leq \norm{Q_{+}^{-1/2}} \leq q_0^{-1/2}$, since the first inequality in \eqref{eq:sec7:Def:Omega} implies
  \begin{equation}
    \label{OpNormBoundForDualLattice}
    1 \geq \norm{Q_{+}^{1/2} B^{-1}} = \norm{((B^T)^{-1} Q_{+}^{1/2})^T} = \norm{(B^T)^{-1} Q_{+}^{1/2}} = \norm{Z Q_{+}^{1/2}}.
  \end{equation}
  In particular, the fundamental domain $C^*$ is contained in $q_0^{-1/2} 
\sqrt{d} [-\frac{1}{2},\frac{1}{2}]^d$. Next, we shall bound the Fourier transform of $\zeta_\eps$. Recall that by definition
  \begin{equation}
    \label{app:eq:l7.6-0}
    \wh{\zeta}_{\eps}(u) = ((\wh{I}_{\Omega_{\eps}} \cdot \wh{k}_{B,\eps}) * \wh\psi) (u).
  \end{equation}
  As verified in \eqref{Appendix:Fourier:Region}, we have in coordinates $u=(u_1, \ldots, u_d)$
  \begin{equation}
    \label{app:eq:l7.6-1}
    \abs{\wh{I}_{\Omega_\eps}(Z^{-1}u)} \ll_d \abs{\det{B}}^{-1} \prod_{j=1}^d \Big\lvert \frac{\sin[2 \pi(1+\eps)u_j]}{(1+\eps) u_j}\Big\rvert \ll_d \abs{\det{B}}^{-1} \prod_{j=1}^d (1+\abs{u_j})^{-1}.
  \end{equation}
  Since \eqref{OpNormBoundForDualLattice} also implies $\norm{Q_{+}^{-1/2}(Z^{-1}u)} \geq \norm{u}$, we can rewrite \eqref{l2.4} by
  \begin{equation}
    \textstyle
    \label{app:eq:l7.6-2}
    \abs{\wh \psi (Z^{-1}u)} \ll_{d,k} \abs{\det{Q}}^{-1/2} (1+\norm{u}^2)^{-k} \ll_{d,k} \abs{\det{Q}}^{-1/2} \prod_{j=1}^d (1+u_j^2)^{-k/d},
  \end{equation}
  where we applied the AM-GM inequality. In view of \eqref{Appendix:Fourier:Smooth} we have the bound
  \begin{equation}
    \textstyle
    \label{app:eq:l7.6-3}
    \abs{ \wh{k}_{B,\eps}(Z^{-1}u)} \leq \exp\{-  \sum_{j=1}^d \abs{\eps \4 
u_j}^{1/2} \}
  \end{equation}
  as well. Combining these estimates yields
  \begin{equation*}
    \abs{\wh{\zeta}_{\eps}(Z^{-1} r w)} \ll_{d,k} d_Q \int_{\R^d}  \prod_{j=1}^d \frac{1}{(1+u_j^2)^{k/d}} \frac{ \exp\{-  \eps^{1/2} \abs{rw_j-u_j}^{1/2}\}}{1+\abs{rw_j-u_j}} \, \dif u.
  \end{equation*}
  Thus, we get for a fixed lattice point $\gamma^* = (\gamma^*_1,\ldots,\gamma^*_d) \in \Gamma^*$
  \begin{equation*}
    I_{\zeta}(\gamma^*) \ll_{d,k} \int_{C^*} \frac{\abs{\det{Q}}^{-1/2} \4 \abs{\det{B}}}{(q r^{-1} + \norm{Z^{-1} v}_\infty)^{d/2}} \int_{\R^d} \prod_{j=1}^d \bar{\omega}(u_j) \frac{ \omega( \eps r (\gamma_j^* + v_j-\frac{u_j}{r})) }{r^{-1}+\abs{\gamma_j^* + v_j -\frac{u_j}{r}}} \, \dif u 
\, \dif v,
  \end{equation*}
  where $\bar{\omega}(x) \tdefi (1 + x^2)^{-k/d}$ and $\omega(x) \tdefi \exp\{- \abs{x}^{1/2} \}$. We now estimate the last double integral coordinatewise: Note that we have $\abs{v_i} \leq \bar{v} \tdefi \sqrt{d}/2$ and
  \begin{equation*}
    \textstyle
    (q^{1/2} r^{-1}+\norm{Z^{-1}v}_\infty)^{d/2} \gg_d q_0^{d/4} (r^{-1} + \norm{v}_\infty)^{d/2} \geq q_0^{d/4} \prod_{j=1}^d (r^{-1} + \abs{v_i})^{1/2},
  \end{equation*}
  since $\norm{Z^{-1}v}_\infty \gg_d \norm{Z}^{-1} \norm{v}_\infty \geq q_0^{1/2} \norm{v}_\infty$. Hence, we find
  \begin{equation*}
    \textstyle
    I_{\zeta}(\gamma^*) \ll_{d,k} q_0^{-d/4} \4 d_Q \abs{\det{B}} \prod_{j=1}^d  J_{\zeta}(\gamma^*_j;\R),
  \end{equation*}
  where
  \begin{equation*}
    J_{\zeta}(\gamma^*_j;D) \defi \int_{-\bar{v}}^{\bar{v}} \frac{1}{(r^{-1} +\abs{v})^{1/2}} \int_{D} \bar{\omega}(u) \frac{ \omega( \eps r (\gamma_j^* + v-\frac{u}{r})) }{r^{-1}+\abs{\gamma_j^* + v -\frac{u}{r}}} \, \dif u \, \dif v.
  \end{equation*}
  In order to estimate $J_{\zeta}(\gamma^*_j;\R)$, we decompose the integral into parts corresponding to the extremal points of the integrands. Defining $D_j \tdefi \{\abs{u} \geq r\abs{\gamma^*_j +v}/2\}$, we get
  \begin{equation*}
    J_{\zeta}(\gamma^*_j;D_j) \le \int_{-\bar{v}}^{\bar{v}} \frac{r}{\abs{v}^{1/2}} \int_{D_j} \bar{\omega}(u) \, \dif u \, \dif v \ll_{k,d} \int_{-\bar{v}}^{\bar{v}} \frac{1}{\abs{v}^{1/2}} \frac{r}{(1+r \abs{\gamma^*_j+v})^{\frac{k}{d}-1}} \, \dif v.
  \end{equation*}
  In the case $\abs{\gamma_j^*} \geq \sqrt{d}$, we have $\abs{\gamma^*_j+v} \geq \abs{\gamma_j^*}/2$ and hence
  \begin{equation*}
    J_{\zeta}(\gamma^*_j;D_j) \ll_d \frac{r}{(1+ \abs{r \gamma_j^*})^{d+2}} \int_{-\bar{v}}^{\bar{v}} \frac{1}{\abs{v}^{1/2}} \, \dif v \ll_d \frac{1}{(1+ \abs{r \gamma_j^*})^{d+1}}
  \end{equation*}
  if we take $k = d (d+3)$. In the other case $\abs{\gamma_j^*} < \sqrt{d}/2$, we split the $v$-integral into two parts as follows in order to find the estimate
  \begin{align*}
    J_{\zeta}(\gamma^*_j;D_j) & \ll_d \int_{-\bar{v}}^{\bar{v}} \frac{\abs{\gamma_j^*}^{-\frac{1}{2}} \4 r \4 I(\abs{v} \geq \abs{ \gamma_j^*}/2)}{(1+ r\abs{\gamma_j^*+v})^{d+2}} \, \dif v + \int_{0}^{\abs{\gamma_j^*}/2} \frac{r}{v^{\frac{1}{2}} (1+ r ( \abs{\gamma^*_j}-v))^{d+2}} \, \dif v \\
                              & \ll_d \abs{\gamma_j^*}^{-\frac{1}{2}} + \frac{\abs{\gamma_j^*}^{\frac{1}{2}} r}{(r\abs{\gamma^*_j}+1)^{d+2}} \int_0^{1/2} \frac{1}{v^\frac{1}{2}(1-v)^{d+2}} \, \dif v \ll_d  \abs{\gamma_j^*}^{-\frac{1}{2}}.
  \end{align*}
  In the complement $u \in D_j^c$ we have $\abs{\gamma_j^* + v-\frac{u}{r}} \geq \abs{\gamma_j^* + v}/2$ and thus
  \begin{equation*}
    J_{\zeta}(\gamma^*_j;D_j^c) \ll_d  \int_{-\bar{v}}^{\bar{v}}\abs{v}^{-\frac{1}{2}} \frac{ \omega( \eps r (\gamma_j^* + v)/2) }{r^{-1}+\abs{ \gamma_j^* + v}} \, \dif v.
  \end{equation*}
  If $\abs{\gamma_j^*} \geq \sqrt{d}$, then we easily conclude that $J_{\zeta}(\gamma^*_j;D_j^c) \ll_d \omega( \eps r \gamma_j^* /4 ) \abs{\gamma_j^*}^{-1}$. At last, we consider the case $\abs{\gamma_j^*} < \sqrt{d}$. The $v$-integral over the region $\{\bar{v} \geq \abs{v} \geq \abs{\gamma^*_j}/2\}$ can be bounded by
  \begin{align*}
     & \ll_d \abs{\gamma_j^*}^{-1/2}\int_{-\bar{v}}^{\bar{v}} \frac{I( \abs{v} \geq \abs{\gamma^*_j}/2)}{(r^{-1}+\abs{\gamma_j^* + v})(1+ \eps r \abs{ \gamma_j^* + v})} \, \dif v        \\
     & \ll_d \abs{\gamma_j^*}^{-1/2} \int_0^{3 \sqrt{d}/2}\frac{1}{r^{-1}+v}\frac{1}{1+\eps r v} \, \dif v \ll_d \abs{\gamma_j^*}^{-1/2} \min \{ \log(\eps^{-1}), \log(r \! + \! 1)\}
  \end{align*}
  and similar over the complement by
  \begin{equation*}
    \ll_d \int_0^{\abs{\gamma^*_j}/2} \frac{v^{-1/2}}{r^{-1}+\abs{ \gamma_j^*} -v} \, \dif v \ll_d  \abs{\gamma^*_j}^{-1/2}.
  \end{equation*}
  Hence we conclude that
  \begin{equation}
    \label{Idef}
    I_{\zeta} \ll_d q_0^{-d/4} \4 d_Q \4 \abs{\det{B}} \sum_{(\gamma^*_1, 
\ldots, \gamma^*_d) \in \Gamma^*\setminus \{0\}} \prod_{j=1}^d \frac{H_{r,\eps}(\gamma_j^*)}{\abs{\gamma_j^*}},
  \end{equation}
  where
  \begin{equation}
    \label{DefOfH:Lemma7.4}
    H_{r,\eps}(x) \tdefi \lambda_{r,\eps} \abs{x}^{1/2} I(\abs{x} < \sqrt{d}) + (1+ \eps r \abs{x})^{-d} I(\abs{x} \ge \sqrt{d}).
  \end{equation}
  In view of the following Lemma \ref{Fourier-sum} this concludes the proof of the bound \eqref{zetaperiod}.\par
  If the region $\Omega$ is not admissible, then we  change variables to $w = r^{-1}v$ split the left-hand side of \eqref{zetaperiod} into integrals over unit cells $E:= [-\frac{1}{2},\frac{1}{2})^d$ in order to find
  \begin{equation*}
    I_{\zeta} = \sum_{ m \in \Z^d \setminus \{0\} } I_{\zeta}(m), \ \  \text{where} \ \  I_{\zeta}(m) \defi  r^d \int_E \frac{\abs{\wh \zeta_{\eps}(r (m +w)}}{( q^{1/2} r^{-1} + \norm{w}_\infty)^{d/2}} \, \dif w.
  \end{equation*}
  Because of $\sum_{j=1}^d \abs{u_j}^{1/2} \geq \norm{u}^{1/2}$ we can further estimate \eqref{app:eq:l7.6-3} by
  \begin{equation*}
    \abs{ \wh{k}_{B,\eps}(Z^{-1}u)} \leq \exp\{-  \norm{\eps \4 u}^{1/2} \}.
  \end{equation*}
  Recalling the definition \eqref{app:eq:l7.6-0} and the estimates \eqref{app:eq:l7.6-1}--\eqref{app:eq:l7.6-2} for $u =Z w$ shows that
  \begin{equation*}
    \abs{\wh{\zeta}_{\eps}(r w)} \ll_k  d_Q \4 \eps^{-k+1} (r \norm{Zw}+1)^{-k} \ll  d_Q \4 \eps^{-k+1} (q \4 c_B)^{k/2} (r \norm{w}+1)^{-k}.
  \end{equation*}
  Thus, taking $k=d+1$ we find
  \begin{equation*}
    I_{\zeta} \ll_d d_Q \4 q^{d/2} \4 c_B^{(d+1)/2} \eps^{-d}.
  \end{equation*}
  The last remark easily follows by comparing the volume of the bodies $\{\norm{Bx} \leq 1\}$ and $\{ \norm{Q_{+}^{1/2}x} \leq 1\}$: Using \eqref{eq:sec7:Def:Omega} leads to $\abs{\det{Q}}^{1/2} \leq \abs{\det{B}} \leq (c_B)^{d/2} \abs{\det{Q}}^{1/2}$.
\end{proof}

\begin{lemma}
  \label{Fourier-sum}
 For an admissible lattice $\Gamma$ we have for any weight function $\omega(x)>0$ on $\R$, such that $\omega_\infty \tdefi 1+ \max_x\omega(x) (1+ \abs{x})^p < \infty$, where $p \in \N$ and $\eps>0$, the bound
  \begin{equation}
    \label{Fourier-sum:eq}
    S_{\Gamma, \eps}\defi \hspace{-1mm} \sum_{(\gamma_1, \ldots, \gamma_d) \in \Gamma \setminus \{0\}} \Big\lvert\frac{\omega_{r,\eps}(\gamma_1) \ldots \omega_{r,\eps} (\gamma_d) } {\gamma_1 \ldots \gamma_d}\Big\rvert \ll_d \omega_\infty \4 \lambda_{r,\eps}^{d-1} \4 \frac{\bar{\lambda}_{r,\eps,\Gamma}}{\Nm(\Gamma)},
  \end{equation}
  where $\omega_{r,\eps}(x) \tdefi \lambda_{r,\eps} \abs{x}^{\frac{1}{2}} 
I(\abs{x} <  \sqrt{d})+ \omega( \eps r x) I(\abs{x} \ge \sqrt{d})$ and $ \lambda_{r,\eps}$, $\bar{\lambda}_{r,\eps,\Gamma}$ are as introduced in Lemma \ref{omega-remainder}.
\end{lemma}

\begin{proof}
  First, we make a decomposition of $\Gamma$ as follows. For any $(x_1,\ldots,x_d) \in \R^d$ with $\abs{x_1 \cdots x_d}\ge \Nm(\Gamma)$ let $m_j \in \Z$ be the unique integers satisfying $2> \abs{ 2^{m_j} x_j} d^{-1/2} \ge 1$ for $j=2, \ldots,d$. We have $\abs{x_1} \ge \Nm(x) \abs{x_2\ldots x_d}^{-1} \ge \Nm(\Gamma) d^{(1-d)/2} \prod_{j=2}^d 2^{m_j -1}$ and this implies that $\abs{2^{m_1}x_1} \in [k\4c_{\Gamma},(k+1)c_{\Gamma})$ for a unique integer $k \geq 1$, where $m_1 \in \Z$ is determined by $m_1 + m_2 +\ldots+ m_d =0$ and $c_{\Gamma}= d^{(1-d)/2} 2^{-d+1}\Nm(\Gamma)$. Introducing the lattice
  \begin{equation*}
    E_d \tdefi \{m=(m_1, \ldots, m_d)\in \Z^d \,: \, m_1 + \ldots +m_d=0 \} \subset \Z^d
  \end{equation*}
  and the interval $B_k\tdefi[k\4c_{\Gamma},(k+1)c_{\Gamma})$, we can write
  \begin{equation*}
    I( \abs{x_1\ldots x_d} \ge \Nm(\Gamma) \4 ) = \sum_{m\in E_d} \sum_{k\in \N} I_{B_k}(\abs{2^{m_1 }\4 x_1}) \prod_{j=2}^{d}I_{[\sqrt{d},2\sqrt{d})}(\abs{2^{m_j} \4 x_j}),
  \end{equation*}
  and hence
  \begin{equation}
    \label{Gamma-sum1}
    S_{\Gamma, \eps} = \sum_{m\in E_d} \sum_{k\in \N} \sum_{\gamma\in \Gamma} I_{B_k}(\abs{2^{m_1}\4 \gamma_1})\prod_{j=2}^{d} I_{[\sqrt{d},2\sqrt{d})}(\abs{2^{m_j} \4 \gamma_j})  \Big\lvert \frac{\omega_{r,\eps}( \gamma_1) \ldots \omega_{r,\eps}( \gamma_d) }{\gamma_1 \ldots \gamma_d} \Big\rvert.
  \end{equation}
  We also introduce the obvious notations $\Nm(x) \tdefi \abs{x_1 \cdots x_d }$, $2^m x= ( 2^{m_1}x_1, \ldots 2^{m_d} x_d ),\\ m \in E_d$ and $2^m \Gamma$ for the rescaled lattice $\{ 2^m \gamma \4 : \4 \gamma \in \Gamma \}$. Note that $\Nm(2^m\gamma)=\Nm(\gamma)$ and hence $\Nm(\Gamma)=\Nm(2^m\Gamma)$. Defining $C_k\tdefi B_k\times [\sqrt{d},2\sqrt{d})^{d-1}$ and $h(x) \tdefi ( 1+ \abs{x})^{-p}$ (where $p \in \mathbb{N}$ is the same as in the assumptions of the lemma), we may rewrite and bound \eqref{Gamma-sum1} by
  \begin{equation}
    \label{Gamma-sum2}
    \begin{aligned}
      S_{\Gamma,\eps} & = \sum_{m\in E_d} \Big( \sum_{k \in \N} \sum_{\eta \in 2^m\Gamma} I_{C_k}(\eta) \prod_{j=1}^d \frac {\omega_{r,\eps}( 2^{-m_j}\eta_j )}{\abs{\eta_j} }\Big)                                     
               \\
                      & \ll_d \omega_\infty \sum_{m\in E_d} \sum_{k \in \N} \Big( \Big( \sum_{\eta \in 2^m\Gamma} I_{C_k}(\eta) \Big) \frac{h_{r,\eps}(c_\Gamma 2^{-m_1}k)}{c_{\Gamma} \4 k} \Big) \prod_{j=2}^d h_{r,\eps}(2^{-m_j}),
    \end{aligned}
  \end{equation}
  where $h_{r,\eps}(x) \tdefi \lambda_{r,\eps} \abs{x}^{\frac{1}{2}} I(\abs{x} <  1)+ h( \eps r x) I(\abs{x} \ge 1)$. In order to perform the summation in $k$ and $\eta$ in \eqref{Gamma-sum2} we first observe that
  \begin{equation}
    \label{box1}
    \sum_{\eta \in 2^m\Gamma} I_{C_k}(\eta) \le 1.
  \end{equation}
  Proof of \eqref{box1}: Assume that two different lattice points $\eta, \eta'\in 2^m\Gamma $ lie in $C_k$. Then we have $\abs{\eta_1-\eta'_1} < c_{\Gamma}$ and $\max_{2\le j\le d}\abs{\eta_j-\eta'_j} < \sqrt{d}$. Since 
$\eta-\eta'\in 2^m \Gamma\setminus \{0\}$ implies $\abs{\eta_2- \eta'_2 }\cdots \abs{\eta_d- \eta'_d }\ge (\Nm \Gamma)/c_{\Gamma}=d^{(d-1)/2}2^{(d-1)}$ and hence $\abs{(\eta_2- \eta'_2) } \ge 2 \sqrt{d}$ for some $j\ge 2$, we get at a contradiction which proves \eqref{box1}.\par
  Estimating the following sum in $k$ by an integral, we obtain
  \begin{equation}
    \label{box2}
    \sum_{k=1}^\infty \frac{h_{r,\eps}(\alpha \4 k)}{k} \ll \lambda_{r,\eps} I(\alpha < 1) + \log \Big(1+\frac{2}{\alpha r \eps}\Big) \defi \bar{h}_{r,\eps}(\alpha).
  \end{equation}
  Hence, making use of \eqref{box1} and \eqref{box2} in \eqref{Gamma-sum2}, shows that
  \begin{equation}
    \textstyle
    \label{Lemma7.5:FinStep1}
    S_{\Gamma,\eps} \ll_d \omega_\infty (c_\Gamma)^{-1} \sum_{m \in E_d} H(2^{-m}),
  \end{equation}
  where $2^{m} \tdefi (2^{m_1},\ldots,2^{m_d})$ and $H(x)\tdefi \bar{h}_{r,\eps}(c_\Gamma x_1) h_{r,\eps}(x_2)\cdots h_{r,\eps}(x_d)$.\par Let $E_d'$ denote the subset of $E_d$ consisting of all lattice points $(m_1,\ldots,m_d) \in E_d$ with $m_1 \leq 0$. We claim that
  \begin{equation}
    \label{Lemma7.4:SimpleCase}
    \sum_{m \in E_d'} H(2^{-m}) \ll_d \big(\lambda_{r,\eps} + \log(1 +\tfrac{1}{\Nm(\Gamma) r \eps }) \big) \lambda_{r,\eps}^{d-1}.
  \end{equation}
  Proof of \eqref{Lemma7.4:SimpleCase}: Let $m \in E_d' \setminus \{0\}$. 
Assume for definiteness that $m_1, \ldots, m_{l-1} \leq 0$ and $m_l, \ldots, m_d > 0$. By definition of $E_d$ we get $2 \sum_{j=l}^m m_j= \sum_{j=1}^d \abs{m_j} \ge \norm{m}_2$. Since $h_{r,\eps}(2^{-k}) \le 1$ for $k \le 0$ and otherwise $h_{r,\eps}(2^{-k}) = \lambda_{r,\eps} 2^{-k/2}$, we obtain
  \begin{align*}
    H(2^{-m}) & \ll_d \big(\lambda_{r,\eps} + \log(1 +\tfrac{1}{\Nm(\Gamma) r \eps }) \big) \lambda_{r,\eps}^{d-l} { \textstyle  \prod_{j=l}^d} 2^{-m_j/2} \\
              & \ll_d \big(\lambda_{r,\eps} + \log(1 +\tfrac{1}{\Nm(\Gamma) r \eps }) \big) \lambda_{r,\eps}^{d-l} 2^{-\norm{m}/4}.
  \end{align*}
  Thus, splitting the sum according to the number of positive coordinates 
and then summing over the $(d-1)$-dimensional lattice $E_d$ yields \eqref{Lemma7.4:SimpleCase}.

  In order to bound the sum over the complement of $E_d'$, we again split 
the sum according to the number of positive coordinates. For simplicity, we may assume that $m_1,m_2,\ldots,m_l >0$ and $m_{l+1},\ldots,m_d \leq 0$. Similar to the previous case, we find that
  \begin{equation*}
    H(2^{-m})\ll_d \big(\norm{m}+\lambda_{r,\eps} + \log(1 {+} \tfrac{1}{\Nm(\Gamma) r \eps }) \big) \lambda_{r,\eps}^{l-1} \Big( {\textstyle \prod_{j=2}^l } 2^{-\frac{m_j}{2}} \Big) \min ( 1, (r\eps)^{-dp}\4 2^{-p\norm{m}/2 }).
  \end{equation*}
  If we parameterize the $(d-1)$-dimensional lattice $E_d$ by $(m_1,\bar{m})$, where $m_1 = - (m_2+ \ldots + m_d)$ and $\bar{m}=(m_2,\ldots,m_d) 
\in \Z^{d-1}$, and split the summation into a ball of radius $\norm{\bar{m}}_2 \le R_{\eps}\tdefi 3d\4 \log (2+(r\eps)^{-1})$ and its complement, where $(r\eps)^{-dp}\4 2^{-p\norm{m}_2/2} \le (r\eps)^{-dp}\4 2^{-p\norm{\bar{m}}_2/2} \le 1$, we can bound the sum corresponding to a fixed $l$ by
  \begin{align*}
     & \ll_d \lambda_{r,\eps}^{l-1} \Big(  \sum_{ \norm{\bar{m}}_2 \le R_{\eps}} (\bar{\lambda}_{r,\eps,\Gamma} + \norm{\bar{m}})\prod_{j=2}^l 2^{-m_j/2} + \sum_{ \norm{\bar{m}}_2 > R_{\eps}} (\bar{\lambda}_{r,\eps,\Gamma} + \norm{\bar{m}}) (r\eps)^{-dp}\4 2^{-p\norm{\bar{m}}_2 /2 } \Big) \\
     & \ll_d \lambda_{r,\eps}^{l-1} \Big( \bar{\lambda}_{r,\eps,\Gamma} \log(2+ \tfrac{1}{r\eps})^{d-1-(l-1)}+  \bar{\lambda}_{r,\eps,\Gamma} \Big) \ll_d \lambda_{r,\eps}^{d-1} \bar{\lambda}_{r,\eps,\Gamma},
  \end{align*}
  where we have estimated the sums by comparison with the corresponding integrals. Using this estimate for each $l=1,\ldots,d-1$ together with \eqref{Lemma7.4:SimpleCase} in \eqref{Lemma7.5:FinStep1} yields the bound \eqref{Fourier-sum:eq}.
\end{proof}



\subsection{Applications of Theorem \ref{maintheorem}}
\label{subsection:Applications}
We start by smoothing the indicator function of the region $\Omega$. We choose weights $\specialv=\specialv_{\pm \eps}$ as defined in \eqref{omega_eps} with $\eps \in (0,\eps_0]$ and the related $\zeta=\zeta_{\eps}$, see Section \ref{ftpoly}, corresponding to parallelepipeds $\Omega=B^{-1}[-1,1]^d$ satisfying $Q_+ \le B^T B \leq c_B Q_+$, compare \eqref{eq:sec7:Def:Omega}. Recalling \eqref{omega-smooth}, where we have used Lemma \ref{l3} to estimate the $\eps$-smoothing error, yields 
a total error\index{D@ $\Delta_r := \abs{ \vol H_r -  \volu H_r}$, lattice point deficiency}
\begin{equation}
  \label{pluseps}
  \Delta_r \defi \abs{\vol(E_{a,b}\cap r\Omega)-\v (E_{a,b}\cap r\Omega)} 
\ll_d d_Q (b-a)\eps r^{d-2}+ \max_{\pm} \abs{R( I_{E_{a,b}} \specialv_{\pm \eps,r})}.
\end{equation}
Now we can apply Theorem \ref{maintheorem} in order to bound the latter remainder $\abs{R( I_{E_{a,b}} \specialv_{\pm \eps,r})}$ as follows. In \eqref{finalb0} we shall estimate $\norm{\wh \zeta_{\eps}}_{*,r}$ by using $\norm{\specialv_{\eps}}_Q \ll_d d_Q$ of Lemma \ref{l3}, $\norm{\wh \zeta_{\eps}}_1\ll_d (\log{\eps^{-1}})^d$ of Lemma \ref{l2} and
\begin{equation}
  \label{proof:th2.1:admis}
  \norm{\wh \zeta_{\eps}}_{*,r} \ll_d q^{d/4} \Big( (\tfrac{q}{q_0})^{d/2} \log(\eps^{-1})^d + q_0^{-d/4} \4 c_B^{d/2} \4 \lambda_{r,\eps}^{d-1} \4 \tfrac{\bar{\lambda}_{r,\eps,\Gamma}}{\Nm(\Gamma)} \Big)
\end{equation}
of Lemma \ref{omega-remainder} for admissible regions $\Omega$, i.e.\ \eqref{Def:Admissible} holds, to get
\begin{equation}
  \label{fullbound}
  \begin{aligned}
    \Delta_r
    \ll_{\beta,d} \  & d_Q r^{d-2}  \Big(  \eps (b-a) + w + a_Q (\log \tfrac{1}{\eps} )^d \rho_{Q,b-a}^{w}(r) \Big) \\
                     & \4 + d_Q \4 q^{d/4}  r^{d/2}  \Big( (\tfrac{q}{q_0})^{d/2} \4 \log(\eps^{-1})^d + q_0^{-d/4} \4 c_B^{d/2} \4 \lambda_{r,\eps}^{d-1} \4 \tfrac{\bar{\lambda}_{r,\eps,\Gamma}}{\Nm(\Gamma)} \Big)  \log \big(1 + \tfrac{b-a}{q_0^{1/2} r}\big),
  \end{aligned}
\end{equation}
where $a_Q \tdefi q \4 c_Q = q \abs{\det{Q}}^{1/4-\beta/2} = C_Q (d_Q)^{-1}$\index{A@ $a_Q \tdefi q \4 c_Q $}, provided that $0 <w < (b-a)/4$. 
This bound holds for admissible parallelepipeds $\Omega$ only. If $\Omega$ is not admissible, then we have to replace the smoothing error \eqref{proof:th2.1:admis} by
\begin{equation}
  \label{fullbound:replace}
  \norm{\wh \zeta_{\eps}}_{*,r} \ll_d q^{d/4} \big( (q/q_0)^{d/2} \log(\eps^{-1})^d + d_Q \4 q^{d/2} \4 (c_B)^{(d+1)/2} \4 \eps^{-d} \big),
\end{equation}
that is \eqref{zetaperiod-notadmiss} of Lemma \ref{omega-remainder}. With 
these bounds we are ready to prove the main statements on the lattice point remainder for hyperbolic shells.


\begin{proof}[Proof of Corollary \ref{variable_smooth}]
  \label{proof:corollary:variable_smooth}
  For wide shells, i.e.\ $b-a > q$, we optimize \eqref{fullbound} in the smoothing parameter $w$ first by choosing $w = \mathrm{W}(qT_{+}/2)^2/T_{+}$, where $\mathrm{W}$ denotes the upper branch, defined on the interval $(-\mathrm{e}^{-1},\infty)$, of the inverse function of $x \mapsto x e^x$. (The function $W$ is also known as the Lambert-$W$-function, see \cite{CGHFK:1996} for more details and some applications.)
  
  Since $x \mapsto W(x)^2/x$ has a global maximum at $x= \mathrm{e}$ with value $\mathrm{e}^{-1}$, we find $w \le q/(2 \e) < (b-a)/4$ as required in the restrictions \eqref{choice_wab}. This leads to the partial bound
  \begin{equation*}
    d_Q \4 w + C_Q c_Q^{-1} (T_+ w)^{-1/2} \4 \e^{- (T_+w)^{1/2}} \ll d_Q 
\tfrac{\mathrm{W}(q T_{+}/2)^2}{T_{+}} \ll d_Q \tfrac{\log(q T_{+}+1)^2}{T_{+}},
  \end{equation*}
  where we used that $W(x) \leq \log(x+1)$ and $W(x)^{-1} \exp(-W(x)) = x^{-1}$. Next, we calibrate the $\eps$-dependent terms in \eqref{fullbound} by choosing $\eps=T_{-}^{\frac{d}{2}-2-\delta} \4 (b-a)^{-1} /15$. Again, this choice satisfies the required restrictions, i.e.\ $\eps \leq \eps_0 = 1/15$. Because of
  \begin{align*}
    \eps (b-a)                                      & \le a_Q \4 (b-a)_q \4 c_Q \4 T_{-}^{\frac{d}{2}-2-\delta}, \quad \log \eps^{-1} \ll \log(r \! + \! 1) \quad \text{and}   \\
    \frac{\bar{\lambda}_{r,\eps,\Gamma}}{\log(r+1)} & \ll \max \Big\{ 1 , 
\frac{ \log(2+\tfrac{r^{d+1}}{\Nm(\Gamma)})}{\log(r+1)} \Big\} \ll_d \log(2+\tfrac{1}{\Nm(\Gamma)}),
  \end{align*}
  compare the definition in Lemma \ref{omega-remainder}, we can simplify \eqref{fullbound} to
  \begin{equation}
    \label{proof:cor2.4:eq1}
    \begin{aligned}
      \Delta_r \ll_{\beta,d} \  & d_Q \4 r^{d-2}  \rho_{Q,b-a}^{\mathrm{hyp}+}(r) \\
      & \4 + d_Q \4 q^{\frac{d}{4}} \4 r^{\frac{d}{2}} \log(r \hspace{-1pt} + \hspace{-1pt} 1)^d \big( (\tfrac{q}{q_0})^{\frac{d}{2}} \hspace{-1pt} + \hspace{-1pt}  \tfrac{c_B^{d/2} q_0^{-d/4}}{\Nm(\Gamma)} \log(2 \hspace{-1pt} + \hspace{-1pt} \tfrac{1}{\Nm(\Gamma)})  \big) \hspace{-1pt} \log \hspace{-1pt} \Big(1 \hspace{-1pt} + \hspace{-1pt} \tfrac{b-a}{q_0^{1/2} r}\Big),
    \end{aligned}
  \end{equation}
  where\index{R@ $\rho_{Q,b-a}^{\mathrm{hyp}+}(r)$}
  \begin{equation*}
    \begin{aligned}
      \rho_{Q,b-a}^{\mathrm{hyp}+}(r) \defi {\inf}^*_{T_{+},T_{-}} \Big\{ \log \big( (b{-}a)T_{-}^{-(\frac{d}{2}-2-\delta)}{+}1 \big)^d \Big(  a_Q \4 q^{(2\beta d -1)/2} ( c_Q T_{-}^{\frac{d}{2}-2-\delta} {+} \gamma_{[T_{-},1],\beta}(r)) & \\
      + a_Q \gamma_{(1,T_{+}],\beta}(r) \log(T_{+}+1)  + \tfrac{\log(q T_{+}+1)^2}{T_{+}} \Big) \Big\}                                                                                                                     
 &
    \end{aligned}
  \end{equation*}
  and the infimum is taken over all $T_{-} \in [\imbound,1]$, $T_{+} \geq 1$. This proves the first part of Corollary \ref{variable_smooth}. Next, we consider the case of thin shells, i.e.\ $b-a \leq q$. Here we take $\eps = T_{-}^{\frac{d}{2}-2-\delta}/15$ and $w= T_{-}^{\frac{d}{2}-2-\delta} (b-a)/4$ in \eqref{fullbound}, noting that $d_Q  ( w + \eps \4 (b-a)) \leq  a_Q (b-a) c_Q T_{-}^{\frac{d}{2}-2-\delta}$, in order to get the bound \eqref{proof:cor2.4:eq1}, whereby the factor $\rho_{Q,b-a}^{\mathrm{hyp}+}(r)$, depending on the Diophantine properties of $Q$, has to be replaced by\index{R@ $\rho_{Q,b-a}^{\mathrm{hyp}-}(r)$}
  \begin{equation*}
    \begin{aligned}
      \rho_{Q,b-a}^{\mathrm{hyp}-}(r) \defi {\inf}_{T_{-},T_{+}}^* \Big\{ a_Q \log \big(1 + T_{-}^{-(\frac{d-4}{2}-\delta)}\big)^d \Big( & (b-a) ( c_Q \4 T_{-}^{\frac{d-4}{2}-\delta}+ \gamma_{[T_{-},1],\beta}(r) \Big) \\
      & \ + \gamma_{(1,T_+],\beta}(r) \4 (\log( (b-a)^* \4 T_+)+1)) \Big\}.
    \end{aligned}
  \end{equation*}
  In the last equation the infimum is taken over all $T_{-} \in [\imbound,1]$ and $T_{+} \geq 1$ with
  \begin{equation*}
    T_{+} \geq 4 (b-a)^{-1} T_{-}^{-(\frac{d}{2}-2-\delta)}  \max \{1,\log \big( c_Q^2 (b-a) T_{-}^{\frac{d}{2}-2-\delta}\big)^2\},
  \end{equation*}
  where the last condition ensures that
  \begin{equation*}
    c_Q^{-1} (T_+ w)^{-1/2} \4 \e^{- (T_+w)^{1/2}} \le c_Q (b-a) T_{-}^{\frac{d}{2}-2-\delta}.
  \end{equation*}
  Finally, we note that Corollary \ref{irr-dio} implies that $\gamma_{[T_{-},1],\beta}(r)\to 0$ and also $\gamma_{[1,T_{+}],\beta}(r)\to 0$ for $r\to\infty$ and any fixed $T_{-} \in [\imbound,1]$, $T_{+} \geq 1$, when $Q$ is irrational. Thus, we conclude that $\rho_{Q,b-a}^{\mathrm{hyp}+}(r)\to 0$, resp.\ $\rho_{Q,b-a}^{\mathrm{hyp}-}(r)\to 0$, for $r\to\infty$ and fixed $b-a$.
\end{proof}

\begin{corollary}
  \label{variable_smooth:non-admissible}
  Consider an indefinite quadratic form $Q$ in $d \geq 5$ variables and a (not necessary admissible) parallelepiped $\Omega$ satisfying \eqref{eq:sec7:Def:Omega} and $\max\{\abs{a},\abs{b}\} \le c_0 r^2$, where $c_0>0$ is chosen as in Lemma \ref{l3}. Then for all $b-a \leq 1$
  \begin{equation*}
    \Delta_r \ll_{\beta,d} d_Q r^{d-2} \big( \rho_{Q,b-a}^{\mathrm{hyp}*}(r) \hspace{-0.5pt} + \hspace{-0.5pt} (b-a) \4 r^{1-d/2} \4 q^{(d-2)/4} \log(1 \hspace{-1pt} + \hspace{-1pt} r)^d (q/q_0)^{(d+1)/2}  (c_B)^{(d+1)/2} \big),
  \end{equation*}
  where $\rho_{Q,b-a}^{\mathrm{hyp}*}$ is defined in \eqref{cor6.3:eq}. In particular, for irrational $Q$ we have $\rho_{Q,b-a}^{\mathrm{hyp}*}(r) \rightarrow 0$ for $r \rightarrow \infty$, provided that $b-a$ is fixed.
\end{corollary}

\begin{proof}
  We shall argue similar as in the previous proof of Corollary \ref{variable_smooth}, but here we can only use \eqref{fullbound:replace} to bound $\norm{\wh \zeta_{\eps}}_{*,r}$, since $\Omega$ is not necessarily admissible. Thus, we have to replace the error bound \eqref{fullbound} for the lattice remainder by
  \begin{equation}
   \label{eq:var_smooth:latticeremainder}
   \begin{aligned}
    \Delta_r \ll_{\beta,d} \  & d_Q r^{d-2} \Big(  \eps (b-a) + w + a_Q (\log \tfrac{1}{\eps} )^d \rho_{Q,b-a}^{w}(r) \Big) \\
    & \4 + d_Q \4 q^{d/4}  r^{d/2}  \Big( (\tfrac{q}{q_0})^{d/2} (\log \tfrac{1}{\eps})^d + d_Q \4 q^{d/2} \4 (c_B)^{(d+1)/2} \4 \eps^{-d} \Big) \log \big(1 + \tfrac{b-a}{q_0^{1/2} r}\big).
   \end{aligned}
 \end{equation}
 Now the right-hand side can be optimized by taking
  \begin{equation*}
    \eps = (15 \log(1+T_{-}^{-(\frac{d}{2}-2-\delta)}))^{-1} \q \ \text{and} \q \ w = T_{-}^{\frac{d}{2}-2-\delta} (b-a)/4
  \end{equation*}
  and this leads to the bound 
  \begin{align*}
    \Delta_r \ll_{\beta,d} d_Q r^{d-2} \rho_{Q,b-a}^{\mathrm{hyp}*}(r)  + d_Q q^{d/4} r^{d/2} \big( & \log(1+r)^d (q/q_0)^{d/2} \\
    & + d_Q q^{d/2} (c_B)^{(d+1)/2} \log(1 \! + \! r)^d \big)  \log \Big(1+ \tfrac{\abs{b-a}}{q_0^{1/2} r} \Big),
  \end{align*}
  where\index{R@ $\rho_{Q,b-a}^{\mathrm{hyp}*}(r)$}
  \begin{equation}
    \label{cor6.3:eq}
    \begin{aligned}
      \rho_{Q,b-a}^{\mathrm{hyp}*}(r) \hspace{-1mm} \defi \hspace{-1mm}  
{\inf} \Big\{a_Q \log (1{+}T_{-}^{-(\frac{d}{2}-2-\delta)})^d \Big( (b-a) 
\4 ( c_Q \4 T_{-}^{(\frac{d}{2}-2-\delta)}+ \gamma_{[T_{-},1],\beta}(r)) \q &        \\
      + \gamma_{(1,T_+],\beta}(r) \log((b-a) \4 T_+) \Big) + \tfrac{b-a}{\log(1 \! + \! T_{-}^{-(\frac{d}{2}-2-\delta)})} & \Big\}
    \end{aligned}
  \end{equation}
  and the infimum is taken over all $T_{-} \in [\imbound,1]$ and
  \begin{equation*}
    T_+ \geq  4 (b-a)^{-1} T_{-}^{-(\frac{d}{2}-2-\delta)} \max \{1, \log(c_Q^2 (b-a) T_{-}^{\frac{d}{2}-2-\delta})^2 \}. \qedhere
  \end{equation*}
\end{proof}

The next corollary provides a lower bound for the number of lattice points and is useful for proving quantitative bounds in the Oppenheim conjecture.

\begin{corollary}
  \label{cor:6.4}
  For the special choice $B= Q_{+}^{1/2}$, i.e.\ $\Omega = Q_{+}^{-1/2} [-1,1]^d$ and $c_B=1$, and all $\max\{\abs{a},\abs{b}\} \leq r^2/5$ and $b-a \leq 1$ there exists constants $b_{\beta,d} >0$ and $\tilde{b}_{\beta,d}>0$, depending on $\beta$ and $d$ only, such that for all $r \geq \tilde{b}_{\beta,d} \4 q^{1/2} (q/q_0)^{(d+1)/(d-2)}$
  \begin{equation}
    \label{eq:cor:6.4}
    \Delta_r \le \frac{\volu H_r}{5} + b_{\beta,d} \4 d_Q \4 r^{d-2} \rho_{Q,b-a}^{\mathrm{hyp}**}(r)
  \end{equation}
  where $c_Q = \abs{\det{Q}}^{1/4-\beta/2}$, $a_Q = q \4 c_Q$ and\index{R@ $\rho_{Q,b-a}^{\mathrm{hyp}**}(r)$}
  \begin{equation}
    \rho_{Q,b-a}^{\mathrm{hyp}**}(r) \hspace{-1mm} \defi \hspace{-1mm} {\inf} \{ a_Q \big( (b-a) (c_Q T_{-}^{\frac{d}{2}-2-\delta} + \gamma_{[T_{-},1],\beta}(r)) + \gamma_{(1,T_{+}],\beta}(r) \log((b-a) \4 T_{+}) \big)\}
  \end{equation}
  and the infimum is taken over all $T_{-} \in [\imbound,1]$ and $T_{+} \ge 1$ with
  \begin{equation*}
    T_{+} \ge C_{\beta,d} \frac{1}{(b-a)} \max \Big\{ \log \Big( \frac{b-a}{q \, c_{\beta,d} } \Big)^2, 1 \Big\}
  \end{equation*}
  and $C_{\beta,d}, c_{\beta,d} \ge 1$ are constants depending on $d$ and 
$\beta$ only.
\end{corollary}

\begin{proof}
  Here we only consider the special region $\Omega = Q_{+}^{-1/2} [-1,1]^d$, i.e.\ $B = Q_{+}^{1/2}$ and thus \eqref{eq:sec7:Def:Omega} is valid with $c_B =1$. Since $\Omega$ is not necessarily admissible, we have to argue as in the previous proof (of Corollary \ref{variable_smooth:non-admissible}): Starting with the estimate \eqref{eq:var_smooth:latticeremainder}, we can take $\eps = (30 \, a_{d} \, b_{\beta,d})^{-1}$ and $w= (b-a) \eps$ in the optimization procedure, where $a_d \ge 1$, resp.\ $b_{\beta,d} \ge 1$, denotes the implicit constant in \eqref{volume_bound} (see Lemma \ref{l3}), resp.\ \eqref{eq:var_smooth:latticeremainder}. (Of course, we have $\eps \in (0,\eps_0]$ and $0< w < (b-a)/4$ as required.) This yields
  \begin{equation*}
   \begin{aligned}
    \Delta_r \leq & \frac{\volu H_r}{15}  + b_{\beta,d} d_Q r^{d-2} a_Q (\log \tfrac{1}{\eps} )^d \rho_{Q,b-a}^{w}(r) + \bar{b}_{\beta,d} (b-a) d_Q \4 q^{\frac{d-2}{4}}  r^{d/2-1} \Big(\frac{q}{q_0}\Big)^{\frac{d+1}{2}},
   \end{aligned}
 \end{equation*}
 where $\bar{b}_{\beta,d} \tdefi b_{\beta,d} ( \eps^{-d} + \log(\eps^{-1})^d)$ depends on $\beta$ and $d$ only. Again referring to Lemma \ref{l3}, 
we also see that
 \begin{equation*}
   \bar{b}_{\beta,d} (b-a) d_Q \4 q^{(d-2)/4}  r^{d/2-1} \Big(\frac{q}{q_0}\Big)^{(d+1)/2} \leq \frac{\volu H_r}{15}
 \end{equation*}
 if we choose $r \geq \tilde{b}_{\beta,d} \4 q^{1/2} (q/q_0)^{(d+1)/(d-2)}$ with $\tilde{b}_{\beta,d} = (15 a_d \bar{b}_{\beta,d})^{-1}$. Finally, we make the restriction $T_+ \ge w^{-1} \max\{\log( (15 a_d \bar{b}_{\beta,d})^{-1} q^{-1} (b-a) )^2,1\}$ to ensure that
 \begin{equation*}
   b_{\beta,d} \, (\log \eps^{-1})^d \, q \, r^{d-2} d_Q \, (T_{+} w)^{-1/2} \exp(-\abs{T_{+} w}^{1/2}) \le \volu H_r/15.
 \end{equation*}
 Collecting the remaining terms proves \eqref{eq:cor:6.4}.
\end{proof}

Now we consider elliptic shells as well and optimize the lattice remainder as in the case of `wide shells'. In contrast to the previous cases, the 
error caused by the smoothing of the region $\Omega$ is not present here.

\begin{proof}[Proof of Corollary \ref{positive}]
  In the case of ellipsoids, i.e.\ $Q$ is a positive definite form, we choose the (not necessary admissible) parallelepiped $\Omega\tdefi B^{-1} [-1,1]^d$ with $B= Q_+^{1/2}$ and $r=\sqrt{2b} \geq q^{1/2}$, resp.\ $2b=r^2$, $a=0$ and $\eps= 1/15$. Then \eqref{eq:sec7:Def:Omega} is satisfied with $c_B =1$ and $E_{0,b} \subset r\Omega$, i.e.\ $H_r := E_{a,b} \cap r \Omega = E_{a,b}$. Moreover, since $E_{0,b}$ does not intersect $r (\partial \Omega)_{2 \eps}$ (the $2 \eps r$-boundary of $r \Omega$ as defined in \eqref{omega_eps}), we get an error $R_{\eps,r}=0$ for smoothing the indicator function of $r \Omega$. Hence, we may remove the term proportional to $(b-a)\eps$ in \eqref{pluseps}. Note that apart from Lemma \ref{l3} the \textit{indefiniteness} of $Q$ has not been used in all arguments so far. In contrast to the case of hyperbolic shells, we optimize \eqref{finalb0} in $w$ first. Again including the bound $\norm{\specialv_{\eps}}_Q \ll_d d_Q$ of Lemma \ref{l3} and here taking $w=\mathrm{W}(q T_+/4)^2/T_+$, where $\mathrm{W}$ denotes the upper branch of the Lambert-$W$-function (for more details on the Lambert-$W$-function see the proof of Corollary \ref{variable_smooth} on p.\ \pageref{proof:corollary:variable_smooth}), and noting that $w \le q/(4e) < (b-a)/4$, leads (as in the proof of Corollary \ref{variable_smooth}) to the bound
  \begin{equation}
    \begin{aligned}
      \Delta_r \ll_{\beta,d} r^{d-2}  \Big( \! C_Q \big( \4 q^{(2 \beta d 
-1)/2} ( c_Q \4 T_{-}^{\frac{d}{2}-2-\delta} \! + \gamma_{[T_{-},1],\beta}(r)) \! + \gamma_{(1,T_+],\beta}(r) \log(T_{+} \! + \! 1) \big) \  &   \\
      + d_Q \tfrac{\log(1+q \4 T_+)^2}{T_+} \Big) + d_Q q^{d/4} r^{d/2} ( 
( q /q_0 )^{d/2} \! + \!  d_Q q^{d/2}) \log \Big( 1\!+ \! \tfrac{r}{q_0^{1/2}} \Big)                                                         & ,
    \end{aligned}
  \end{equation}
  where $T_{-} \in [\imbound,1]$ and $T_+ \geq 1$. This can be rewritten as
  \begin{equation*}
    \Delta_r \ll_{\beta,d} d_Q \4 r^{d-2} \4 \rho_Q(r) + d_Q \4 q^{d/4} \4 r^{d/2} (q/q_0)^{d/2} \log(1+r/q_0^{1/2})
  \end{equation*}
  with\index{R@ $\rho_Q^{\mathrm{ell}}(r)$}
  \begin{equation*}
    \rho_Q^{\mathrm{ell}}(r) \hspace{-1mm} \defi \hspace{-1mm} \inf \Big\{ a_Q \big( q^{\beta d -\tfrac{1}{2}} (c_Q \4 T_{-}^{\frac{d}{2}-2-\delta} \! + \gamma_{[T_{-},1],\beta}(r)) \! +  \gamma_{(1,T_+],\beta}(r) \log(T_{+} \! + \! 1) \big) \! + \tfrac{\log(1+q T_+)^2}{T_+} \Big\},
  \end{equation*}
  where the infimum is taken over all $T_{-} \in [\imbound,1]$ and $T_+ \geq 1$. Note that as in the indefinite case $\lim_{r\to \infty} \rho_Q^{\mathrm{ell}}(r)=0$ if $Q$ is irrational by Corollary \ref{irr-dio}. This proves Corollary \ref{positive}. Furthermore, we remark that $\volu H_r 
=\volu (r\Omega \cap E_{0,b}) =  d_Q\4 \omega_d \4 r^{d}$, where $\omega_d$ denotes the volume of the unit $d$-ball.
\end{proof}

Similar arguments can be used in order to obtain related bounds for both wide ($b-a >r$) and narrow ($b-a < r$) shells in the case of ellipsoidal shells $E_{a,b}$.\\[2mm] Given a quadratic form $Q$ of Diophantine type $(\kappa,A)$, i.e.\ $Q$ satisfies \eqref{diophant}, we shall apply Corollary \ref{irr-dio} in order to estimate the Diophantine factors explicitly. 
Hereby, we prove quantitative bounds in the Oppenheim conjecture (for indefinite quadratic forms $Q$ of Diophantine type $(\kappa,A)$) by comparing the volume with the corresponding lattice sum.

\begin{proof}[Proof of Corollary \ref{th:diophante:smallzeros}]
  We begin by applying Corollary \ref{cor:6.4} with $b = - a = \eps$ and $\beta = 2/d+\delta'/d$ for an appropriate $\delta'>0$: Taking $T_{-} \asymp_{\beta,d}  q^{-1/(d(1/2-\beta)) } \4 \abs{\det{Q}}^{-1/d}$, so that $b_{\beta,d} (b-a) d_Q \4 r^{d-2} a_Q \4 c_Q \4 T_{-}^{d(1/2-\beta)} 
\leq (\volu H_r)/5$ holds, yields the lattice remainder bound
  \begin{equation*}
      \Delta_r \leq \frac{2\volu H_r}{5} + r^{d-2} C_Q \4 b_{\beta,d} \4 ( 2 \eps \4 \gamma_{[T_{-},1],\beta}(r) + \gamma_{(1,T_{+}],\beta}(r) \log(2\eps T_{+}) ).
  \end{equation*}
  This estimate is valid provided that $r \gg_{\beta,d} (q/q_0)^{(d+1)/(d-2)} q^{1/2+2/(d-4)+\delta}$. Note that we have $T_{-} \in[\imbound,1]$ as required and that the assumptions of Corollary \ref{cor:6.4} are satisfied as well. Next we calibrate the parameter $T_{+}$ by taking
  \begin{equation*}
    T_{+} \asymp_{\beta,d} \eps^{-1} \max\{1, \log(2 \eps (q c_{\beta,d})^{-1})^2\}.
  \end{equation*}
  Since $Q$ is of Diophantine type $(\kappa,A)$, we can use Corollary \ref{irr-dio} in order to find that
  \begin{equation*}
    \gamma_{[T_{-},1],\beta}(r) \ll_{Q,\beta,d} A^{-\frac{1-2\beta}{2(\kappa+1)}} \4 r^{-\frac{1-2\beta}{\kappa+1}}
  \end{equation*}
  and also that
  \begin{equation*}
    \gamma_{(1,T_{+}],\beta}(r) \ll_{Q,\beta,d} A^{-\frac{1-2\beta}{2(\kappa+1)}} r^{-\frac{1-2\beta}{\kappa+1}} (\eps^{-1} \log(\eps^{-1}))^{\frac{\kappa}{\kappa+1} (\frac{1}{2}-\beta)}.
  \end{equation*}
  In view of \eqref{volume_bound}, we may increase $r \gg_{Q,\beta,d} \max \{A^{-1},1\}$ to get
  \begin{equation*}
    2 b_{\beta,d} \4  C_Q \4 r^{d-2} \gamma_{[T_{-},1],\beta}(r)  \leq (\volu H_r)/5.
  \end{equation*}
  Now, we choose $r \asymp_{A,Q,\delta,d} \eps^{-(2d + 3 \kappa d - 4 \kappa)/(2d-8) - \delta} $ in order to obtain
  \begin{equation*}
    b_{\beta,d}  \4 C_Q \4 r^{d-2} \log(2\eps T_{+}) \4 \gamma_{(1,T_{+}],\beta}(r) \leq (\volu H_r)/5.
  \end{equation*}
  All in all, we have
  \begin{equation*}
    5 \vol H_r \ge \volu H_r \gg_d d_Q \4 \eps \4 r^{d-2}.
  \end{equation*}
  Since $(2d + 3 \kappa d - 4 \kappa)/(2d-8) \geq 1/(d-2)$ holds if $d \geq 5$, we find that $\vol H_r > 1$. This means that there exists at least 
one non-zero lattice point $m \in \Z^d$ satisfying both $\abs{Q[m]} < \eps$ and also $\norm{Q_{+}^{1/2}m} \ll_d r$.
\end{proof}

We can argue similarly to investigate the density of values of a quadratic form:

\begin{proof}[Proof of Corollary \ref{corgaps}]
   It is sufficient to prove that $\volu_{\Z^d}(r\Omega\cap E_{a,b})>0$ for any $\max\{\abs{a}, \abs{b}\} \leq c_0 r^2/2 $, where $c_0$ is as in Lemma \ref{l3}, with $r^{-\nu_0+ \delta} =  b-a$ for $r \geq c_{\delta,d,\Omega,Q,A,\kappa}$ and a sufficiently large constant $c_{\delta,d,\Omega,Q,A,\kappa} >1$. In particular, we consider small shells, i.e.\ $b-a \leq 1$. Repeating the proof of Corollary \ref{cor:6.4}, we see that Corollary \ref{cor:6.4} is also valid for arbitrary parallelepipeds satisfying \eqref{eq:sec7:Def:Omega}, but then the constants depend additionally on the scaling parameter $c_B \geq 1$. Also repeating the previous proof (of Corollary \ref{th:diophante:smallzeros}) in this situation shows that we can take $r = c_{\delta,d,\Omega,Q,A,\kappa} (b-a)^{-1/\nu_0}$, where $\nu_0 \tdefi \frac{2(d-4)}{2d + 3 \kappa d - 4 \kappa}$, to ensure that $\volu_{\Z^d}(r\Omega\cap E_{a,b})>0$.
\end{proof}

Using the Diophantine estimates for quadratic forms $Q$ of Diophantine type $(\kappa,A)$, we can estimate $\rho_{Q,b-a}^{\mathrm{hyp}+}(r)$ and $\rho_{Q,b-a}^{\mathrm{hyp}-}(r)$ in Corollary \ref{variable_smooth} explicitly as follows.

\begin{proof}[Proof of Corollary \ref{diophantex}]
  First, we consider `wide shells', i.e.\ $b-a \geq q$. By applying Corollary \ref{irr-dio}, we can bound the Diophantine factor from Corollary \ref{variable_smooth} by
  \begin{equation*}
    \begin{aligned}
      \rho_{Q,b-a}^{\mathrm{hyp}+}(r) \ll_d {\inf}_{T_{-},T_{+}}^* \big\{ \log \big( (b-a) T_{-}^{-(\frac{d-4}{2}-\delta)} \! + \!  1 \big)^d \big( 
q \big( q^{\frac{3}{2}+\delta} (a_Q^2 T_{-}^{\frac{d-4}{2}-\delta} \! + \! q^\nu A^{-\nu} \4 T_{-}^{-\nu} r^{- 2\nu } ) &  \\
      + q^\nu A^{-\nu} \4 T_{+}^{ \kappa \nu} r^{-2 \nu} \log(T_{+}+1) \big) + c_Q \tfrac{\log(q \4 T_{+}+1)}{T_{+}} \big) & \! \big\},
    \end{aligned}
  \end{equation*}
  where $\nu \tdefi (1-2\beta)/(2\kappa+2)$ and the infimum is taken over all $T_{-} \in [\imbound,1]$ and $T_{+} \geq 1$. Next we optimize this expression by taking $T_{-} = r^{- 2\nu /(\nu+\sigma)}$ and $T_{+} = r^{(2\nu)/(\kappa \nu +1)}$, where $\sigma \tdefi d(1/2-\beta)$: This parameter choice is permissible, since $T_{-} \in [\imbound,1]$ holds (because of $\sigma \ge \nu$), and thus we obtain
  \begin{equation*}
    \rho_{Q,b-a}^{\mathrm{hyp}+}(r) \ll_{\beta,d} \log(r+1)^d \4 h_Q \4 q^{\frac{3}{2}+ \delta + \nu} (1+ A^{-\nu}) ( r^{-\frac{2 \nu \sigma}{\nu + \sigma}} + r^{-\frac{2\nu}{\kappa \nu +1}} \log(q \4 r +1 )),
  \end{equation*}
  where $h_Q \tdefi q \4 \abs{\det{Q}}^{1/2-\beta}$ (here we avoided to give an optimal estimate in terms of $\abs{\det{Q}}$ to reduce the notational burden). In view of the bound from Corollary \ref{variable_smooth} and \eqref{volume_bound} we get the relative lattice error
  \begin{align*}
    \Big\lvert \frac{\volu_{\Z} H_r}{\volu H_r} -1 \Big\rvert \ll_{Q,\Omega,\beta,d} (b-a)^{-1} \log(r+1)^d \Big( r^{-\frac{2 \nu \sigma}{\nu + \sigma}} & + 
r^{-\frac{2\nu}{\kappa \nu +1}} \log(r +1 ) \\
    & + r^{-\frac{d}{2}+2} \log \big(1 {+} \tfrac{b-a}{r} \big) \Big).
  \end{align*}
  For `thin shells', i.e.\ $b-a \leq q$, we have
  \begin{align*}
    \rho_{Q,b-a}^{\mathrm{hyp}-}(r) \ll_{\beta,d} {\inf}_{T_{-},T_{+}}^* \big\{ \4 h_Q \log \big(1+T_{-}^{-\frac{d-4}{2}+\delta)} \big)^d  \big( (b-a) (T_{-}^{\frac{d}{2}-2-\delta} + q^\nu A^{-\nu} \4 T_{-}^{-\nu} r^{- 2\nu } ) \q & \\
    + q^\nu A^{-\nu} \4 T_{+}^{ \kappa \nu} r^{-2 \nu} (\log((b-a)^*T_{+}) \big)+1)\big)                                                           
                                                                      & \big\},
  \end{align*}
  where the infimum is taken over all $T_{-} \in [r^{-1},1]$ and $T_{+} \geq 1$ satisfying
  \begin{equation*}
    T_{+} \geq 4 (b-a)^{-1} T_{-}^{-(\frac{d}{2}-2-\delta)} \max \Big\{1, \log(c_Q^2 (b-a) T_{-}^{-(\frac{d}{2}-2-\delta)})^2 \Big\}. \qedhere
  \end{equation*}
\end{proof}



\section{Small Values of Quadratic Forms at Integer Points}
\label{sec:small_values:qform}

\noindent Finally we shall prove Theorem \ref{dio-ineq} by using our effective equidistribution results (in form of Corollary \ref{cor:6.4}) together with bounds on small zeros of indefinite integral quadratic forms. Our proof is based on the following strategy: If $Q$ has `good' Diophantine properties, we can compare the volume with the number of lattice points to establish bounds for non-trivial lattice points $m \in \Z^d \setminus \{0\}$ satisfying the Diophantine inequality $\abs{Q[m]} < \eps$. Otherwise $Q$ is near a rational form and here we shall use Schlickewei's bound \cite{Schlickewei:1985} for small zeros of integral quadratic forms. 

\subsection{Integer-valued Quadratic Forms}
\label{section:integer_val:qforms}
In this section we summarize some essential results on small zeros of integer-valued quadratic forms. Here $A[m]$ denotes an integer-valued indefinite quadratic form on a lattice $\Lambda$ in $\R^d$ of full rank. Meyer \cite{Meyer:1884} proved in 1884 that such a form represents zero non-trivially on $\Lambda$ if $d \ge 5$. Nowadays, this result is usually deduced from the Hasse-Minkowski theorem, which is a \textit{local-global principle} (see \cite{Gerstein:2008}, Theorem 5.7, Corollary 5.10).\par
Similarly to the result of Birch and Davenport \cite{birch-davenport:1958} on diagonal forms in five variables, our quantitative bounds in Theorem 
\ref{dio-ineq} depend essentially on explicit bounds for small zeros of integral forms (see Corollary \ref{bound:schlickewei:corollary}). First bounds of this kind were proved by Cassels \cite{Cassels:1955}, based on a geometric argument. Birch and Davenport improved Cassels' result as follows: If $d \geq 3$ and $A[m]$ admits a non-trivial zero on the lattice $\Lambda$, then there exists an isotropic lattice point $m \in \Lambda \setminus \{0\}$ with Euclidean norm
\begin{equation}
  \label{Meyer_Birch_Davenport}
  0<\norm{m}^2 \leq \gamma_{d-1}^{d-1} \4 (2 \Tr A^2 )^{(d-1)/2}\4 (\det{\Lambda})^2,
\end{equation}
where $\gamma_d$ denotes the Hermite constant in dimension $d$ (see \cite{davenport:1957,birch-davenport:1958a}). This bound is essentially best possible in view of an example by M.~Kneser, see \cite{Cassels:1956}, if $A$ has signature $(d-1,1)$. In 1985 Schlickewei \cite{Schlickewei:1985} extended Cassels' argument non-trivially by showing that the dimension, say $d_0$, of a maximal rational isotropic subspace has an essential impact on the size of small zeros, rather than mere indefiniteness (i.e. $d_0 \geq 1$). He established the following relation between small zeros of integral forms and the dimension $d_0$.

\begin{theorem}[Schlickewei \cite{Schlickewei:1985}]
  \label{th:Schlickewei}
  Let $\Lambda$ be a $d$-dimensional lattice and $A$ a non-trivial quadratic form in $d$ variables taking integral values on $\Lambda$. Also let $d_0 \geq 1$ be maximal such that there exists a $d_0$-dimensional sublattice of $\Lambda$ on which $A$ vanishes. Then there exist linearly independent lattice points $m_1,\ldots,m_{d_0} \in \Lambda$, spanning an isotropic subspace, of size
  \begin{equation}
    (\norm{m_1} \ldots \norm{m_{d_0}})^2 \ll_d ( \Tr A^2 )^{(d-d_0)/2} (\det{\Lambda})^2.
  \end{equation}
\end{theorem}

In the same way as Birch and Davenport \cite{birch-davenport:1958a} deduce their Theorem B from their Theorem A, we may conclude

\begin{theorem}[Schlickewei \cite{Schlickewei:1985}]
  \label{th:Schlickewei:alternative}
  Let $F,G \ne 0$ be quadratic forms in $d$ variables and suppose in addition that $G$ is positive definite. Let $d_0$ be maximal such that $F$ vanishes on a rational subspace of dimension $d_0$. Then there exist $d_0$ linearly independent lattice points $m_1,\ldots,m_{d_0} \in \Z^d$ such that $F$ vanishes on the corresponding subspace and
  \begin{equation*}
    G[m_1] \cdots G[m_{d_0}] \ll_d ( \Tr (FG^{-1})^2)^{(d-d_0)/2} \det{G},
  \end{equation*}
  where the implicit constant depends on $d$ only.
\end{theorem}

Using an induction argument combined with Meyer's theorem, Schlickewei derived also the following lower bound \eqref{bound:schlickewei} - which we only state for non-singular forms - for the dimension of a maximal rational isotropic subspace in terms of the signature $(r,s)$. For notational convenience, we may suppose that $r \geq s$. Then \textit{Hilfsatz} of Section 4 in \cite{Schlickewei:1985} reads

\begin{equation}
  \label{bound:schlickewei}
  d_0 \geq \begin{cases}
    s   & \text{if } r \geq s+3                \\
    s-1 & \text{if } r = s+2 \text{ or } r=s+1 \\
    s-2 & \text{if } r=s.
  \end{cases}
\end{equation}
 
\begin{remark}
  \label{bound:schlickewei:complement}
  One can complement Schlickewei's lower bound \eqref{bound:schlickewei} with the upper bound $d_0 \leq \min \{r,s\}$, which follows immediately by a dimension argument: If we decompose $\R^d = V_{+} \oplus V_{-}$ into subspaces $V_{+}$, $V_{-}$, on which $Q$ is positive or negative definite, and if $V_{\mathrm{iso}}$ denotes an isotropic subspace, then $V_\mathrm{iso} \cap V_{\pm} = \{0\}$ and thus
  \begin{equation*}
    \dim(V_\mathrm{iso}) =  \dim(V_\mathrm{iso} + V_{\pm}) - \dim(V_{\pm}) \leq d - \dim(V_{\pm}).
 \end{equation*}
 In particular, the lower bound \eqref{bound:schlickewei} is essentially optimal.
\end{remark}

Obviously, a straightforward combination of the upper bound \eqref{bound:schlickewei} together with Theorem \ref{th:Schlickewei} yields explicit bounds on the smallest non-trivial isotropic vector. However this application can be improved in the cases $r=s+2$ and $r=s$ by reducing the problem to dimension $d-1$ as done by Schlickewei in \textit{Folgerung 3} of \cite{Schlickewei:1985}, were he proved that for any integral quadratic 
form $A$ of signature $(r,s)$ there exists an isotropic lattice point $m \in \Z^d \setminus \{0\}$ such that $\norm{m}^2 \ll_d (\Tr A^2 )^\rho$\index{R@ $\rho=\rho(r,s)$, Schlickewei exponent}, where
\begin{equation*}
  \rho := \rho(r,s) := \begin{cases}
    \frac{1}{2} \frac{r}{s}     & \text{for} \ r \geq s+3                \\
    \frac{1}{2} \frac{s+2}{s-1} & \text{for} \ r=s+2 \ \text{or} \ r=s+1 \\
    \frac{1}{2} \frac{s+1}{s-2} & \text{for} \ r=s
  \end{cases}
\end{equation*}
as defined in \eqref{dio-ineq-rho} (see Section \ref{subsection:dio_ineq}). We shall extend this result to general lattices leading to the following strengthening of \eqref{Meyer_Birch_Davenport}.

\begin{corollary}
  \label{bound:schlickewei:corollary}
  Suppose that $A$ is a non-singular quadratic form of signature $(r,s)$ in $r+s=d \ge 5$ variables, which takes integral values on $\Lambda$. Additionally suppose that $\abs{\det(\Lambda)} \ge 1$, then the smallest non-trivial isotropic vector $m \in \Lambda$ of $A$ satisfies
  \begin{equation}
    \label{bound:schlickewei:2}
    0 < \norm{m}^2 \ll_d \max \{ (\Tr A^2 )^\frac{1}{2}, (\Tr A^2 )^\rho\}  \abs{\det{\Lambda}}^\frac{4\rho+2}{d}
  \end{equation}
  where $\rho$ is as defined in \eqref{dio-ineq-rho}.
\end{corollary}

Compared to \eqref{Meyer_Birch_Davenport}, the exponent in \eqref{bound:schlickewei:2}  is considerably smaller for a wide range of signatures $(r,s)$. Especially, if  $r \sim s$, then $\rho \sim 1/2$ and therefore $(2\rho+1)/d \sim 2/d$.

\begin{proof}
  As can be checked easily, in the cases $r \ge s+3$ and $r=s+1$ the bound \eqref{bound:schlickewei:2} follows immediately from Theorem \ref{th:Schlickewei} together with \eqref{bound:schlickewei}, since  $d/d_0 \le 2 \rho +1$ and $2 \leq d/d_0$ (by Remark \ref{bound:schlickewei:complement}) in both cases. (Here we estimate $(\Tr A^2)^{(d-d_0)/2}$ by $(\Tr A^2)^{1/2}$ if $\Tr A^2 < 1$ and by $(\Tr A^2)^{\rho}$ if $\Tr A^2 \geq 1$.) If $r=s$ or $r=s+2$, then the first relation does not hold. Here we fix a reduced basis $v_1,\ldots,v_d$ of $\Lambda$ with 
  \begin{equation*}
  \norm{v_1} \le \ldots \le \norm{v_d} \quad \text{and} \quad \abs{\det(\Lambda)} \asymp_d \norm{v_1} \ldots \norm{v_d}.
  \end{equation*}
  Let $\Lambda_0 := \Z v_1  + \ldots+ \Z v_{d-1}$, which is a $d{-}1$ dimensional sublattice of $\Lambda$, and note that Hadamard's inequality shows that $\det(\Lambda_0) = \norm{v_1 \wedge \ldots \wedge v_{d-1}} \le \norm{v_1} \ldots \norm{v_{d-1}}$. Thus
  \begin{equation*}
    \det(\Lambda_0) \ll_d  \det(\Lambda)^{(d-1)/d}.
  \end{equation*}
  Now denote by $A_0$ the restriction of $A$ to the subspace generated by $v_1,\ldots,v_{d-1}$. It follows that $A_0$ has signature either $(r,s-1)$ or $(r-1,s)$ and, since $(\Tr A^2)^{1/2} = \norm{A}_{\mathrm{HS}}$, also that $\Tr A_0^2 \le \Tr A^2$. Applying Theorem \ref{th:Schlickewei} (resp.\ Theorem \ref{th:Schlickewei:alternative} after a coordinate change) to $A_0$ and $\Lambda_0$ shows that there exists an isotropic lattice point $m \in \Lambda_0 \setminus \{0\}$ such that
  \begin{equation*}
    \norm{m}^2 \ll_d ( \Tr A_0^2 )^{\frac{d-1-d_0}{2d_0}} \abs{\det{\Lambda_0}}^{\frac{2}{d_0}} \ll_d  ( \Tr A^2 )^{\frac{d-1-d_0}{2d_0}} \abs{\det{\Lambda}}^{\frac{d-1}{d}\frac{2}{d_0}},
  \end{equation*}
  where $d_0$ denotes the dimension of a maximal isotropic subspace of $A_0$ (instead of $A$). Completing the proof, we note that in both cases $r=s+2$ and $r=s$ one has
  \begin{equation*}
    2 \leq (d-1)/d_0 \leq 2 \rho +1,
  \end{equation*}
  as can be readily seen.
\end{proof}

\begin{remark} In 1988 Schlickewei and Schmidt \cite{Schlickewei-Schmidt:1988}
complemented their work \cite{Schlickewei-Schmidt:1987} on isotropic subspaces of quadratic forms showing that Schlickewei's bound in terms of $d_0$ is best possible. Additionally, one can also ask if Schlickewei's bound \eqref{bound:schlickewei}  in terms of $(r,s)$ is best possible, as was 
already conjectured by Schlickewei himself in \cite{Schlickewei:1985}. At 
least for the cases $r \ge s + 3$ and $(3{,}2)$ this is known and due to Schmidt, see \cite{Schmidt:1985}.  
\end{remark}

\begin{remark}
 As a final remark we note that in the Geometry of Numbers it is often the case that one can use the existence of a lattice points satisfying some 
inequality in order to get several independent points satisfying a joint inequality. This argument was used by Schlickewei and Schmidt \cite{Schlickewei-Schmidt:1987,Schlickewei-Schmidt:1989} to prove an extension of Theorem \ref{th:Schlickewei}, in which they considered several isotropic subspaces and their relative position.
\end{remark}

\subsection{Proof of Theorem \ref{dio-ineq}}
Now we are in position to prove the second main theorem of this paper. To 
simplify the notation we may replace $Q$ by $Q/\eps$ and consider the solubility of the Diophantine inequality $\abs{Q[m]} < 1$. Notice that this rescaling does not change the constant $c_B =1$ occuring in Corollary \ref{cor:6.4}.

\begin{proof}[Proof of Theorem \ref{dio-ineq}]
  Let $d \ge 5$, $q_0 \ge 1$ and 
  \begin{equation}
  \label{th:dioeq:firstr}
  r \geq \tilde{b}_{\beta,d} q^{1/2} (q/q_0)^{(d+1)/(d-2)}
  \end{equation} 
  as in Corollary \ref{cor:6.4} and $\beta =2/d +\delta'/d$ with fixed $\delta' >0$ depending on $\delta >0$. Applying Corollary \ref{cor:6.4} with $b=-a=1/5$ (note that both conditions $\max\{\abs{a},\abs{b}\} \le r^2/5$ and $b-a \le 1$ are satisfied) gives the bound
  \begin{equation*}
      \Delta_r \le \frac{\volu H_r}{5}+ b_{\beta,d} \4 d_Q \4 r^{d-2} \4 q c_Q \Big( c_Q \4 T_{-}^{d(1/2-\beta)} +\gamma_{[T_{-},1],\beta}(r) + \gamma_{(1,T_{+}],\beta}(r) \log(T_{+}) \Big)
  \end{equation*}
  for any $T_{-} \in [\imbound,1]$ and
  \begin{equation*}
    T_{+} \gg_{\beta,d} \max \{1, \log( 10 \4 c_{\beta,d} q) )^2\} \gg_{\beta,d} \log(q+1)^2.
  \end{equation*}
  Hence, we can take $T_{+} \asymp_{\beta,d} \log(q \! + \! 1)^2$. Additionally, by taking
  \begin{equation*}
   T_{-} \asymp_{\beta,d} q^{-2/(d-4)-\delta/4} \abs{\det{Q}}^{-1/d}
  \end{equation*}
  we can also ensure that 
  \begin{equation*}
    b_{\beta,d}\4 d_Q \4 r^{d-2} q \abs{\det{Q}}^{1/2-\beta} T_{-}^{d(1/2-\beta)} \leq (\volu H_r)/10,
  \end{equation*}
  compare the lower bound \eqref{volume_bound} of Lemma \ref{l3}. At this 
step we have to choose
  \begin{equation}
    \label{th:dioeq:eq2}
    r \gg_{\beta,d} (q/q_0)^{1/2} q^{1/2+2/(d-4)+\delta/4} \ge q_0^{-1/2} 
\abs{\det{Q}}^{1/d} q^{2/(d-4)+\delta/4}
  \end{equation}
  in order to guarantee that $T_{-} \in [\imbound,1]$ is satisfied.\\[2mm]
  \noindent \textsl{First Case:} We consider first classes of quadratic forms $Q$ for which the lattice remainder is 'small': Corresponding to Diophantine properties of $Q$, we assume that
  \begin{equation}
    \label{proof:th1.4:eq1}
    \begin{aligned}
      b_{\beta,d} \4 q \4 \abs{\det{Q}}^{1/4-\beta/2} \gamma_{[T_{-},1],\beta}(r) & \le h_{\beta,d} \q \text{and} \\   b_{\beta,d} \4 q \4 \abs{\det{Q}}^{1/4-\beta/2} \gamma_{[1,T_{+}],\beta}(r) \log(T_{+}) & \le h_{\beta,d}
    \end{aligned}
  \end{equation}
  with some constant $h_{\beta,d}>0$ depending on $d$ and $\beta$ only (compare again with \eqref{volume_bound}) such that $5 \vol H_r \geq \volu H_r$.  Note that $r \geq q^{1/2}$ is fixed here. According to Corollary \ref{cor:6.4} and \eqref{th:dioeq:eq2} we shall take a priori
  \begin{equation}
    \label{th:dioeq:eq3}
    r \asymp_{\beta,d} (q/q_0)^{(d+1)/(d-2)} q^{1/2 +2/(d-4)+\delta}.
  \end{equation}
  Increasing the implict constant guarantees that $\vol H_r \ge 2$, i.e.\ there exists at least one non-zero lattice point $m \in \Z^d \setminus \{0\}$ satisfying both $\abs{Q[m]} \le 1$ and $\norm{Q_{+}^{1/2}m} \le r$. Because of $\rho \geq 1/2$, it is easy to see that the right-hand side of \eqref{th:dioeq:eq3} is bounded, up to absolute constants, by the right-hand side of \eqref{size_bound}. \\[2mm]
  \textsl{Second Case:} Now we assume that one of the inequalities in \eqref{proof:th1.4:eq1} fails. Then there exists a $t_0\in [T_{-}, T_{+}] $ such that the reciprocal $\alpha_d$-characteristic satisfies at least
  \begin{equation}
    \label{proof:th1.4:eq3}
    \beta_{t_0;r}^{-1} = d_Q \4 r^d \alpha_d(\Lambda_{t_0})^{-1}  \ll_{\beta,d} E(t_0) \defi ( q \log \log(q+\exp(1)) )^{\frac{2d}{d-4}+\delta/4}
  \end{equation}
  Following the proof of Lemma \ref{alpha_dio}, we see that there exists a $d$-dimensional sublattice $\Lambda' \subset \Lambda_{t_0}$ with $\alpha_d(\Lambda_{t_0}) = \abs{\det{\Lambda'}}^{-1} = \norm{w_1 \wedge \ldots \wedge w_n}^{-1}$, where
  \begin{equation*}
    w_j = \begin{pmatrix} r Q_+^{-1/2} (m_j - 4 t_0 Q n_j) \\ r^{-1} Q_{+}^{1/2} n_j \end{pmatrix}
  \end{equation*}
 is a basis of $\Lambda'$ determined by integral vectors $m_j,n_j \in \Z^d$, $j=1,\ldots,d$. We have also proven, writing $N = (n_1,\ldots,n_d), M = (m_1,\ldots,m_d) \in \mathrm{M}(d,\Z)$, that $N$ is invertible with $\beta_{t_0;r}^{-1} > \abs{\det{N}}$ and that the estimate
  \begin{equation*}
   \alpha_d(\Lambda_{t_0})^{-1} \gg_d r^{-(d-2)} q^{-1} \abs{\det{Q}}^{1/2} \abs{\det{N}} \norm{M N^{-1} -4t_0 Q} 
  \end{equation*}
  holds, provided that $\alpha_d(\Lambda_{t_0}) > q d_Q r^{d-2}$. In view of \eqref{proof:th1.4:eq3} the last condition is satisfied if we take a priori
  \begin{equation}
    \label{proof:th1.4:eq3:0}
    r \gg_{\beta,d} (E(t_0) q)^{1/2}.
  \end{equation}
  Now we are in position to apply Corollary \ref{bound:schlickewei:corollary} with the rescaled lattice $\Lambda = r \Lambda'$, noting that $\det(\Lambda) = r^d \det(\Lambda') \geq \abs{\det{Q}}^{1/2} \abs{\det{N}} \geq 1$, and the quadratic form $A[x] = \langle x , A x \rangle$ induced by the symmetric matrix
  \begin{equation*}
   A \defi \begin{pmatrix}
             0 & r^{-2} \mathbbm{1}_{d} \\
             r^{-2} \mathbbm{1}_{d} & 8 t_0 S
           \end{pmatrix}
  \end{equation*}
  with $\langle w_i, A w_j \rangle = \langle m_i, n_j \rangle + \langle 
m_j,n_i \rangle$. In other words, the quadratic form $A$ is represented by the symmetric matrix $A_0 \tdefi N^T M + M^T N$ in coordinates $w_1,\ldots,w_d$. In particular, $A$ is integer-valued on $\Lambda$. Since $A_1[n] := A_0[N^{-1}n]$, i.e.\ $A_1 = M N^{-1} + (M N^{-1})^T$, has the same signature as $A_0$, we need to check that the signature of $A_1$ is $(r,s)$. Because of
  \begin{equation*}
    \norm{A_1 - 8 t_0 Q} \ll_{\beta,d} \abs{\det{N}}^{-1} r^{-2} q E(t_0)
  \end{equation*}
  we may choose a priori $r \gg_{\beta,d} (q/q_0)^{1/2} \max\{1,t_0^{-1/2}\} q^{d/(d-4) +\delta}$, i.e.\ 
  \begin{equation}
    \label{proof:th1.4:eq4:before}
    r \gg_{\beta,d} (q/q_0)^{1/2} q^{1/2+(d+1)/(d-4)+\delta}
  \end{equation}
  to ensure that $A_1$ and $t_0 Q$ have the same number of eigenvalues with the same sign, i.e.\ the same signature (e.g.\ apply the Hoffman-Wielandt inequality, see Theorem 6.3.5 in \cite{horn-johnson:2013}). Thus, there exists a non-trivial lattice point $w = a_1 rw_1+\ldots + a_d rw_d \in \Lambda$, where $(a_1,\ldots,a_d) \in \Z^d \setminus \{0\}$, which satisfies $A[w]=0$ and, writing $n_0 = a_1 n_1 + \ldots + a_d n_d \in \Z^d \setminus \{0\}$, is of size
  \begin{equation}
    \label{proof:th1.4:eq4}
    \begin{aligned}
         \norm{Q_+^{1/2} n_0}^2 \leq \norm{w}^2 &\ll_d \max\{ (\Tr A^2)^\frac{1}{2}, (\Tr A^2)^\rho\} \abs{\det{\Lambda}}^\frac{4 \rho+2}{d} \\
         &\ll_{\beta,d} \log(q+1)^{4 \rho} (\abs{\det{Q}}^{1/2} E(t_0))^\frac{4 \rho+2}{d}\\
         &\ll_{\beta,d} q^{\delta+ \frac{8 \rho+4}{d-4}} \abs{\det{Q}}^\frac{2 \rho +1}{d}
    \end{aligned}
  \end{equation}
  where we used $\Tr A^2 \ll_d (r^{-2} + t_0)^2 \ll t_0^2 \ll_{\beta,d} \log(q+1)^4$ and \eqref{proof:th1.4:eq3}. Writing $w = (w_1,w_2) \in \mathbb{R}^d \times \mathbb{R}^d$ we also see that $0 = A[w] = r^{-2} \langle w_1, w_2 \rangle + 8 t_0 Q[n_0]$ and thus
  \begin{equation}
    \label{proof:th1.4:eq5}
    \begin{aligned}
      \abs{Q[n_0]} &\ll (r^2 t_0 )^{-1} \norm{w_1} {\cdot} \norm{w_2} \leq  (r^2 t_0 )^{-1}  \norm{w}^2 \\
      &\ll_d \max \{ 1, t_0^{2\rho-1}\} \abs{\det{\Lambda}}^\frac{4 \rho+2}{d} r^{-2} \ll_{\beta,d} q^{\delta+ \frac{8 \rho+4}{d-4}} \abs{\det{Q}}^\frac{2 \rho +1}{d} r^{-2}.
    \end{aligned}
  \end{equation}
  Hence, requiring in addition
  \begin{equation}
  \label{th:dioeq:lastr}
  r \gg_{\beta,d}  q^{\frac{1}{2}+ \frac{d \rho+2}{d-4}+\delta} \geq q^{\delta+ \frac{4 \rho+2}{d-4}} \abs{\det{Q}}^\frac{2 \rho +1}{2d},
  \end{equation}
  it follows from \eqref{proof:th1.4:eq5} that $\abs{Q[n_0]} \ll_{\beta,d} 1$, which in turn guarantees $\abs{Q[n_0]} < 1$ as long as $r$ is taken large enough in terms of $\beta$ and $d$. Combining this choice with the lower bounds on $r$ already required in \eqref{th:dioeq:firstr}, \eqref{th:dioeq:eq2}, \eqref{proof:th1.4:eq3:0} \eqref{proof:th1.4:eq4:before} and \eqref{th:dioeq:lastr}, we observe that an appropriate choice for $r$ is given by
 \begin{equation}
  r \asymp_{\beta,d} (q/q_0)^{\frac{d+1}{d-2}} q^{\frac{1}{2} + \frac{\max\{ \rho d +2,d+1\}}{d-4}  +\delta},
 \end{equation}
 where the implicit constant is chosen large enough depending on $\beta$ and $d$ only. This concludes the proof of Theorem \ref{dio-ineq}.
\end{proof}

\nocite{*}
\printindex
\printbibliography
\vspace{2mm}
\textsc{\scriptsize Mathematisches Institut, Bunsenstrasse 3-5, D-37073 G\"{o}ttingen, Germany} \\
\indent {\scriptsize \textit{Email address:} \texttt{buterus@mathematik.uni-goettingen.de}} \\[2mm]
\textsc{\scriptsize Faculty of Mathematics, Univ.\ Bielefeld, P.O.Box 100131, 33501 Bielefeld, Germany} \\
\indent {\scriptsize \textit{Email address:} \texttt{goetze@math.uni-bielefeld.de}} \\[2mm]
\indent \textsc{\scriptsize Mathematics Department, Northwestern University, 2033 Sheridan Road, Evanston, IL 60208, USA} \\
\indent {\scriptsize \textit{Email address:} \texttt{thomas.hille@northwestern.edu}}\\[2mm]
\indent \textsc{\scriptsize Dept.\ of Mathematics, Yale University, New Haven, CT, USA} \\
\indent {\scriptsize \textit{Email address:} \texttt{grigorii.margulis@yale.edu}}\\[2mm]

\end{document}